\documentclass[11pt]{amsart}
\usepackage[parfill]{parskip}    
\usepackage{fullpage}
\usepackage{graphicx}
\usepackage{amssymb,amsmath,latexsym}
\usepackage{tikz}
\usetikzlibrary{arrows,decorations.markings,decorations.pathreplacing,calc,matrix,intersections}
\tikzset{->-/.style={decoration={markings,mark=at position #1 with {\arrow{>}}},postaction={decorate}}}

\newcommand{\ranktworatio}{1.5}
\newenvironment{tikzrankzero}[2]{\begin{scope}[xshift=#1,yshift=#2] \def \aspectratiox {1} \def \aspectratioy {1} \path[name path=graph] (0,0) circle (0.001);\fill (0,0) circle (0.01);}{\end{scope}}
\newenvironment{tikzrankone}[2]{\begin{scope}[xshift=#1,yshift=#2] \def \aspectratiox {1} \def \aspectratioy {1} \draw[name path=graph] (0,0) circle (1);}{\end{scope}}
\newenvironment{tikzranktwo}[2]{\begin{scope}[xshift=#1,yshift=#2] \def \aspectratiox {\ranktworatio} \def \aspectratioy {1} \draw[name path=graph] (0,0) ellipse (1.5 and 1); \draw (0,-1) to (0,1); }{\end{scope}}
\newenvironment{tikzranktwov}[2]{\begin{scope}[xshift=#1,yshift=#2] \def \aspectratiox {1} \def \aspectratioy {\ranktworatio} \draw[name path=graph] (0,0) ellipse (1 and 1.5); \draw (-1,0) to (1,0); }{\end{scope}}
\newcommand{\tikzhair}[2]{{ \path[name path=line] let \p0 = (#2:1.5), \n1= {\aspectratiox * \aspectratiox * \x0 }, \n2= {\aspectratioy * \aspectratioy * \y0 } in  (0,0) to ( \n1,\n2 ); \path[name intersections={of=line and graph, by=#1base}]; \coordinate (#1) at ($(#1base) + (#2:0.5)$); \draw (#1base) to (#1); } }
\newcommand{\tikznohair}[2]{{ \path[name path=line] let \p0 = (#2:1.5), \n1= {\aspectratiox * \aspectratiox * \x0 }, \n2= {\aspectratioy * \aspectratioy * \y0 } in  (0,0) to ( \n1,\n2 ); \path[name intersections={of=line and graph, by=#1}]; } }

\usepackage{epstopdf}
\usepackage{ifpdf}
\usepackage{color}
\usepackage{float}

\usepackage{dsfont}
\definecolor{red}{rgb}{1,0,0} 

 \definecolor{darkgreen}{rgb}{0, .7, 0}

 \definecolor{purple}{rgb}{.7, 0, 1}

\newcommand{\lplus}{{\mathfrak{h}}}        
\newcommand{\F}{{\mathds{k}}}
\newcommand{\Z}{{\mathbb{Z}}}
\newcommand{\Q}{{\mathbb{Q}}}

\newcommand{\C}{{\mathbb{C}}}
\newcommand{\hairy}{\mathcal H}  
\newcommand{\ext}{\bigwedge\nolimits}
\newcommand{\h}{{\rm H}} 
\newcommand{\hone}{{\hat\h}} 
\newcommand{\X}{{\rm X}} 
\newcommand{\Xone}{{\hat\X}} 

\newcommand{\tV}[2]{{ {#1}^{\otimes {#2}}}} 
\newcommand{\tH}[1]{{ {\h}^{\otimes {#1}}}} 
\newcommand{\wV}[2]{{ {#1}^{\wedge {#2}}}} 
\newcommand{\wH}[1]{ {\h}^{{\wedge {#1}}}} 
\newcommand{\vcd}{{\textsc{vcd}}}

\newcommand{\bdry}{\partial}
\DeclareMathOperator{\Sym}{Sym} 
\DeclareMathOperator{\SG}{{\mathfrak S}}  
\newcommand{\GL}{{\mathrm{GL}}}  
\newcommand{\SL}{{\mathrm{SL}}}  
\newcommand{\SP}{{\mathrm{Sp}}}  
\newcommand{\Aut}{{\mathrm{Aut}}}  
\newcommand{\Out}{{\mathrm{Out}}}  
\newcommand{\Ind}{{\mathrm{Ind}}}
\newcommand{\Res}{{\mathrm{Res}}}
\newcommand{\sign}{{\mathrm{sign}}}
\newcommand{\alt}{{\mathrm{alt}}}
\newcommand{\x}{{\bf x}}
\newcommand{\y}{{\bf y}}

\newcommand{\SF}[1]{{\mathbb S}_{#1}}  
\newcommand{\Gr}{X}             
\newcommand{\Outn}{\Out(F_n)}
\newcommand{\Autn}{\Aut(F_n)}
\newcommand{\incl}{\hookrightarrow}
\newcommand{\iso}{\cong}

\newtheorem{proposition}{Proposition}[section]

\newtheorem{theorem}[proposition]{Theorem}
\newtheorem*{theorem*}{Theorem}
\newtheorem{lemma}[proposition]{Lemma}
\newtheorem*{claim}{Claim}
 \newtheorem*{notation}{Notation}
\newtheorem{corollary}[proposition]{Corollary}
\newtheorem{conjecture}[proposition]{Conjecture}
\newtheorem{question}[proposition]{Question}
\newtheorem{example}[proposition]{Example}

\theoremstyle{remark}
\newtheorem{remark}[proposition]{Remark}
\newtheoremstyle{red}{3pt}{3pt}{\color{red}}{}{\itshape}{.}{.5em}{}
\theoremstyle{red}

\title{Assembling homology classes in\\ automorphism groups of free groups}
\author{James Conant}
\author{Allen Hatcher}
\author{Martin Kassabov}
\author{Karen Vogtmann}

\begin{document}

\begin{abstract}
The observation that a graph of rank $n$ can be assembled from graphs of smaller rank $k$ with $s$ leaves by pairing the leaves together  leads to a process for assembling homology classes for $\Out(F_n)$ and $\Aut(F_n)$ from classes for groups $\Gamma_{k,s}$, where the $\Gamma_{k,s}$ generalize $\Out(F_k)=\Gamma_{k,0}$ and $\Aut(F_k)=\Gamma_{k,1}$.  The symmetric group $\SG_s$ acts on $H_*(\Gamma_{k,s})$ by permuting leaves, and for trivial rational coefficients we compute the $\SG_s$-module structure on $H_*(\Gamma_{k,s})$ completely for $k \leq 2$.  Assembling these classes then produces all the known  nontrivial rational homology classes for $\Aut(F_n)$ and $\Out(F_n)$ with the possible exception of classes for $n=7$ recently discovered by L. Bartholdi.  It also produces an enormous number of candidates for other nontrivial classes, some old and some new, but we limit the number of these which can be nontrivial using the representation theory of symmetric groups.  We  gain new insight into some of the most promising candidates by finding small subgroups of $\Aut(F_n)$ and $\Out(F_n)$ which support them
 and by finding geometric representations for the candidate classes as maps of 
 closed manifolds 
 into the moduli space of graphs.  Finally, our results have implications for the homology of the Lie algebra of symplectic derivations.
\end{abstract}

\maketitle

\section*{Introduction}
In this paper we develop a new approach to studying the homology of automorphism groups of free groups which gives fresh group theoretic and geometric insight into known families of homology classes, and also helps direct the search for new classes.  We restrict attention to homology and cohomology with untwisted coefficients in a field $\F$ of characteristic zero unless explicitly specified otherwise. 

Let us recall briefly what is known about these homology groups.  First of all, $H_i\big(\Autn\big)$ and $H_i\big(\Outn\big)$ are finite-dimensional over $\F$ for all $i$, and vanish for $i$ greater than the virtual cohomological dimension ($\vcd$), which is $2n-2$ for $\Autn$ and $2n-3$ for $\Outn$~\cite{CV}.  The groups $H_i\big(\Autn\big)$ and $H_i\big(\Outn\big)$ are independent of $n$ for $n\ge 5(i+1)/4$ as shown in~\cite{HV2,HV}, and these stable groups are in fact zero as well, as Galatius proved in~\cite{Gal}.  Thus in the first quadrant of the $(i,n)$ plane (see Figure~\ref{fig:introchart} below) there is a wedge-shaped region bounded by lines of slope $1/2$ and $5/4$ that contains all the nonzero groups $H_i\big(\Autn\big)$, and similarly for $H_i\big(\Outn\big)$.  There are only a small number of these groups which are explicitly known to be nonzero.  For $\Autn$ these occur for $(i,n)=(4,4)$, $(7,5)$, $(8,6)$,  $(8,7),$ $(11,7)$ and $(12,8)$;  for $\Outn$ the list is the same  except that $(7,5)$ is omitted.  (The natural map $H_i\big(\Autn\big)\to H_i\big(\Outn\big)$ is known to be surjective for all $i$ and $n$~\cite{Kawazumi} and we give a different proof of this in Theorem~\ref{thm:injection}.) These low-dimensional calculations are done mostly by computer; see~\cite{HV2,Morita,CVMorita,Ger,Gray,Bartholdi}.
Complete calculations of $H_i\big(\Autn\big)$ have been made only for $n\leq 5$ and for $H_i\big(\Outn\big)$ only for  $n\leq 7$.

\begin{figure}
\begin{center}
\begin{tikzpicture}[scale=.45] 
\newcommand\Mbox[2]{\draw (#1-.2,#2-.2) to (#1-.2,#2+.2) to (#1+.2,#2+.2) to (#1+.2,#2-.2) to (#1-.2,#2-.2)}
\newcommand\Mbbox[2]{\draw [fill=black] (#1-.2,#2-.2) to (#1-.2,#2+.2) to (#1+.2,#2+.2) to (#1+.2,#2-.2) to 
(#1-.2,#2-.2)}
\newcommand\ylabel[1]{\node (n#1) at (-.75,#1) {$\scriptstyle{#1}$}}
\newcommand\xlabel[1]{\node (i#1) at (#1,-.75) {$\scriptstyle{#1}$}}
\newcommand\xrule[1]{\draw [color=lightgray] (#1,0) -- (#1,15);}
\newcommand\yrule[1]{\draw [color=lightgray] (0,#1) -- (25,#1);}
 \draw (0,0)-- (25,0);
  \draw (0,0) -- (0,15);
 \fill [black!10] (0,1) to  (0,1.25) to (11,15) to (25,15) to (25,13.5) to (0,1);
\foreach \i in {1,...,24} {\xrule{\i};} 
\foreach \i in {1,...,14} {\yrule{\i};} 
\draw [dotted] (0,1.25) to (11,15);
\draw (0,1) to (25,13.5);
  \Mbbox {4}{4};
    \Mbbox {8}{6};
      \Mbbox {12}{8};
        \Mbox {16}{10};
          \Mbox {20}{12};
            \Mbox {24}{14};
 \draw  [fill=black]  (7,5) circle (.25);
  \begin{scope} [xshift=8cm, yshift=7cm]
  \fill [black!50] (0:.37) \foreach \x in {60,120,...,359} {
                -- (\x:.37)
            }-- cycle (90:.37); 
 \end{scope}
\begin{scope} [xshift=11cm, yshift=7cm]
\fill [black!50] (0:.4) \foreach \x in {60,120,...,359} {
                -- (\x:.4)
            }-- cycle (90:.4); 
 \end{scope}
   \fill [black] (11,7) circle (.2);
     \draw  (15,9) circle (.25);
       \draw (19,11) circle (.25);
         \draw (23,13) circle (.25);
\node (i) at (25.6,-.75) {$i$};
\node (n) at (-.75,15.5) {$n$};
\foreach \i in {1,...,24} {\xlabel{\i};}
\foreach \i in {1,...,14} {\ylabel{\i};} 
\fill [white] (12.3,3.2) -- (12.3,4.8) -- (23.8,4.8)--(23.8,3.2)--(20.6,3.2)--(20.6,2.3)--(15.5,2.3)--(15.5,3.2)--cycle;
\fill [white] (.6, 13.75) -- (7.5,13.75)--(7.5, 12)--(6,10.2)--(2,10.2)--(.6,12.)--cycle;
\node (vcd) at (18,4) {$H_i\big(\Aut(F_n)\big)=0$ above {\sc vcd}};
\node (vcdrange) at (18,3) {$i>2n-2$};
\node (vcd) at (4,13) {$H_i\big(\Aut(F_n)\big)=0$};
\node (vcd2) at (4,12) {in stable range};
\node (srange) at (4,11) {$n\geq \frac{5(i+1)}{4}$};
\end{tikzpicture}
\end{center}
\caption{Classes in the homology of $Aut(F_n)$ for $n\leq 14$.  The Morita classes  are shown as squares  and the Eisenstein classes are shown as circles, filled in if the classes are known to be nontrivial.  The nontrivial classes recently found by Bartholdi are shown as hexagons.  
}
\label{fig:introchart}
\end{figure}
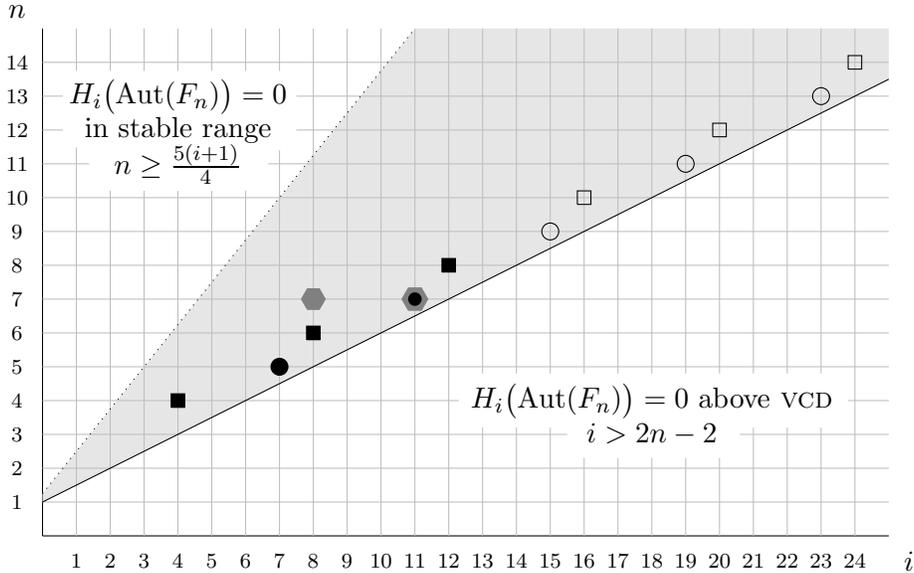

 There are two potentially infinite sequences which begin with nontrivial classes: these are classes $\mu_k$ for $(i,n)=(4k,2k+2)$ defined by Morita~\cite{Morita} and classes $\mathcal E_k$ for $(i,n)=(4k+3,2k+3)$ constructed in~\cite{CKV1}.  The latter are known as \emph{Eisenstein classes\/} because they arise from Eisenstein series via the connection between modular forms and the cohomology of $\SL_2(\Z)$ established by the Eichler-Shimura isomorphism.  The Morita classes $\mu_k$ are defined for both $\Aut$ and $\Out$, while the $\mathcal E_k$'s are defined for $\Aut$ and map to zero in $\Out$.  Note that these classes are all either one or two dimensions below the $\vcd$.

One of the big open questions is to determine which of the classes $\mu_k$ and $\mathcal E_k$ are nonzero.  However, even if they are nonzero it seems that they account for only a small fraction of the homology.  The Euler characteristic for $H_*\big(\Outn\big)$ was computed for $n\le 11$ by Morita, Sakasai, and Suzuki in~\cite{MSS1,MSS2}, and after starting with the values $1$ and $2$ for $n\le 8$, it becomes $-21,-124,-1202$ for $n=9,10,11$.  If this trend continues for larger $n$, it would say there are many odd-dimensional classes for $\Outn$,  though the only one discovered to date is the $11$-dimensional class in $\Out(F_7)$ recently found by Bartholdi \cite{Bartholdi}.  (This class is balanced by a single $8$-dimensional class, consistent with the Euler characteristic calculation for $n=7$.)
  

\begin{figure}[h]
\begin{center}
\begin{tikzpicture}[xscale=.55, yscale=.5] 
\newcommand\vbar[1]{\draw   (#1,-1) -- (#1,1);}
\draw (0,0) -- (25,0);
\foreach \i in {2,...,10} {\vbar{2*\i};} 
\vbar{22.5};
\node (n) at (2,.5) {$n$};
\foreach \j in {3,...,10} {\node () at (2*\j-1,.5) {$\j$};}
\node () at (21.25,.5) {$11$};
\node () at (23.5,.5){$12$};
\node (chi) at (2,-.6) {$\chi(\Out(F_n))$};
\node (chi3) at (5,-.6) {$1$};
\node (chi4) at (7,-.6) {$2$};
\node (chi4) at (9,-.6) {$1$};
\node (chi4) at (11,-.6) {$2$};
\node (chi4) at (13,-.6) {$1$};
\node (chi4) at (15,-.6) {$1$};
\node (chi4) at (17,-.6) {$-21$};
\node (chi4) at (19,-.6) {$-124$};
\node (chi4) at (21.1,-.6) {$-1202$};
\node (chi4) at (23.5,-.6) {$?$};
\end{tikzpicture}
\end{center}
\caption{Euler characteristic of $Out(F_n)$}\label{fig:eulertable}
\end{figure}

The Morita classes $\mu_k$ were first constructed using Lie algebra techniques underpinned by Kontsevich's ``formal noncommutative symplectic geometry"~\cite{Ko2,Ko1}. In~\cite{CVMorita} these classes were interpreted explicitly in Lie graph cohomology and generalized;  further generalizations including the classes $\mathcal E_k$ were obtained using ``hairy graph homology" in~\cite{CKV1}.  In the present paper we show how to construct  all of these classes in an elementary fashion which bypasses both graph homology and Kontsevich's work. The idea is to use the fact that $\Out(F_n)$ and $\Autn$ are the first two groups in a series $\Gamma_{n,0},  \Gamma_{n,1},  \Gamma_{n,2}, \cdots$ where $\Gamma_{n,s}$ is the group of homotopy classes of self-homotopy equivalences of a rank $n$ graph fixing $s$ leaves (vertices of valence one)~\cite{BF, Hat95, HV}.
These groups are related by natural surjective homomorphisms $\Gamma_{n,s}\to \Gamma_{n,s-k}$ with kernel $(F_n)^k$. These homomorphisms split for $k<s$ but not for $k=s$.

The groups $\Gamma_{n,s}$ are of interest because by gluing graphs together along a subset of their leaf vertices we obtain many homomorphisms $\Gamma_{n_1,s_1}\times\cdots\times \Gamma_{n_k,s_k}\to \Gamma_{n,s}$.  On the level of homology, each such map induces a homomorphism
$$
H_*(\Gamma_{n_1,s_1})\otimes\cdots\otimes H_*(\Gamma_{n_k,s_k})\longrightarrow H_*(\Gamma_{n,s}),
$$
which we call an {\em assembly map}
(see Section~\ref{sec:gluing}).
For example by pairing up all of the leaves of two rank one graphs with $s$ leaves (in any way) we obtain  an assembly map
$$
H_*(\Gamma_{1,s})\otimes H_*(\Gamma_{1,s})\longrightarrow H_*(\Gamma_{s+1,0})=H_*\big(\Out(F_{s+1})\big).
$$
Restricting to the case that $s$ is odd, say $s=2k+1$, it is easy to calculate that  $H_{2k}(\Gamma_{1,2k+1})\iso \mathbb \F$ (see Section~\ref{subsec:rankone}),
and in Section~\ref{subsec:Morita} we note that the Morita class $\mu_{k}$ is the image of $\alpha_k\otimes \alpha_k$ under this assembly map, where  $\alpha_{k}$ is a generator of $H_{2k}(\Gamma_{1,2k+1})$.
This graphical interpretation of the original Morita series allows us to give two new proofs that Morita classes vanish after one stabilization, one proof being algebraic (Section~\ref{subsec:stab_n}) and the other geometric (Section~\ref{sec:geometric}).  This result was first proved via a more elaborate combinatorial computation in graph homology in~\cite{Stable}.

As a consequence of the elementary construction, we find that all the classes $\mu_k$ in Morita's original series, as well as the generalized Morita classes given in~\cite{CVMorita}, are supported on abelian subgroups of $\Autn$. This naturally gives rise to the question of whether the standard maximal abelian subgroup can support nontrivial homology classes, and we show in Section~\ref{subsec:maximal} that it cannot.
For the Eisenstein classes we find slightly more complicated nonabelian subgroups that support them.

Parallel to these group-theoretic descriptions of Morita and Eisenstein classes there are geometric representations as maps of closed orientable manifolds into the classifying spaces for $\Aut(F_n)$ or $\Out(F_n)$ carrying top-dimensional homology classes of the manifolds to the Morita or Eisenstein classes. In the case of Morita classes the manifolds are tori while for the Eisenstein classes they are products of a certain $3$-manifold with tori.

The computational heart of the paper is in Section~\ref{sec:cohomology} where we use the natural action of the symmetric group $\SG_s$ on $\Gamma_{n,s}$ to study $H_*(\Gamma_{n,s})$. For $n=1$ and $n=2$ we determine the $\SG_s$-module structure of $H_*(\Gamma_{n,s})$ completely.  This can be applied in the search for nontrivial classes in $ H_*(\Gamma_{n,s})$ that lie in the images of assembly maps.  In particular we show in Section~\ref{subsec:genMorita} that many of the generalized Morita classes are in fact zero, though this does not apply to the $\mu_k$'s themselves. In Section~\ref{subsec:odd} we show that certain odd-dimensional classes constructed in~\cite{MSS1} must vanish, but we also find some  new  candidates for nontrivial odd-dimensional classes. 

The calculation of  $H_*(\Gamma_{1,s})$  is an easy consequence of the fact that  $\Gamma_{1,s}\iso \Z_2 \ltimes \Z^{s-1}$.  To calculate the homology of $\Gamma_{2,s}$ we use the short exact sequence
$$
1 \longrightarrow F_2^{\,s} \longrightarrow \Gamma_{2,s} \longrightarrow \Gamma_{2,0}=\Out(F_2)\longrightarrow 1.
$$
In the Leray-Serre spectral sequence associated to this short exact sequence we note that all differentials are zero,
allowing us to completely calculate the homology.  (Actually, for convenience we use cohomology rather than homology for spectral sequence arguments and indeed for most algebraic calculations.)
The results of our computations for $n=1$ and $n=2$ and small values of $s$ are recorded in several tables at the end of the paper.

These computations show that even though the dimension of $H_i(\Gamma_{n,s})$ as a vector space over $\F$ increases with $s$ for fixed $n=1,2$, it is nevertheless true that as representations of $\SG_s$ these vector spaces eventually stabilize in the sense of~\cite{CF}.  This representation stability holds for all $n$ in fact, as a corollary of a result of Jim\'enez Rolland~\cite{JR}; see Proposition~\ref{prop:repstab}.

We also apply some elementary representation theory to show that the group $H_i(\Gamma_{n,s})$ is nontrivial whenever $i$ is an even multiple of $n$ and $s$ is sufficiently large with respect to $i$ and $n$.  This can be contrasted with the situation for stabilization with respect to $n$, where $H_i(\Gamma_{n,s})$ becomes trivial as $n$ increases, by Galatius' theorem for $s=1$ and hence for all $s$ since the groups $H_i(\Gamma_{n,s})$ are independent of both $n$ and $s$ when $n\ge 2i+4$ by~\cite{HV}.

In Section~\ref{sec:HairyLie} we point out the relationship of the homology of $\Gamma_{n,s}$ with both hairy graph homology~\cite{CKV1, CKV2} and the cohomology of the  Lie algebra of symplectic derivations, as studied by Kontsevich, Morita and many others.  In particular, we show how our computations for $\Gamma_{n,s}$ imply that the cohomology in each dimension of this Lie algebra contains infinitely many simple $\SP$-modules.

 Section~\ref{sec:conj} contains some conjectures and open questions. 
 Most nontrivial rational homology classes for any $\Gamma_{n,s}$ which occur below the $\vcd$~$2n-3+s$ have been shown to be in the image of assembly maps.  The only  exceptions are new classes for $\Out(F_7)$ and $\Aut(F_7)$  recently found by Bartholdi \cite{Bartholdi}, for which this is still unclear.   
 It is then natural to ask whether assembly maps generate all classes below the $\vcd$. The number of potential homology classes for $H_i(\Gamma_{n,s})$ constructed from assembly maps grows exponentially with $n$, leading to the expectation that the rank  of the homology also grows very fast.
For  $s=0$ this expectation coincides nicely with
the Euler characteristic calculations of Morita, Sakasai, and Suzuki cited earlier.
 
 Finally, we remark that the rational classifying spaces for the groups $\Gamma_{n,s}$  have natural compactifications, whose homology has recently been studied by Chan,  Galatius and  Payne (\cite{Chan,CGP}).  One thing they show is that this homology vanishes in dimensions less than $s-3$.  It is easy to see that all homology classes which are in the image of  an assembly map must vanish in this compactification, consistent with their calculations.

{\bf Acknowledgements:}
Martin Kassabov was partially supported by  Simons Foundation grant 305181 and NSF grants DMS 0900932 and 1303117.
Karen Vogtmann was partially supported by NSF grant DMS 1011857 and a Royal Society Wolfson Award.

\section{The groups $\Gamma_{n,s}$}
\label{sec:groups}

\subsection{Definitions}
\label{subsect:Gns}

The group $\Outn$ is the group of homotopy classes of self-homotopy equivalences of a finite connected graph $X$ of rank $n$, and $\Autn$ is the basepointed version of this, the homotopy classes of homotopy equivalences of $X$ fixing a basepoint, where homotopies are also required to fix the basepoint.  A natural generalization is to choose $s$ distinct marked points $x_1,\dots,x_s$ in $X$ and then define $\Gamma_{n,s}$ to be the group of homotopy classes of self-homotopy equivalences of $X$ fixing each $x_i$, with homotopies also required to fix these points.  The group operation in $\Gamma_{n,s}$ is induced by composition of homotopy equivalences, which is obviously associative with an identity element.  To check that inverses exist one uses the following elementary fact:

\begin{lemma}
\label{lem:homeq}
If $f : X\to Y$ is a homotopy equivalence of finite connected graphs taking a set of $s$ marked points $\x=\{x_1,\dots,x_s\}$ bijectively to another such set $\y=\{y_1,\dots,y_s\}$, then $f$ is a homotopy equivalence of pairs $(X,\x)\to(Y,\y)$, so there is a map $g : Y\to X$ restricting to $f^{-1}$ on $\y$ with the compositions $gf$ and $fg$  homotopic to the identity fixing $\x$ and $\y$ respectively.
\end{lemma}

\begin{proof}
Let $Z$ be the quotient of the mapping cylinder of $f$ obtained by collapsing $\x\times I$ to $\x=\y$. The quotient map collapses a finite number of intervals to a point so it is a homotopy equivalence.  If $f$ is a homotopy equivalence, then the inclusions of $X$ and $Y$ into the mapping cylinder are homotopy equivalences, hence the same is true for the inclusions into $Z$.  It follows that $Z$ deformation retracts onto the copies of $X$ and $Y$ at either end.  The deformation retraction to $X$ gives the map $g$.
\end{proof}

This lemma also shows that $\Gamma_{n,s}$ does not depend on the choice of $(X,\x)$, up to isomorphism.
Throughout most of the paper we will take $X$ to be a rank $n$ graph with exactly $s$ leaves, with the leaf vertices as the marked points.  Here a \emph{leaf\/} means a vertex of valence one together with the adjoining edge. Our generic notation for a graph of rank $n$ with $s$ leaves will be $X_{n,s}$.
Two examples of rank $3$ graphs with $4$ leaves are shown in  Figure~\ref{fig:thornedrose}.
\begin{figure}[h]
\begin{center}
\begin{tikzpicture} [thick, scale=.25]
\draw  (0,0) to [out=130, in=180]  (90:4.1);
\draw  (0,0) to [out=50, in=0]  (90:4.1);
\draw   (0,0) to [out=50, in=100]  (10:4.1);
\draw   (0,0) to [out=-30, in=-80] (10:4.1);
\draw   (0,0) to [out=130, in=80]  (170:4.1);
\draw   (0,0) to [out=210, in=-100] (170:4.1);
\draw (0,0) to [out=210, in = 90] (-2,-3);
\fill[black] (-2,-3) circle (.15);
\draw (0,0) to [out=250, in = 90] (-.5,-3);
\fill[black] (-.5,-3) circle (.15);
\draw (0,0) to [out=290, in = 90] (.5,-3);
\fill[black] (.5,-3) circle (.15);
\draw (0,0) to [out=330, in = 90] (2,-3);
\fill[black] (2,-3) circle (.15);
\begin{scope} [xshift=15cm, yshift = 1cm, scale=1.1]
\draw (0,0) circle (2);
\draw (-4,0) to (4,0);
\fill[black] (-4,0) circle (.15);
\fill[black] (4,0) circle (.15);
\draw (0,0) to (0,2);
\draw (-60:2) to (-60:4);
\fill[black] (-60:4) circle (.15);
\draw (-120:2) to (-120:4);
\fill[black] (-120:4) circle (.15);
\end{scope}
\end{tikzpicture}
\end{center}
\caption{Two possibilities for $\Gr_{3,4}$
}
\label{fig:thornedrose}
\end{figure}
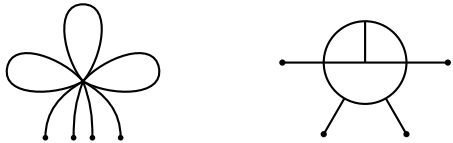

A homotopy equivalence $f\colon X_{n,s}\to X_{n,s}$ that permutes the leaf vertices induces an automorphism of $\Gamma_{n,s}$ via conjugation by $f$.  If $f$ fixes the leaf vertices this is an inner automorphism, hence induces the identity on the homology of $\Gamma_{n,s}$, so there is an induced action of the symmetric group $\SG_s$ on the homology. If we choose $X_{n,s}$ to have a single nonleaf vertex as in the left half of Figure~\ref{fig:thornedrose} then this action of $\SG_s$ on homology comes from the action on $X_{n,s}$ permuting the leaves. The $\SG_s$-action on $H_*(\Gamma_{n,s})$ will play a major role in later sections of the paper.

The groups $\Gamma_{0,s}$ are trivial since any $\Gr_{0,s}$ is a tree and any homotopy equivalence of a tree which fixes all of its leaf vertices is homotopic to the identity by a homotopy fixing the leaf vertices.

As shown in~\cite{BF}, the group   $\Gamma_{1,s}$ is the semidirect product  $\Z_2 \ltimes \Z^{s-1}$.
The free abelian subgroup $\Z^{s-1}$ is generated by homotopy equivalences which wrap one leaf edge around the (unique) loop while fixing the leaf vertex and the rest of the graph. These generators commute since they have disjoint supports. Note that wrapping all of the leaf edges around the loop in the same direction results in a map which is homotopic to the identity fixing the leaf vertices, so there are only $s-1$ independent generators.   The generator of $\Z_2$ flips the loop, so acts on $\Z^{s-1}$ by $x\mapsto -x$.

We remark that $\Gamma_{n,s}$ could also be defined as the mapping class group of the $3$-manifold $M_{n,s}$ formed by removing $s$ disjoint balls from the connected sum of $n$ copies of $S^1\times S^2$, modulo the subgroup generated by Dehn twists along embedded 2-spheres.  This follows from results of Laudenbach and is made explicit in  Proposition~1 of~\cite{HV}.

The groups  $\Gamma_{n,s}$ for $s>1$  were first considered in~\cite{Hat95} in work on homological stability
and also appeared in  Bestvina and Feighn's proof that $\Out(F_n)$ is a virtual duality group~\cite{BF}.  It was observed in~\cite{CKV1} that their homology is very closely related to  hairy graph homology groups for the Lie operad.

\subsection{Short exact sequences}
\label{subsec:sequences}

In this section we observe that there are natural short exact sequences relating the groups $\Gamma_{n,s}$.

\begin{proposition}
\label{thm:extension}\
If $n > 1$ and $k\leq s$ there is a short exact sequence
$$
1\longrightarrow  F_n^{k}\longrightarrow \Gamma_{n,s}\longrightarrow \Gamma_{n,s-k}\longrightarrow 1
$$
which splits if $k<s$. This holds also when $n=1$ and $k < s$, but in the exceptional case $(n,k)=(1,s)$ there is a split short exact sequence
$$
1\longrightarrow \Z^{s-1}\longrightarrow \Gamma_{1,s} \longrightarrow \Gamma_{1,0}\longrightarrow 1
$$
expressing $\Gamma_{1,s}$ as the semidirect product  $\Z_2 \ltimes \Z^{s-1}$.
\end{proposition}

For $k=s-1$  the proposition follows from~\cite{BF}, section 2.5, where it is shown that $\Gamma_{n,s}\iso \Autn \ltimes F_n^{s-1}$.

\begin{proof}
Let $X$ be a rank $n$ graph containing a set $\x=\{x_1,\dots,x_s\}$ of $s$ distinct marked points.  Let $E_{n,s}$ be the space of homotopy equivalences $X\to X$ fixed on $\x$, so $\Gamma_{n,s}=\pi_0(E_{n,s})$.  For $k\leq s$ there is an inclusion $E_{n,s} \subset E_{n,s-k}$ obtained by no longer requiring homotopy equivalences to fix $x_1,\dots,x_k$. Evaluating homotopy equivalences $X\to X$ on $x_1,\dots,x_k$ gives a map $E_{n,s-k}\to X^k$ which is a fibration with fiber $E_{n,s}$ over the point $(x_1,\dots,x_k)$.  The long exact sequence of homotopy groups for this fibration ends with the terms
$$
\pi_1(E_{n,s-k})\longrightarrow \pi_1(X^k) \longrightarrow \Gamma_{n,s} \longrightarrow \Gamma_{n,s-k} \longrightarrow 1.
$$
When $k<s$ the first term $\pi_1(E_{n,s-k})$ is trivial by obstruction theory. Namely, we can assume $X$ is obtained by attaching $1$-cells to a set of $s-k$ $0$-cells, and then any loop of homotopy equivalences $f_t : X \to X$ fixing the $0$-cells can be deformed to the trivial loop since $\pi_2(X)=0$.  Thus we obtain the first short exact sequence in statement of the proposition when $k<s$, for arbitrary $n$.

To split this short exact sequence when $k<s$ it suffices to find a map $E_{n,s-k}\to E_{n,s}$ such that the composition $E_{n,s-k}\to E_{n,s} \to E_{n,s-k}$ is homotopic to the identity.   We are free to choose the marked points $x_1,\dots,x_k$ anywhere in the complement of the remaining $s-k$ points, so we choose them in a small contractible neighborhood $N$ of the point $x_{k+1}$. We can then deformation retract $E_{n,s-k}$ onto the subspace $E'_{n,s-k}$ of homotopy equivalences that are fixed on $N$.  (This is particularly easy if we choose $X$ to have a valence one vertex with $x_{k+1}$ as this vertex.) The subspace $E'_{n,s-k}$ includes naturally into $E_{n,s}$, and this inclusion gives the desired map $E_{n,s-k}\to E_{n,s}$ as the composition of the first two maps $E_{n,s-k}\to E'_{n,s-k} \hookrightarrow E_{n,s}\hookrightarrow E_{n,s-k}$, the first map being the retraction produced by the deformation retraction. The composition of the three maps is homotopic to the identity by the deformation retraction itself.

There remain the cases $k=s$.  The issue is whether $\pi_1(E_{n,0})$ is trivial, so that the long exact sequence becomes a short exact sequence. To settle this, consider the fibration $E_{n,1}\to E_{n,0}\to X$ which gives a long exact sequence
$$
1\longrightarrow\pi_1(E_{n,0})\longrightarrow \pi_1(X) \longrightarrow \Gamma_{n,1} \longrightarrow \Gamma_{n,0} \longrightarrow 1
$$
where the initial $1$ is $\pi_1(E_{n,1})$. The middle map in this sequence is the map from $\pi_1$ of the base of the fibration to $\pi_0$ of the fiber, and it is easy to check the definitions to see that this is the map $F_n\to \Autn$ sending an element of $F_n$ to the inner automorphism it determines. The kernel of this map is the center of $F_n$ so it is trivial when $n>1$ and we deduce that $\pi_1(E_{n,0})=1$ in these cases, so we again have the short exact sequence claimed in the proposition.

When $n=1$ and $k=s$ the space $E_{1,0}$ is homotopy equivalent to $S^1$ and the exact sequence of the fibration $E_{1,s}\to E_{1,0} \to X^s$ becomes
$$
1 \longrightarrow \Z \longrightarrow \Z^{s} \longrightarrow \Gamma_{1,s} \longrightarrow\Gamma_{1,0} \longrightarrow 1,
$$
with the map $\Z\to\Z^s$  the diagonal inclusion.  This yields the short exact sequence displaying $\Gamma_{1,s}$ as the semidirect product $\Z_2 \ltimes \Z^{s-1}$.
\end{proof}

\begin{remark}
\label{rem:Birman}
 If we use Laudenbach's theorem to express $\Gamma_{n,s}$ in terms of the mapping class group of $M_{n,s}$ then the short exact sequence of Proposition~\ref{thm:extension} can be derived from a 3-dimensional analog of 
the Birman exact sequence for mapping class groups of surfaces [2].  From this viewpoint the space $E_{n,s}$ is replaced by the diffeomorphism group of $M_{n,s}$, and the resulting fibration is a very simple special case of much more general fibrations due to Cerf, Palais, and Lima.
\end{remark}

\subsection{Homology splitting}

If $s=k$ and $n\ge 2$ the map $\Gamma_{n,s}\to\Gamma_{n,s-k}=\Gamma_{n,0}$ does not split. The reason is that $\Gamma_{n,0}=\Outn$ contains finite subgroups which do not lift.  For example, consider the symmetry group of the graph consisting of two vertices joined by $n+1$ edges.
This is a subgroup of $\Gamma_{n,0}$ which cannot be realized on any graph of rank $n$ by graph symmetries which fix a basepoint, so the subgroup does not lift to any $\Gamma_{n,s}$ with $s\geq 1$.

Homology with coefficients in $\F$ does not see finite subgroups, and in fact when we pass to homology we do obtain a splitting.  Note that it suffices to prove this for $s=1$, i.e. the map from $\Aut(F_n)$ to $\Out(F_n)$. Homology splitting of this map was first proved by Kawazumi in~\cite{Kawazumi}. There is a simple proof of this using the fact that the moduli space of graphs (respectively basepointed graphs) is a rational $K(\pi,1)$ for $\Outn$ (respectively $\Autn$).  The idea is that although there is no natural way to choose a basepoint in a graph one can compensate by taking a suitably weighted sum of all possible basepoints.

\begin{theorem}
\label{thm:injection}
The natural map $\Autn\to \Out(F_n)$ splits on the level of rational homology, so $H_k\big(\Out(F_n)\big)$ embeds into $H_k\big(\Autn\big)$.
\end{theorem}

\begin{proof}
We define a backwards map on the chain level. We take  $C_*\big(\Autn\big)$ and $C_*\big(\Out(F_n)\big)$ to be defined in terms of the spine of the moduli space of (basepointed) graphs.  We refer to~\cite[section 2]{Stable} for complete details.   The chain complex $C_*\big(\Autn\big)$ is spanned by graphs with specified subforests and a chosen basepoint, while $C_*\big(\Out(F_n)\big)$ is defined in the same way except the graphs do not have basepoints. In both cases, the edges in the subforests are ordered, and changing the order incurs the sign of the permutation.
There are two boundary operators $\partial_C$ and $\partial_R$ which sum over contracting and removing forest edges respectively. In both the basepointed and unbasepointed cases, contracting the $i$th edge of a forest comes with the sign $(-1)^{i+1}$, while removing that edge comes with the sign $(-1)^i$.

The natural projection $\pi\colon\Autn\to\Out(F_n)$ corresponds to the map
$$
\pi_*\colon C_*\big(\Autn\big)\to C_*\big(\Out(F_n)\big)
$$
which forgets the basepoint.

\begin{figure}
\begin{center}
\begin{tikzpicture}[scale=.55]
\draw (-1,-1) to (0,0) to (1,-1);
\draw[ultra thick] (0,0) to (0,2);
\draw   (0,2) to (0,3);
\draw (-1,3) to (0,2) to (1,3);
\node (v) at (0,2)[left] {$v$\,\,};
\node (w) at (0,0)[left] {$w$\,\,};
\draw[->] (5,1)  to  node[above, pos=0.45] {$\partial_C$} (7,1);
\begin{scope}[xshift = 12.5cm, yshift = 1cm]
\draw (-1,-1) to (0,0) to (1,-1);
\draw (-1,1) to (0,0) to (1,1);
\draw (0,0) to (0,1);
\node (vw) at (0,0)[left] {$vw$\,\,\,};
\end{scope}
\draw[->] (0,-2) to node[left, pos=0.45]{$r$} (0,-4);
\begin{scope}[xshift = -2cm, yshift = -8cm]
\draw (-1,-1) to (0,0) to (1,-1);
\draw (0,0) to (0,3);
\draw (-1,3) to (0,2) to (1,3);
\draw[ultra thick] (0,0) to (0,2);
\fill[black] (0,2) circle (.125);
\node (vlab) at (0,1)[left] {$(|v|-2)$\,\,};
\end{scope}
 \begin{scope}[xshift = 3cm, yshift = -8cm]
 \draw (-1,-1) to (0,0) to (1,-1);
\draw (0,0) to (0,3);
\draw (-1,3) to (0,2) to (1,3);
\fill[black] (0,0) circle (.125);
\draw[ultra thick] (0,0) to (0,2);
\node (wlab) at (0,1)[left] {${}+(|w|-2)\,\,$};
\end{scope}
\begin{scope}[yshift = -8cm]
\draw[->] (5,1)  to  node[above, pos=0.45] {$\partial_C$} (7,1);
\end{scope}
\begin{scope}[xshift = 12.5cm, yshift = -7cm]
\draw (-1,-1) to (0,0) to (1,-1);
\draw (-1,1) to (0,0) to (1,1);
\draw (0,0) to (0,1);
\fill[black] (0,0) circle (.125);
\node (exp) at (0,0)[left] {$(|v|+|w|-4)$\,\,\,};
\end{scope}
\begin{scope}[xshift = 12 cm]
\draw[->] (0,-2) to node[left, pos=0.45]{$r$} (0,-4);
\end{scope}
\end{tikzpicture}
\caption{Diagram commutes because $|vw|=|v|+|w|-2$}
\label{fig:partialC}
\end{center}
\end{figure}
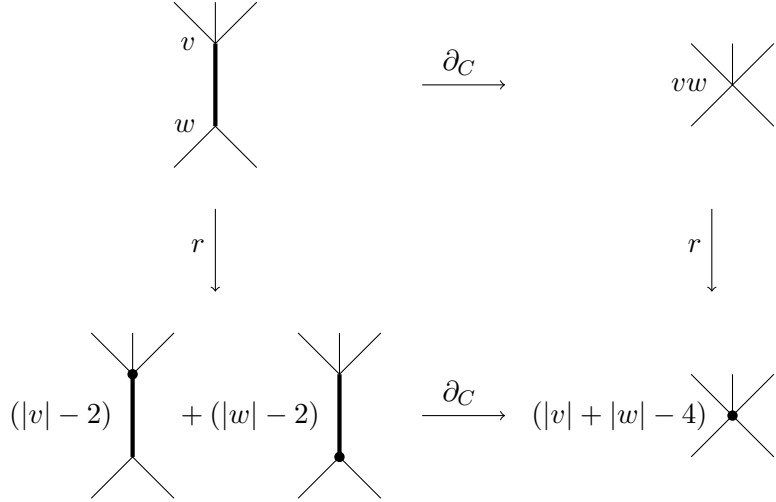
We now define a map
$
r\colon C_*\big(\Out(F_n)\big)\to C_*\big(\Autn\big).
$ by
$$
r(G)=\sum_{v\in V(G)} (|v|-2)r_v(G).
$$
Here $V(G)$ is the vertex set of $G$, $|v|$ is the valence of $v$ and $r_v(G)$ is the forested graph $G$ with $v$  specified as the basepoint.
We need to check that $r$ is a chain map. Clearly $\partial_Rr=r\partial_R$ since the definition of $r$ makes no reference to the forest, and the signs in $\partial_R$ make no reference to the basepoint.  For $\partial_C$ we must check whether the order of performing the two operations of adding a basepoint and contracting an edge matters, the signs in $\partial_C$ being the same in both cases. If $e$ is an edge with vertices $v$ and $w$, then adding a basepoint distinct from $v$ and $w$ clearly commutes with contracting $e$.  Adding basepoints at $v$ and at $w$ followed by contracting $e$ results in the same basepointed graph with multiplicity $|v|+|w|-4$, whereas contracting $e$ first results in a vertex $vw$ of valence $|v|+|w|-2$, so adding a basepoint there also gives multiplicity $|v|+|w|-4$ (see Figure~\ref{fig:partialC}).

Now observe that $\pi_*\circ r(G)=k_GG$, where $k_G=\sum_{v\in V(G)}(|v|-2)=2n-2$.
Thus if we are not in the trivial case $n=1$ the composition $\pi_*\circ r$ is represented by a diagonal matrix with nonzero diagonal entries and is therefore invertible. So $r_*$ is injective on homology.
\end{proof}

\section{Cohomology of $\Gamma_{n,s}$}
\label{sec:cohomology}

We are interested in studying the homology of $\Outn$ and $\Autn$, with trivial coefficients in a field $\F$ of characteristic $0$.  The idea is to glue together homology classes of the $\Gamma_{n,s}$ using the assembly maps
described briefly in the Introduction and defined more precisely in Section~\ref{sec:gluing}.
 To find nontrivial classes which can be fed to the assembly maps
 we use some elementary representation theory of symmetric groups and $\GL_n(\Z)$ together with the Leray-Serre spectral sequence applied to the group extensions
\begin{equation}
\label{eq:ses}
1 \longrightarrow F_n^{\,s} \longrightarrow \Gamma_{n,s} \longrightarrow\Gamma_{n,0}=\Outn \longrightarrow 1
\end{equation}
from Section~\ref{subsec:sequences}.

For the calculations it will be convenient to switch from homology to cohomology, which is isomorphic by the universal coefficient theorem since we are taking coefficients in $\F$ and all homology is finite-dimensional over $\F$. In the course of our study we will exploit the structure of $H_i(\Gamma_{n,s})$ as an $\SG_s$-module.  Since all the modules we consider are finite-dimensional and all $\SG_s$-modules are self-dual, the cohomology is isomorphic to the homology also as an $\SG_s$-module, though the isomorphism is not canonical.

\subsection{A little representation theory}

 In this section we establish some notation and collect some results from representation theory which we will use.  All the contents of this section are well-known and can be found, for example, in~\cite{FH}.

Recall that the irreducible representations of  $\SG_s$ correspond to partitions of $s$ and are often represented by drawing Young diagrams with $s$ boxes arranged in rows of non-increasing size.  We use $P_\lambda$ to denote the representation corresponding to the partition $\lambda=(\lambda_1,\ldots,\lambda_k)$, where $\lambda_1\geq\lambda_2\geq\dots\geq\lambda_k$ and $\sum_i\lambda_i=s$. Exponential notation denotes equal values of $\lambda_i$, e.g., $P_{(2,1,1,1,1)}$ is written as $P_{(2,1^4)}$.

\begin{example}
The module $P_{(s)}$ is the $1$-dimensional trivial module.
The module $P_{(1^s)}=\alt$ is the $1$-dimensional {\em alternating} representation of $\SG_s$, where a permutation $\sigma$ acts as multiplication by $\sign(\sigma)=\pm 1$.
The module $P_{(s-1,1)}$ is the $(s-1)$-dimensional {\em standard}  representation $\F^s/\F$ of $\SG_s$.
It contains distinguished elements $v_i$, $1\leq i \leq s$, which satisfy $\sum v_i=0$.
\end{example}

The tensor product of two $\SG_s$-representations is also an $\SG_s$-representation with the diagonal action.
In general the multiplicity of an irreducible representation $P_\nu$ in the decomposition of $P_\lambda\otimes P_\mu$ is difficult to compute, but for $\nu=(s)$ it is known that $P_{(s)}$  occurs with multiplicity $1$ if $\lambda=\mu$ and with multiplicity $0$ otherwise.
One tensor product we will encounter is $P_{\lambda} \otimes \alt$. This is equal to $P_{\lambda ^\prime}$ where $\lambda ^\prime$ denotes the {\em transpose partition}, obtained by switching the rows and columns of the Young diagram.

If $P$ is a representation of $\SG_{s-k}$ and $Q$ is a representation of $\SG_{k}$, then $P\otimes Q$ is a representation of
$\SG_{s-k}\times\SG_{k}$.  If we consider $\SG_{s-k}\times\SG_{k}$ as a subgroup of $\SG_s$ we can form the induced representation.
Following Fulton and Harris~\cite{FH}, we denote this induced representation by $P\circ Q$, i.e.,
$$
P\circ Q = \Ind_{\SG_{s-k}\times\SG_{k}}^{\SG_s} P\otimes Q.
$$
The Littlewood-Richardson rule can be used to compute the decomposition of $P_\lambda \circ P_\mu$ into irreducible modules. When $\mu=(k)$ this specializes to the
{\it Pieri rule}, see~\cite[appendix A]{FH}. This says that the terms of $P_\lambda\circ P_{(k)}$ correspond to all Young diagrams which can be obtained by adding $k$  boxes to the diagram for $\lambda$, each in a different column. An example is illustrated in Figure~\ref{fig:pieri}.
\begin{figure}[b]
\begin{center}
 \begin{tikzpicture}[scale=.5] 
\draw (0,0) -- (2,0);\draw (0,1) -- (2,1);\draw (0,2) -- (2,2);
\draw (0,0) -- (0,2);\draw (1,0) -- (1,2);\draw (2,0) -- (2,2);
\node (circ) at (3,1) {$\circ$};

\begin{scope}[xshift=4cm]
\draw (0,.5) -- (3,.5);\draw (0,1.5) -- (3,1.5);\draw (0,.5) -- (0,1.5);
\draw (1,.5) -- (1,1.5);\draw (2,.5) -- (2,1.5);\draw (3,.5) -- (3,1.5);
\fill [red] (.5,1) circle (.075);\fill [red] (1.5,1) circle (.075);\fill [red] (2.5,1) circle (.075);
 \end{scope}
\node (to) at (8,1) {$=$};

 \begin{scope}[xshift=9cm]
 \draw (0,0) -- (2,0);\draw (0,1) -- (5,1);\draw (0,2) -- (5,2);
\draw (0,0) -- (0,2);\draw (1,0) -- (1,2);\draw (2,0) -- (2,2);
\draw (3,1) -- (3,2);\draw (4,1) -- (4,2);\draw (5,1) -- (5,2);
\fill [red] (2.5,1.5) circle (.075);\fill [red] (3.5,1.5) circle (.075);\fill [red] (4.5,1.5) circle (.075);
 \end{scope}
\node (plus1) at (15,1) {$+$};
 
\begin{scope}[xshift=16cm]
 \draw (0,0) -- (2,0);\draw (0,1) -- (4,1);\draw (0,2) -- (4,2);\draw (0,-1) -- (1,-1);
\draw (0,-1) -- (0,2);\draw (1,-1) -- (1,2);\draw (2,0) -- (2,2);
\draw (3,1) -- (3,2);\draw (4,1) -- (4,2); 
\fill [red] (2.5,1.5) circle (.075);\fill [red] (3.5,1.5) circle (.075);\fill [red] (.5,-.5) circle (.075);
 \end{scope}
\node (plus2) at (21,1) {$+$};
 \begin{scope}[xshift=22cm]
 \draw (0,0) -- (2,0);\draw (0,1) -- (3,1);\draw (0,2) -- (3,2);\draw (0,-1) -- (2,-1);
\draw (0,-1) -- (0,2);\draw (1,-1) -- (1,2);\draw (2,-1) -- (2,2);
\draw (3,1) -- (3,2);  
\fill [red] (2.5,1.5) circle (.075);\fill [red] (1.5,-.5) circle (.075);\fill [red] (.5,-.5) circle (.075);
 \end{scope}
\end{tikzpicture}\caption{Pieri rule for decomposing $P_{(2,2)}\circ P_{(3)}$}
\label{fig:pieri}
\end{center}
\end{figure}
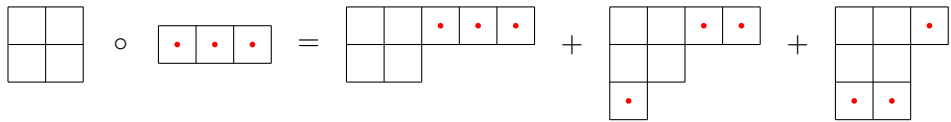

Now let $V$ be an $n$-dimensional vector space.  The irreducible representations of $\GL(V)=\GL_n(\F)$ also correspond to partitions, and we let  $\SF{\lambda}V$ denote the  $\GL_n$-representation associated to the partition $\lambda$.  Since $\dim(V)=n$ only partitions into at most $n$ pieces occur.   {\em Schur-Weyl duality} gives the irreducible decomposition of the representation
$\tV{V}{q}$ as a module over $\GL(V) \times \SG_q$, namely
$$
\tV{V}{q}
\iso \bigoplus_\lambda \, \SF{\lambda}V\otimes P_\lambda,
$$
where the sum is over all partitions of $q$ into at most $n$ pieces (if $\lambda$ has more than $n$ rows the module $\SF{\lambda}V$ is zero) (see, e.g.,~\cite{FH}, Cor.~6.6). We emphasize that $\GL(V)$ acts trivially on $P_{\lambda}$ and $\SG_q$ acts trivially on $\SF{\lambda}V$.

\begin{example} For $q=2$  the Schur-Weyl formula gives
$$
V\otimes V = \big(\SF{(2)}V \otimes P_{(2)}\big) \oplus \big(\SF{(1^2)}V \otimes P_{(1^2)}\big) =
\Sym^2 V \oplus {\textstyle \ext^2} V.
$$
where $\Sym^k$  denotes the $k$-th symmetric power functor on vector spaces
and $\ext^k$ is the $k$-th exterior power.
\end{example}

\begin{notation} We denote by $\wV{V}{q}$  the $\SG_q$-module which is isomorphic as a vector space to  $\tV{V}{q}$, with  $\SG_q$ acting by permuting the factors and multiplying by the sign of the permutation, i.e.,
$$
V^{\wedge q}= V^{\otimes q} \otimes \alt.
$$
\end{notation}

The Schur-Weyl formula translates to a similar formula for $\wV{V}{q}$:
$$
\wV{V}{q} = \tV{V}{q}\otimes \alt \iso \bigoplus_\lambda \, \SF{\lambda}V\otimes P_{\lambda}\otimes \alt = \bigoplus_\lambda \, \SF{\lambda}V\otimes P_{\lambda^\prime},
$$
where the sum is over all partitions of $q$ into at most $n$ pieces.

Finally, we record a computation we will use later.
\begin{lemma}
\label{lem:decomp}
Suppose $\dim(V)=2$. Then
$\SF{(q-k,k)}V\iso \Sym^{q-2k}V\otimes \det^k$ as $\GL(V)$-modules, where $\det^k=(\ext^2V)^{\otimes k}$ is the $1$-dimensional representation given by the $k$-th power of the determinant.
\end{lemma}
\begin{proof}
This can be seen by calculating the Schur polynomials $S_\lambda$ for the two sides, which determine the representations uniquely. Using the formula A.4 of~\cite[Appendix A]{FH}  one obtains
$$
S_{(a,b)}=(x_1x_2)^b\Big[\frac{x_1^{a-b+1}-x_2^{a-b+1}}{x_1-x_2}\Big]=(S_{(1,1)})^b S_{(a-b)}.
$$
The lemma  now follows because Schur polynomials of tensor products multiply,
$\SF{(1,1)}H=\det$ and $\SF{(c)}V=\Sym^{c}V$.
\end{proof}

\subsection{The Leray-Serre spectral sequence}
Shifting from homology to cohomology now, the Leray-Serre spectral sequence of a group extension $1\to N\to G\to Q\to 1$ is
a first-quadrant spectral sequence with $E_2^{p,q}=H^p\big(Q;H^q(N)\big)$, which converges to $H^{p+q}(G)$.  Applied to the short exact sequence~\eqref{eq:ses} it reads
\begin{equation}
\label{eq:LS}
E_2^{p,q} = H^p\big(\Out(F_n); H^q(F_n^{\,s})\big) \implies H^{p+q}(\Gamma_{n,s}).
\end{equation}

The symmetric group $\SG_s$ which permutes the factors of $F_n^{\,s}$ induces an action on each of the $E_2$ terms which commutes with all differentials.
We begin by identifying the coefficients $H^q(F_n^{\,s})$  as $\SG_s$-modules.

Throughout this section we set $\h=H^1(F_n)\iso \F^n$.  Note that the action of $\Outn$ on $\h $ factors through the natural $\GL_n(\Z)$ action on $\h$.

\begin{lemma}
\label{lem:Fn}
The cohomology of $F_n^{\,s}$ is given as an $\SG_s$-module by the formula
$$
H^q(F_n^{\,s}) = \wV{\h}{q} \circ P_{(s-q)}.
$$
\end{lemma}

\begin{proof}
The K\"unneth formula gives  an isomorphism
$$
H^*(F_n)\otimes \cdots\otimes H^*(F_n)\iso H^*(F_n\times \cdots \times F_n)
$$
via the cohomology cross product.  The group $\SG_s$ acts by permuting the factors, with signs determined by the permutation and the dimension of the cohomology  groups on the left-hand side (see, e.g.,~\cite[Chapter 3B]{Hatcher}).  The cohomology of $F_n$  is $\F$ in dimension $0$, $\h$ in dimension $1$ and zero in higher dimensions, so in dimension $q$ the cohomology   of $F_n^{\,s}$ is the direct sum of $\binom{s}{q}$
copies of $\h^{\otimes q}$.   These copies are permuted by the action of $\SG_s$.  The stabilizer of each copy is isomorphic to $\SG_q\times \SG_{s-q}$, where the action of $\SG_q$ on $\h^{\otimes q}$ is modified by the sign of the permutation since all classes are in dimension $1$.

In other words, $H^q(F_n^{\,s})$ is obtained by inducing  up to $\SG_s$ the $\SG_q\times \SG_{s-q}$-module $\wV{\h}{q} \otimes P_{(s-q)}$.
\end{proof}

We  now read off information which we obtain immediately from the spectral sequence~\eqref{eq:LS}. The first observation applies to the case $s=1$. The same result was obtained earlier by Kawazumi~\cite{Kawazumi} using a different method.

\begin{proposition}
\label{prop:Aut-Out}
There is an isomorphism
$H^k\!\big(\Autn;\!\F\big)\iso H^k\!\big(\Out(F_n);\!\F\big)\oplus H^{k-1}\!\big(\Out(F_n);\h\big)$.
\end{proposition}
\begin{proof}
In the spectral sequence associated to $1\to F_n\to \Autn\to\Out(F_n)\to 1$ we have $E_2^{p,q}\iso H^p\big(\Out(F_n);H^q(F_n)\big)$ with differentials of bidegree $(2,-1)$. Since $F_n$ has cohomological dimension one there are only two nontrivial rows, namely $q=0$ and $q=1$, so  the only possible nonzero differentials in the entire spectral sequence are on the $E_2$ page; they  start in the top row $q=1$ with target in the bottom row $q=0$.

The map on cohomology induced by $p\colon \Autn\to \Outn$ factors through the edge
homomorphism $e\colon E_\infty^{p,0} \to  H^p\big(\Autn\big)$
$$
\begin{tikzpicture} 
\node (On) at (0,1.2) {$H^p\big(\Outn\big)$};
\node (An) at (7,1.2) {$H^p\big(\Autn\big)$};
\node (Q) at (3.5,0) {$H^p\big(\Outn\big)/\mathrm{Im}(d^2)$};
 \draw[->] (On) to node[above, pos=0.45] {$p^*$} (An);
\draw [->] (On) to (Q);
\draw [->] (Q) to node[right, pos=0.35] {\,\,\,\,$e$} (An);
\end{tikzpicture}
$$
The top arrow is injective by Theorem~\ref{thm:injection}, so  the left arrow is as well, i.e., $\mathrm{Im}(d^2)=0$ and the differentials on the $E_2$ page are also trivial.  Thus $$
\begin{aligned}
H^k\big(\Autn\big) &= E_2^{k,0}\oplus E_2^{k-1,1}\\ &= H^k\big(\Outn\big)\oplus H^{k-1}\big(\Outn;H^{1}(F_n)\big) \\&= H^k\big(\Outn\big)\oplus H^{k-1}\big(\Out(F_n);\h\big).
\end{aligned}
$$
\end{proof}

The next observation has to do with the top-dimensional cohomology of $\Gamma_{n,s}$.

\begin{proposition}
\label{prop:vcd-nonvanish}
$H^k(\Gamma_{n,s})$ vanishes for $k>2n-3+s$ and   $H^{2n-3+s}(\Gamma_{n,s})$ is given as an $\SG_s$-module by
\begin{align*}
H^{2n-3+s}(\Gamma_{n,s})&\iso H^{2n-3}\left(\Out(F_n);\wV{\h}{s}\right).
\end{align*}
\end{proposition}
\begin{proof}
The cohomology group $H^p\big(\Out(F_n);H^q(F_n^{\,s})\big)$ vanishes if either  $p>2n-3$ or $q>s$ since the virtual cohomological dimension of $\Out(F_n)$ is $2n-3$ and the virtual cohomological dimension of $F_n^{\,s}$ is equal to $s$.  Thus the only possible nonzero terms in the spectral sequence~\eqref{eq:LS} lie in a rectangle with $E_2^{2n-3,s}$ at its upper right-hand corner, so all differentials into or out of  $E_2^{2n-3,s}$ are zero and
$$
H^{2n-3+s}(\Gamma_{n,s})\iso H^{2n-3}\big(\Out(F_n);H^s(F_n^{\,s})\big)\iso H^{2n-3}\left(\Out(F_n);\wV{\h}{s}\right).
$$
\end{proof}

\subsection{Rank zero}
Since $\Gamma_{0,s}$ is trivial, we just have
$$
H^i(\Gamma_{0,s}) = \begin{cases}
P_{(s)} = \F & \mbox{if } i=0 \\
0 & \mbox{if } i\not=0. \\
\end{cases}
$$

\subsection{Rank one}
\label{subsec:rankone}
For $n=1$ the short exact sequence~\eqref{eq:ses} is a restatement of the fact that
$$
\Gamma_{1,s}=\Z_2 \ltimes  \Z^{s-1},
$$
where the $\Z_2$ acts via $x\mapsto-x$. We can use this to compute the cohomology of $\Gamma_{1,s}$ as an $\SG_s$-module without  appealing to the Leray-Serre spectral sequence, as follows.

\begin{proposition}
\label{prop:rankone}
As a representation of $\SG_s$
$$
H^i(\Gamma_{1,s})\iso
\begin{cases} P_{(s-i,1^i)} &\mbox{if } i \mbox{ is even}   \\
0 & \mbox{if } i\mbox{ is odd}. \end{cases}
$$
In particular, $H^{2k}(\Gamma_{1,2k+1})\iso  P_{(1^{2k+1})} = \F$ with the alternating action.
\end{proposition}
\begin{proof}
The rational  cohomology of $\Gamma_{1,s}=\Z_2 \ltimes  \Z^{s-1}$ is the  invariants of the $\Z_2$-action on the cohomology of $\Z^{s-1}$ induced from the action on $\Z^{s-1}$.  The cohomology of $\Z^{s-1}$  is the exterior algebra on $s-1$ generators.
Thus we get the even degree part of this exterior algebra:
$$
H^i(\Gamma_{1,s})\iso
\begin{cases}
\bigwedge^i \F^{s-1} &\mbox{if } i \mbox{ is even}   \\
0 & \mbox{if } i\mbox{ is odd}.
\end{cases}
$$
To see the $\SG_s$ action,    write $\Gamma_{1,s}=\Z_2 \ltimes (\Z^s/\Z)$.
The representation $\F^s/\F$ is the  standard representation  $P_{(s-1,1)}$ of $\SG_s$, and  by~\cite[Ex. 4.6]{FH} we have $\ext^i(\F^s/\F)=  P_{(s-i, 1^{i})}$ as an $\SG_{s}$-module.
\end{proof}
We record this calculation for small values of $s$ in table form at the end of the paper (Section~\ref{sec:tables}).
We note that the results agree with the calculations  via dihedral homology in~\cite{CKV2}.

\subsection{Rank two}
Recall that $\Gamma_{2,0}=\Out(F_2)\iso \GL_2(\Z)$, so that for $n=2$ the
$E_2$ term of~\eqref{eq:LS} is
$$
E_2^{p,q}=H^p\big(\GL_2(\Z);H^q(F_2^{\,s})\big).
$$
Since   $ \GL_2(\Z)$ has virtual cohomological dimension 1, the only potentially nonzero terms on the $E_2$-page of this spectral sequence lie in the first two columns $p=0$ and $p=1$. For $p=1$ the cohomology of $\GL_2(\Z)$ is closely related to modular forms; we review this relation in the next subsection.

\subsubsection{Modular forms}\label{sec:modular}
\label{subsec:modular}
Let $\mathcal M_s$ be the vector space of  classical modular forms for $\SL_2(\Z)$ of weight $s$, and let $\mathcal S_s\subset \mathcal M_s$ be the subspace of cusp forms. See~\cite{Lang} for an elementary introduction to these spaces.  They satisfy
$$
\bigoplus_{s\geq 0}\, \mathcal M_s\iso  \F[E_4,E_6],
$$
where
$E_4$ and $E_6$ are generators of weight
$4$ and $6$ respectively. In particular, $\mathcal M_s$ is nonzero only for $s>2$ even.  In these cases the subspace $\mathcal S_s$ has codimension $1$.
The classical Eichler-Shimura isomorphism  (see, e.g.,~\cite{Hab}) relates modular forms to the cohomology of $\SL_2(\Z)$:
$$
H^1\big(\SL_2(\Z); \Sym^s(\F^2)\big)\iso\mathcal M_{s+2} \oplus \mathcal S_{s+2}.
$$

We next review the relation between cusp forms and the stabilizer of the cusp at infinity.  Let $P\leq \SL_2(\Z)$ be the (parabolic) subgroup generated by the matrix
$\left(\!\begin{smallmatrix} 1 & 1 \\ 0 & 1 \end{smallmatrix}\!\right)$ and
consider the map
$$
\rho\colon H^1\big(\SL_2(\Z); \Sym^{s}(\F^2)\big) \longrightarrow H^1 \big(P; \Sym^{s}(\F^2)\big)
$$
induced by inclusion. Since $P\iso \Z$, its first cohomology with any coefficients is simply the coinvariants of the action, which is isomorphic to the space of invariants.  If $x$ and $y$ are a basis for $\F^2$, the generator of  $P$ acts on $\Sym^s(\F^2)=\F[x^s,x^{s-1}y,\ldots,xy^{s-1},y^s]$ by sending $x\mapsto x$ and $y\mapsto x+y$, so the space of invariants is  $1$-dimensional, spanned by $x^s$.  The map $\rho$ can be identified with the map
$\mathcal{M}_{s+2} \oplus \mathcal{S}_{s+2} \to \mathcal{M}_{s+2}/\mathcal{S}_{s+2}$ (projection on the first factor, zero on the second factor) 
given by the normalized value of the modular form at infinity (see~\cite{Hab}).

We can reinterpret $\rho$ in terms of the cohomology of $\GL_2(\Z)$ using the short exact sequence $1\to \SL_2(\Z)\to \GL_2(\Z)\to \Z_2\to 1$ (see, e.g.,~\cite{CKV1}); this  gives
$$
H^1\big(\GL_2(\Z); \Sym^s(\F^2)\otimes \det\big)\iso \mathcal{M}_{s+2} \hbox{\ \,\, and \,\, }
H^1\big(\GL_2(\Z); \Sym^s(\F^2)\big)\iso \mathcal{S}_{s+2}.
$$
Since $H^1\big(\SL_2(\Z); \Sym^{s}(\F^2)\big)=\mathcal M_{s+2} \oplus \mathcal S_{s+2}$ we see that  the restriction of $\rho$ to the second factor
$$
\rho\colon H^1\big(\GL_2(\Z); \Sym^{s}(\F^2)\big) \longrightarrow H^1\big(P; \Sym^{s}(\F^2)\big)
$$
is zero, but on the first factor
$$
\rho\colon H^1\big(\GL_2(\Z); \Sym^{s}(\F^2)\otimes \det\big) \longrightarrow H^1 \big(P; \Sym^{s}(\F^2)\otimes \det \big)
$$
has $1$-dimensional image when $s>0$ is even.

\subsubsection{Cohomology calculations}
\begin{lemma}
\label{thm:GL}
Let $\h=H^1(F_2)\iso \F^2$. Then as an $\SG_q$-module,
$$
H^0\big(\GL_2(\Z);\wV{\h}{q}\big)=
\begin{cases}  P_{(2^{2m})} &\mbox{if } q=4m   \\
0 & \mbox{otherwise}. \end{cases}
$$

$$
H^1\big(\GL_2(\Z);\wV{\h}{q}\big)=
\begin{cases}  0 &\mbox{if $q$ is odd} \\
&\\
W_q:=\displaystyle\bigoplus_{0\leq i<\frac{q}{2}}\mathcal X_{q,i}\otimes P_{(2^i,1^{q-2i})} &\mbox{if $q$ is even}
\end{cases}
$$
where  $\mathcal X_{q,i}= \mathcal S_{q+2-2i}$ if $i$ is even and $\mathcal M_{q+2-2i}$ if $i$ is odd.  In either case $\mathcal X_{q,i}$ is   trivial as an $\SG_q$-module.
\end{lemma}

\begin{remark}
\label{rem:WComputations}
The formula in the statement above gives the following pattern for the first few $W_{q}:$
\begin{align*}
W_0&=0\\
W_2&=(\mathcal S_4\otimes P_{(1^2)})  \\
W_4&=(\mathcal S_6\otimes P_{(1^4)}) \oplus (\mathcal M_4\otimes P_{(2,1^2)}) \\
W_6&=(\mathcal S_8\otimes P_{(1^6)}) \oplus (\mathcal M_6\otimes P_{(2,1^4)}) \oplus (\mathcal S_4\otimes P_{(2^2,1^2)}) \\
W_8&=(\mathcal S_{10}\otimes P_{(1^8)}) \oplus (\mathcal M_8\otimes P_{(2,1^6)}) \oplus
            (\mathcal S_6\otimes P_{(2^2,1^4)}) \oplus (\mathcal M_4\otimes P_{(2^3,1^2)}) \\
W_{10}&=(\mathcal S_{12}\otimes P_{(1^{10})}) \oplus (\mathcal M_{10}\otimes P_{(2,1^8)}) \oplus
               (\mathcal S_{8}\otimes P_{(2^2,1^6)}) \oplus (\mathcal M_{6}\otimes P_{(2^3,1^4)})\oplus (\mathcal S_4\otimes P_{(2^4,1^2)}).
\end{align*}
However, the dimension of $\mathcal M_k$ is  $1$ for $k=4,6,8,10$, and hence the dimension of $\mathcal S_k$ is trivial in those degrees. The module  $\mathcal S_{12}$ is $1$-dimensional, so we see interesting modular forms entering the picture starting with $W_{10}$. Using this information, the above list simplifies to
\begin{align*}
W_0&=0  \\
W_2&=0  \\
W_4&=    P_{(2,1^2)} \\
W_6&=   P_{(2,1^4)} \\
W_8&=   P_{(2,1^6)} \oplus  P_{(2^3,1^2)}\\
W_{10}&=P_{(1^{10})} \oplus P_{(2,1^8)} \oplus   P_{(2^3,1^4)}.
\end{align*}
\end{remark}

\begin{proof} We first decompose the coefficients $\wV{\h}{q}$ into irreducible components using Schur-Weyl duality.
Since $\h$ has dimension $2$,  this gives
$\wV{\h}{q}\iso \bigoplus_\lambda \SF{\lambda}\h\otimes P_{\lambda^\prime}$, where the sum is over all partitions of $q$ into at most $2$ pieces, i.e., $\lambda=(q-k,k)$.

Now $H^0\big(\GL_2(\Z);\wV{\h}{q}\big)$ is equal to the $\GL_2(\Z)$-invariants of $\wV{\h}{q}$, so we are looking for the trivial representations $\SF{\lambda}\h$ appearing in the Schur-Weyl formula.  By Lemma~\ref{lem:decomp} we have
$\SF{(q-k,k)}\h\iso \Sym^{q-2k}\h\otimes \det^k$,
which is clearly trivial only  if $q=2k$ and $k$ is even. Therefore
as an $\SG_q$-module we have
$$
H^0\big(\GL_2(\Z);\wV{\h}{q}\big) =
\begin{cases}
  P_{(2m,2m)^\prime}=P_{(2^{2m})} &\hbox{if  } q=4m\\
  0 &\hbox{otherwise.}
\end{cases}
$$

For the first cohomology we have
\begin{align*}
 H^1\big(\GL_2(\Z);\wV{\h}{q}\big) &=\bigoplus_{0\leq k\leq \frac{q}{2}} H^1\big(\GL_2(\Z);\SF{(q-k,k)}\h \otimes P_{(q-k,k)^\prime}\big)\\
 &=\bigoplus_{0\leq k\leq \frac{q}{2}} H^1\big(\GL_2(\Z);\SF{(q-k,k)}\h\big)\otimes P_{(q-k,k)^\prime}\\
 &=\bigoplus_{0\leq k\leq \frac{q}{2}} H^1\big(\GL_2(\Z);\Sym^{q-2k}\h \otimes {\det}^{k} \big)\otimes P_{(q-k,k)^\prime}\\
 &=\bigoplus_{0\leq k\leq \frac{q}{2}} H^1\big(\GL_2(\Z);\Sym^{q-2k}\h \otimes {\det}^{k} \big)\otimes P_{(2^k,1^{q-2k})}.
\end{align*}
The computations in Section~\ref{sec:modular} now give
$$
H^1\big(\GL_2(\Z);\Sym^{r}\h\otimes{\det}^{\ell}\big)=
\begin{cases}
0&\text{if $r$ is odd,}\\
\mathcal S_{r+2} &\text{if $r$ and $\ell$ are even,}\\
\mathcal M_{r+2} &\text{if $r$ is even and $\ell$ is odd.}\\
\end{cases}
$$

Substituting $r=q-2k$, $\ell=k$ into these formulas and  using the fact that $\mathcal M_2=0$  completes the calculation.
\end{proof}

We  now have the tools we need to completely compute the cohomology of $\Gamma_{2,s}$ as an $\SG_s$-module. {}

\begin{theorem}
\label{thm:ranktwo}
The cohomology of $\Gamma_{2,s}$   is
$$
H^{i}(\Gamma_{2,s}) = \begin{cases}
P_{(2^{2m})}\circ P_{(s-4m)} & i=4m\leq s\\
0 & i=4m+2\\
W_{2m}\circ P_{(s-2m)} & i= 2m+1\leq s+1\\
0& \hbox{otherwise,}
\end{cases}
$$
where $W_{2m}$ is the module defined in the statement of Lemma~\ref{thm:GL}.
\end{theorem}

\begin{proof}
 For all $r\geq 2$ the differential on the $r$-th page of the Leray-Serre spectral sequence has bidegree $(r, -r + 1)$.  Since only the first two columns are nonzero all differentials are too wide to be nonzero, so $E_2=E_\infty$ and
$$
H^k(\Gamma_{2,s})\iso H^0\big(\GL_2(\Z);H^k(F_2^{\,s})\big)\oplus H^1\big(\GL_2(\Z); H^{k-1}(F_2^{\,s})\big).
$$

By Lemmas~\ref{lem:Fn} and~\ref{thm:GL} we have
\begin{align*}
H^0\left(\GL_2(\Z);H^k(F_2^{\,s})\right) &\iso H^0\left(\GL_2(\Z);\wV{\h}{k}\circ P_{(s-k)}\right)\\
&\iso H^0\left(\GL_2(\Z);\wV{\h}{k}\right)\circ P_{(s-k)}\\
&\iso  \begin{cases} P_{(2^{2m})}\circ P_{(s-4m)} & \hbox{if } k=4m  \leq s\\
                                       0 &\hbox{otherwise,}\end{cases}
\end{align*}
where the second isomorphism holds because the
$\GL_2(\Z)$ action on $\wV{\h}{k}$ commutes with the $\SG_{k}\times \SG_{s-k}$ action and the $\SG_{s-k}$ action is trivial.

We calculate the irreducible decomposition of  $P_{(2^{2m})}\circ P_{(s-4m)}$ using the Pieri rule, which says that the components are obtained by adding $s-4m$ boxes in different columns to the Young diagram for $\lambda=(2^{2m})$. The only legal way to do this is to put $0$, $1$, or $2$ boxes in a new bottom row and add the rest to the first row. The resulting partitions are  $(s-4m-j,2^{2m-1},j)$ for $j=0,1,2$.

The second summand is
$$
H^1\big(\GL_2(\Z); H^{k-1}(F_n^{\,s})\big) \iso H^1\big(\GL_2(\Z);\wV{\h}{k-1}\big)\circ P_{(s-k+1)}.
$$
By Lemma~\ref{thm:GL}, $H^1(\GL_2(\Z);\wV{\h}{k-1})$ is  nonzero only when $k$ is odd, in which case we have identified it as an $\SG_{k-1}$-module which we named $W_{k-1}$. Inducing this up to $\SG_s$ produces $\binom{s}{k-1}$
copies of $W_{k-1}$, permuted by the action of $\SG_s$.

The first few rows and columns of the spectral sequence  look like this:
\begin{center} \begin{tikzpicture}
\draw (-1,0) to (8,0);
\draw (-.5,-.5) to (-.5,5.5);
\node (00) at (1,.5) {$P_{(s)}$};
\node (01) at (1,1.0) {$0$};
\node (02) at (1,1.5) {$0$};
\node (03) at (1,2.0) {$0$};
\node (04) at (1,2.5) {$P_{(2^2)}\circ P_{(s-4)}$};
\node (05) at (1,3) {$0$};
\node (06) at (1,3.5) {$0$};
\node (07) at (1,4) {$0$};
\node (08) at (1,4.5) {$P_{(2^4)}\circ P_{(s-8)}$};
\node (09) at (1,5) {$0$};
\node (10) at (4,.5) {$0$};
\node (11) at (4,1) {$0$};
\node (12) at (4,1.5) {$W_2\circ P_{(s-2)}$};
\node (13) at (4,2) {$0$};
\node (14) at (4,2.5) {$W_4\circ P_{(s-4)}$};
\node (15) at (4,3) {$0$};
\node (16) at (4,3.5) {$W_6\circ P_{(s-6)}$};
\node (17) at (4,4) {$0$};
\node (18) at (4,4.5) {$W_8\circ P_{(s-8)}$};
\node (19) at (4,5) {$0$};
\node (10) at (7,.5) {$0$};
\node (11) at (7,1) {$0$};
\node (12) at (7,1.5) {$0$};
\node (13) at (7,2) {$0$};
\node (14) at (7,2.5) {$0$};
\node (14) at (7,3) {$0$};
\node (14) at (7,3.5) {$0$};
\node (14) at (7,4) {$0$};
\node (14) at (7,4.5) {$0$};
\node (14) at (7,5) {$0$};
\end{tikzpicture}
\end{center}
Since $E_2=E_\infty$, the result follows.
\end{proof}

\begin{remark}
The dimension of $P_{(2^{2m})}$ can be computed by the hook-length formula (see, e.g.,~\cite{FH}); it is the $2m$-th Catalan number $C_{2m}=\frac{1}{2m+1} {\binom{4m}{2m}}$.
The induced representation $P_{(2^{2m})}\circ P_{(s-4m)}=H^{4m}(\Gamma_{2,s})$ consists of
$\binom{s}{4m}$
copies of this, so has  dimension  equal to
$$
\frac{1}{2m+1}
\binom{4m}{2m}\binom{s}{4m}
=\frac{s!}{(s-4m)!(2m+1)!(2m)!}.
$$
If $s\geq 4m+2$ then the irreducible decomposition of $P_{(2^{2m})}\circ P_{(s-4m)}$ obtained by the Pieri rule is
$$P_{(2^{2m})} \circ P_{(s-4m)} = P_{(s-4m+2,2^{2m-1})} \oplus P_{(s-4m+1,2^{2m-1},1)}\oplus P_{(s-4m,2^{2m})}
$$
\end{remark}

\smallskip
\begin{remark}
Using the decomposition of $W_q$ into irreducible $\SG_q$-modules in Remark~\ref{rem:WComputations} one  can  use the Pieri rule to obtain the decomposition of $H^{i}(\Gamma_{2,s})$ into irreducible $\SG_s$-modules  for odd $i$. For example
$$
H^7(\Gamma_{2,10}) = W_6\circ P_{(4)} = P_{(2,1^4)}\circ P_{(4)}=
P_{(6,1^4)} \oplus
P_{(5,2,1^3)} \oplus
P_{(5,1^5)} \oplus
P_{(4,2,1^4)}.
$$
\end{remark}

The  dimension and module structure of the cohomology of $\Gamma_{1,s}$ and $\Gamma_{2,s}$ for $s\leq 10$ are summarized in the tables at the end of the paper.

\begin{remark}
\label{rem:cusp}
The calculation of the map on cohomology induced by inclusion $P\to \GL_2(\Z)$ in Section~\ref{subsec:modular} together with the decomposition
$$
H^1\big(\GL_2(\Z);\wV{\h}{q}\big)  =
\bigoplus_{0\leq k\leq \frac{q}{2}} H^1\big(\GL_2(\Z);\Sym^{q-2k}\h \otimes {\det}^{k} \big)\otimes P_{(2^k,1^{q-2k})}.
$$
given in the proof of   Lemma~\ref{thm:GL}   shows that the image of the map
$$
H^1\big(\GL_2(\Z); \wV{\h}{q} \big) \longrightarrow H^1\big(P; \wV{\h}{q} \big)
$$
is isomorphic to $\displaystyle \bigoplus_{2k < q,\,\, k\,\, \mathrm{odd}} P_{(2^k,1^{q-2k})}$
 for $q$ even.  Combining with this with Theorem~\ref{thm:ranktwo} gives a projection
$$
H^i(\Gamma_{2,s}) \longrightarrow \bigoplus_{2k < i-1,\,\, k\,\, \mathrm{odd}} P_{(2^k,1^{i-2k-1})} \circ P_{(s-i+1)}
$$
for $i$ odd. This projection will be useful for constructing nice homology classes from cohomology classes, which we do in Section~\ref{subsec:Rank2odd}.
\end{remark}

\subsection{Arbitrary rank}
The representation theory we used to compute the cohomology of $\Gamma_{2,s}$ gives information about the cohomology of $\Gamma_{n,s}$ for all values of $n$.  In this section we show how this works.

\begin{theorem}
\label{thm:2mn}
If $s \geq n(2m+1)$  then $H^{2mn}(\Gamma_{n,s})$ contains the $\SG_s$-module $P_{(s-2mn,n^{2m})}$ as a direct summand
with multiplicity $1$.  In particular, $H^{2mn}(\Gamma_{n,s})\neq 0$ for all $s\geq 2mn+n$.
\end{theorem}
\begin{proof}
The
 $E_2$ term  of the spectral sequence~\eqref{eq:LS} is $H^p\big(\Out(F_n);H^q(F_n^{\,s})\big)$. The $p=0$ column is straightforward to calculate because it is simply a calculation of $\GL_n(\Z)$ invariants of a well-understood module. The other columns consist of groups that are not known, so our strategy will be to look for $\SG_s$-representations in the $p=0$ column that cannot appear in the other columns. Such a representation cannot support a nontrivial differential, as all differentials are $\SG_s$-equivariant, so survives to $E_\infty$ and hence to $H^*(\Gamma_{n,s})$.

The action of $\Outn$ on $\h=H^1(F_n)\iso \F^n$ factors through the usual action of $\GL_n(\Z)$ on $\h$, and as before, using Lemma~\ref{lem:Fn} we have
\begin{align*}
H^0\big(\Out(F_n);H^q(F_n^{\,s})\big)&=H^0\left(\Out(F_n);\wV{\h}{q}\circ P_{(s-q)}\right)\\
&=H^0\left(\Out(F_n);\wV{\h}{q}\right)\circ P_{(s-q)}\\
&=H^0\left(\GL_n(\Z);\wV{\h}{q}\right)\circ P_{(s-q)}.
\end{align*}

By Schur-Weyl duality, $\wV{\h}{q} \iso \bigoplus_{|\lambda|=q} \SF{\lambda}\h\otimes P_{\lambda^\prime} $, where $ \SF{\lambda}\h$ is the irreducible $\GL_n$-representation corresponding to $\lambda$. It follows from the character formula~\cite[Theorem 6.3]{FH},
that $\SF{\lambda} \h$ is $1$-dimensional if and only if  $q$ is a multiple of $n$, say $q=kn$ and $\lambda = (k^n)$. In this case $\SF{(k^n)}\h$  is the $1$-dimensional $\GL_n(\F)$-representation which is the $k$th power of the determinant (the Schur polynomial is $S_{k^n}=(x_1x_2\ldots x_n)^k$).  Thus  $\SF{\lambda}\h$ is a trivial $\GL_n(\Z)$-module only when $q=2mn$ is an even multiple of $n$   and  $\lambda=((2m)^n)$.
We conclude that
$H^0\big(\Out(F_n);H^q(F_n^{\,s})\big)=0$  unless $q=2mn$
in which case we have
$$
H^0\big(\Out(F_n);H^{2mn}(F_n^{\,s})\big)=  P_{((2m)^n)^\prime}\circ P_{(s-2mn)}= P_{(n^{2m})}\circ P_{(s-2mn)}.
$$

Using the Pieri rule to decompose this representation, we see that as long as $s-2mn\geq n$ one of the terms we get is $P_{(s-2mn,n^{2m})}$, obtained by adding one box below each existing column and the rest to the right of the first row; this is illustrated in Figure~\ref{fig:induced}.
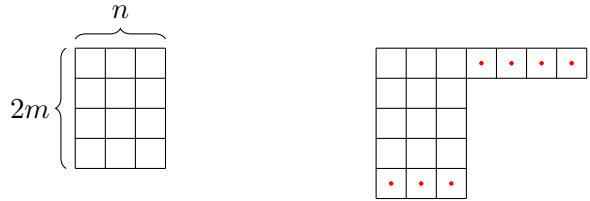
\begin{figure}[b]
\begin{center}
 \begin{tikzpicture}[scale=.4]
 \draw (0,0) to (3,0);
 \draw (0,1) to (3,1);
 \draw (0,2) to (3,2);
 \draw (0,3) to (3,3);
 \draw (0,4) to (3,4);
 \draw (0,0) to (0,4);
 \draw (1,0) to (1,4);
 \draw (2,0) to (2,4);
 \draw (3,0) to (3,4);
 \draw[decorate,decoration={brace,  amplitude=4pt}] (0,4.3) -- (3,4.3);
 \node (n) at (1.5,5.2) {$n$};
 \draw[decorate,decoration={brace,  amplitude=4pt}] (-.3,0) -- (-.3,4);
 \node (2m) at (-1.5,2) {$2m$};
\begin{scope} [xshift=10cm]
 \draw (0,-1) to (3,-1);
 \draw (0,0) to (3,0);
 \draw (0,1) to (3,1);
 \draw (0,2) to (3,2);
 \draw (0,3) to (7,3);
 \draw (0,4) to (7,4);
 \draw (0,-1) to (0,4);
 \draw (1,-1) to (1,4);
 \draw (2,-1) to (2,4);
 \draw (3,-1) to (3,4);
 \draw (4,3) to (4,4);
  \draw (5,3) to (5,4);
   \draw (6,3) to (6,4);
    \draw (7,3) to (7,4);
\fill [red] (.5,-.5) circle (.075);
\fill [red] (1.5,-.5) circle (.075);
\fill [red] (2.5,-.5) circle (.075);
\fill [red] (3.5,3.5) circle (.075);
\fill [red] (4.5,3.5) circle (.075);
\fill [red] (5.5,3.5) circle (.075);
\fill [red] (6.5,3.5) circle (.075);
 \end{scope}
 \end{tikzpicture}
 \end{center}
\caption{Adding boxes to $(n^{2m})$ to obtain the Young diagram for one term of the induced module $P_{(n^{2m})}\circ P_{(s-2mn)}$}
\label{fig:induced}
\end{figure}

We have shown $P_{(s-2mn,n^{2m})}$ occurs in $E^{0,2mn}_2$.  We now claim it does not appear in any row below the $2mn$-th row, so that all differentials from $P_{(s-2mn,n^{2m})}$ must vanish, and $P_{(s-2mn,n^{2m})}$ survives in $H^*(\Gamma_{n,s})$.  Since
\begin{align*}
E_2^{p,q}&=H^p\big(\Outn, H^q(F_n^s)\big)\\
&=H^p\left(\Outn, \wV{\h}{q}\circ P_{(s-q)}\right)\\
&=\bigoplus_{|\lambda|=q} H^p\left(\Outn, \SF{\lambda}\h\otimes P_{\lambda^\prime}\circ P_{(s-q)}\right)\\
&=\bigoplus_{|\lambda|=q} H^p\left(\Outn, \SF{\lambda}\h\right)\otimes P_{\lambda^\prime}\circ P_{(s-q)},
\end{align*}
it suffices to show that $P_{(s-2mn,n^{2m})}$ cannot occur as a term in any of the induced modules $P_{\lambda^\prime}\circ P_{(s-q)}$ with $|\lambda|=q <2mn$. The is the case because the first row of any diagram appearing in $P_{\lambda^\prime}\circ P_{(s-q)}$ has length
at least $s-q > s-2mn$.
\end{proof}

\begin{remark}
The module $P_{(s-2mn,n^{2m})}$ used in the above proof is only a tiny piece of $E^{0,2mn}_2=H^0\big(\Out(F_n);H^q(F_n^{\,s})\big)$.   It seems likely that a much larger part  survives to infinity in the spectral sequence and thus  contributes to the cohomology of $\Gamma_{n,s}$.
\end{remark}

\section{Subgroups supporting homology classes in $\Gamma_{1,s}$ and $\Gamma_{2,s}$.}

In later sections of the paper it will be more natural to work with homology than cohomology.  The universal coefficient theorem formally allows us to do this, since we have finite-dimensional homology groups and coefficients in a field.  In this section we show that in rank 1 and 2 we can also describe some homology classes more directly as classes supported on certain easily understood subgroups.

\subsection{Rank one} \label{subsec:alpha_k} The situation for rank 1 is quite simple so we
describe this first.  By Proposition~\ref{prop:rankone} the odd-dimensional cohomology of $\Gamma_{1,s}$ vanishes and
the inclusion of $\Z^{s-1}$ into $\Gamma_{1,s}$ induces an isomorphism on cohomology in even degrees, hence this holds also for homology.  This implies that the top homology class of any subgroup of even rank in
$\Z^{s-1}$ maps to a nontrivial class in $H_*(\Gamma_{1,s})$, and $H_*(\Gamma_{1,s})$ has a basis of such classes.
If $s=2k+1$ then the entire subgroup $\Z^{s-1}=\Z^{2k}$ has even rank and its top homology class maps to a nontrivial class
$\alpha_k\in H_{2k}(\Gamma_{1,2k+1})$, which is well-defined up to sign.  The class $\alpha_k$ will be used to construct the Morita class $\mu_k$ in Section~\ref{subsec:Morita}.

\subsection{Rank two, even homology degree} Now we turn to rank 2, where $H_*(\Gamma_{2,s})$ is considerably more complicated.
This extra complication is relatively mild in even degrees, so we examine those first.  It suffices to consider $H_{4k}(\Gamma_{2,s})$ since $H_{4k+2}(\Gamma_{2,s})=0$ by Theorem~\ref{thm:ranktwo}.

\begin{notation} Throughout this section and the next we fix generators $x$ and $y$ for $F_2$ and we let $\x$ and $\y$ denote their images in $H_1(F_2)$, with  $\x^*$ and $\y^*$  the dual basis of $H^1(F_2)$.  We also set $\hone=H_1(F_2)$; the notation is meant to distinguish it from $\h=H^1(F_2)$.
\end{notation}

For disjoint subsets $I$ and $J$ of $\{1,2,\ldots,s\}$ let $A_{I,J}$ be the abelian subgroup of $F_2^s$ consisting of $s$-tuples with powers of $x$ in the $I$ coordinates, powers of $y$ in the $J$ coordinates, and the identity in the other coordinates.  We have inclusions $A_{I,J}\subset F_2^s\subset\Gamma_{2,s}$, and we let $\alpha_{I,J}\in H_*(\Gamma_{2,s})$ be the image of a generator of the top-dimensional homology of $A_{I,J}$.

\begin{proposition}\label{prop:alphaIJ}
If $|I|=|J|=2k$ for some $k$ then the class $\alpha_{I,J}\in H_{4k}(\Gamma_{2,s})$ is nonzero and these classes $\alpha_{I,J}$ generate $H_{4k}(\Gamma_{2,s})$.
\end{proposition}

\begin{proof}
First we show that $\alpha_{I,J}$ is nonzero when $|I|=|J|=2k$ for some $k\geq 1$. (Here $s\geq 4k$ since $A_{I,J}\subset F_2^s$.)  We do this by finding a cohomology class in $H^{4k}(\Gamma_{2,s})$ that pairs nontrivially with $\alpha_{I,J}$.

By Lemma~\ref{thm:GL} we have $H^1(\GL_2(\Z);\wH{4k-1})=0$.  Therefore
$$
H^1\big(\GL_2(\Z);H^{4k-1}(F_2^s)\big)=H^1\big(\GL_2(\Z);\wH{4k-1}\big)\circ P_{(s-4k)}=0,
$$
and
$$
H^{4k}(\Gamma_{2,s})\iso H^0\big(\GL_2(\Z);H^{4k}(F_2^s)\big) \iso H^0\big(\GL_2(\Z);\wH{4k}\big)\circ P_{(s-4k)}.
$$
Thus to compute $H^{4k}(\Gamma_{2,s})$ as an $\SG_s$-module it suffices to understand the invariants of the diagonal action of $\GL_2(\Z)$ on $\wH{4k}$.  As a  $\GL_2(\Z)$-module, $\wH{4k}$ is the same as $\tH{4k}$, and we describe the (classical) answer below.

A straightforward calculation shows that the diagonal action of an element $T \in \GL_2(\Z)$ on $\h\otimes \h$ sends $\omega^* = \x^*\otimes \y^*-\y^*\otimes \x^*$ to $(\det T)\cdot\omega^*$.  The diagonal action on $\h^{\otimes 4k}$ sends $(\omega^*)^{\otimes 2k}$ to $(\det T)^{2k}(\omega^*)^{\otimes 2k}$, so since $\det T=\pm 1$ and $2k$ is even this  is an invariant.  Any permutation of the indices $\{1,\ldots,4k\}$ produces another invariant, and the invariants defined in this way span the entire space of invariants (see, e.g.,~\cite{FH} for details).  Note that each term in each of these invariants has an equal number of $\x^*$'s and $\y^*$'s, so this is true of any invariant.

Suppose first that $s=4k$ and let $I= \{1,3,\ldots,4k-1\}$, the odd indices, and $J=\{2,4,\ldots,4k\}$, the even indices.  Then the image  of $H_{4k}(A_{I,J})\iso H_1(\Z)\otimes\cdots\otimes H_1(\Z)$ in $H_1(F_2^{4k})\iso \hone \otimes\cdots\otimes \hone$ is generated by ${\bf z}= \x\otimes \y\otimes\cdots \otimes \x\otimes \y$.  Since this matches the first term of $(\omega^*)^{\otimes 2k}$, the cohomology class $(\omega^*)^{\otimes 2k}\in H^{4k}(\Gamma_{2,4k})$ pairs nontrivially with $\alpha_{I,J}$, which is the image of ${\bf z}$ in $H_{4k}(\Gamma_{2,4k})$.  This shows that $\alpha_{I,J}$ is nonzero.  Permuting the indices produces other nonzero classes $\alpha_{I,J}$ that span $H_{4k}(\Gamma_{2,4k})$ since the corresponding cohomology classes span $H^{4k}(\Gamma_{2,4k})$.

If $s>4k$, any of the natural inclusions $\Gamma_{2,4k}\to \Gamma_{2,s}$ (given by gluing extra leaves to the leaf vertices of $\Gr_{2,4k}$ and extending maps by the identity) induces an injection $H_{4k}(\Gamma_{2,4k})\to H_{4k}(\Gamma_{2,s})$ mapping each $\alpha_{I,J}$  nontrivially. On homology, the image of this map depends only on the inclusion $\{1,\ldots,2k\}\to \{1,\ldots,s\}$ of leaf vertices.  Since $H_{4k}(\Gamma_{2,s})=H_{4k}(\Gamma_{2,4k})\circ P_{(s-4k)}$, these classes span all of $H_{4k}(\Gamma_{2,s})$.
\end{proof}

\begin{remark} \label{remark:alpha-unequalsize}
If $I$ and $J$ are disjoint subsets of $\{1,\ldots,s\}$ of
 different size, then $\alpha_{I,J}$ is trivial because the top-dimensional homology class of $A_{I,J}$ is a simple tensor with an unequal number of $\x$'s and $\y$'s, so every invariant evaluates trivially on it.
\end{remark}

\begin{remark}
\label{remark:alpha-basis}
The classes $\alpha_{I,J} \in H_{4k}(\Gamma_{2,s})$ are not linearly independent.  There are several possible ways to obtain a subset of these classes which form  a basis of $H_{4k}(\Gamma_{2,s})$.
Since the dimension of $H_{4k}(\Gamma_{2,s})$ is closely related to the dimension of $P_{(2k,2k)}$,
which is equal to the Catalan number $C_{2k}$, one can use  combinatorial objects such as non-crossing partitions or Young tableaux  to describe such a basis.  Here is one possible description of a basis.

\begin{claim}
The space $H_{4k}(\Gamma_{2,s})$ has a basis consisting of those $\alpha_{I,J} $ for which
$I= \{i_1 < i_2 <  \dots < i_{2k} \}$ and $J= \{j_1 < j_2 <  \dots < j_{2k} \}$ are disjoint subsets of $\{1,\dots,s\}$ such that $i_t < j_t$ for each $t=1,\dots,2k$.
\end{claim}

The proof of this involves a deeper use of representation theory so we will not give it here.
\end{remark}

\subsection{Rank two, odd homology degree} \label{subsec:Rank2odd}
Constructing classes of odd homology degree is more difficult since no subgroups of $F_2^s$ support such classes.  As a result  we must use slightly more complicated subgroups of $\Gamma_{2,s}$. We use the same notation as in the previous section for generators of $F_2$ and its homology and cohomology.

Fix disjoint subsets $I, J \subset \{1,\dots,s\}$.
Let $B_{I,J} \iso F_2^{|I|}\times \Z^{|J|}$
be the subgroup of $F_2^s$ consisting of $s$-tuples with arbitrary elements in the $I$ coordinates, powers of $x$ in the $J$-coordinates and the identity in coordinates not indexed by $I$ or $J$.
Recall that $\Gamma_{2,s}$ maps onto $\Gamma_{2,0}\iso \GL(2,\Z)$ with kernel $F_2^s$.  Let $P\iso\Z$ denote the unipotent subgroup of $\GL_2(\Z)$ generated by $\left(\!\begin{smallmatrix} 1 & 1 \\ 0 & 1 \end{smallmatrix}\!\right)$, corresponding to the (outer) automorphism of $F_2$ fixing $x$  and sending $y$ to $xy$.  Lift the generator of $P$ to an element $\varphi \in \Gamma_{2,s}$ that wraps the $y$-loop of $X_{2,s}$ around both itself and the $x$-loop, and fixes the $x$-loop and all leaves.  This normalizes $B_{I,J}$, and we  define $M_{I,J}$ to be the subgroup of $\Gamma_{2,s}$ generated by $\varphi$ and $B_{I,J}$.  We now have a commutative diagram
$$
\begin{tikzpicture} 
\matrix (m)[matrix of math nodes, row sep=2em, column sep=2.5em, text height=1.5ex, text depth=0.25ex]
{
1 &  F_2^{s} &  \Gamma_{2,s} & \GL_2(\Z) & 1 \\
1 & B_{I,J}      &  M_{I,J}              & P
& 1.\\
};
\draw[->] (m-1-1) edge (m-1-2);
\draw[->] (m-1-2) edge (m-1-3);
\draw[->] (m-1-3) edge (m-1-4);
\draw[->] (m-1-4) edge (m-1-5);
\draw[->] (m-2-1) edge (m-2-2);
\draw[->] (m-2-2) edge (m-2-3);
\draw[->] (m-2-3) edge (m-2-4);
\draw[->] (m-2-4) edge (m-2-5);
\draw[right hook->] (m-2-2) edge (m-1-2);
\draw[right hook->] (m-2-3) edge (m-1-3);x
\draw[right hook->] (m-2-4) edge (m-1-4);
\end{tikzpicture}
$$
Note that $M_{I,J}$ splits as the product $M_I\times \Z^{|J|}$ where $M_{I}=M_{I,\varnothing}$ and $\Z^{|J|}=B_{\varnothing,J}$.

We will be interested in the cases when $|I|$ and $|J|$ are even, and we let $|I|+|J|=2k$. The top-dimensional homology of $B_{I,J}$ is $H_{2k}(B_{I,J})\iso{\hone}^{\otimes |I|}\otimes {\Xone}^{\otimes |J|}$, where
$\Xone\iso \F$ is the subspace of $\hone=H_1(F_2)$ spanned by $\x$.
From the Leray-Serre spectral sequence it follows that $H_*(M_{I,J})$ vanishes above dimension $2k+1$ and $H_{2k+1}(M_{I,J})=H_1\big(P;H_{2k}(B_{I,J})\big)$.

\subsubsection{The case $|I|$=2.}  The analysis of the case $|I|=2$ is  easier than the general case and  will suffice for our construction of the Eisenstein classes in Section~\ref{subsec:eisenstein}, so we begin with this case.

We first compute $H_{2k+1}(M_{I,J})$.  From the splitting $M_{I,J} = M_I\times \Z^{|J|}$ we have
$$
H_{2k+1}(M_{I,J})\iso H_3(M_I)\otimes H_{2k-2}(\Z^{2k-2}) \iso H_3(M_I) \iso H_1\big(P;H_2(B_{I,\varnothing})\big).
$$
Since $P\iso\Z$ the first homology $H_1\big(P;H_2(B_{I,\varnothing})\big)$ is just the invariants of the action of $P$ on $H_2(B_{I,\varnothing})=\hone\otimes \hone$.  This is the diagonal action, where  $P$  acts on $\hone$ by sending $\x\to \x$ and $\y\to \x+\y$.  It is easy to compute that the space of invariants  is $2$-dimensional, spanned by $\x\otimes\x$ and $\omega=\x\otimes\y -\y\otimes\x$.  Thus
 $$
 H_{2k+1}(M_{I,J})\iso H_1\big(P;H_{2k}(B_{I,J})\big)\iso [H_{2k}(B_{I,J})]^P\iso [{\hat H}^{\otimes 2}\otimes {\Xone}^{\otimes 2k-2}]^P \iso\F^2
$$
with basis $\x^{2k}$ and $\omega\otimes\x^{2k-2}$.

The
class in $H_{2k+1}(M_{I,J})$ corresponding to $\omega\otimes\x^{2k-2}$ is the one whose image $m_{I,J}$ in $H_{2k+1}(\Gamma_{2,s})$
will be used
as a building block for Eisenstein classes.  In Section~\ref{subsec:geoEisenstein}  we give a different, more geometric construction of this class as the image of the fundamental class of a manifold mapped into a moduli space of graphs.

The natural actions of $\SG_I=\SG_2$ and $\SG_J=\SG_{2k-2}$ on $H_{2k+1}(M_{I,J})$ are easy to describe since these two symmetric groups act separately on the factors of the splitting $M_{I,J}=M_I\times \Z^{|J|}$.  For $\SG_I$ the transposition $\sigma$ interchanges the $F_2$ factors of $F_2^2$, so $\sigma(\x\otimes\x)=-\x\otimes\x$ since the cross product in the K\"unneth formula is anti-symmetric. For the class $\omega=\x\otimes\y -\y\otimes\x$ we have $\sigma(\omega)=\omega$ since  we get one minus sign from the minus sign in $\omega$ and another from anti-symmetry in the K\"unneth formula. For an element $\sigma\in\SG_J$ the action on $\x^{2k-2}$ is just by the sign of $\sigma$.

Because of the anti-symmetric action of $\SG_J$, the classes $m_{I,J}\in H_{2k+1}(\Gamma_{2,s})$ are well-defined only up to sign.  We now show they are nontrivial and describe how much of $H_{2k+1}(\Gamma_{2,s})$ they account for.

\begin{proposition}\label{prop:oddtwo}
For $I, J \subset \{1,\dots,s\}$ with $|I|=2$ and $|J|=2k-2\geq 2$
the map $H_{2k+1}(M_{I,J}) \to H_{2k+1}(\Gamma_{2,s})$ induced  by inclusion has $1$-dimensional image spanned by  $m_{I,J}$, and the $\SG_s$-module generated by $m_{I,J}$ is isomorphic to $P_{(2,1^{2k-2})} \circ P_{(s-2k)}$.
\end{proposition}

In particular, when $s=2k$ the classes $m_{I,J}$ generate $H_{2k+1}(\Gamma_{2,2k})$ only when $k\leq 3$; this follows from Theorem~\ref{thm:ranktwo}.

\begin{proof}
To prove that $m_{I,J}$ is nonzero we find a cohomology class that pairs nontrivially with it.

Assume first that $s=2k$.  To simplify notation we also assume  $I=\{1,2\}$ and $J=\{3,\ldots, 2k\}$  and set
$B=B_{I,J},$   $M=M_{I,J}$ and $m=m_{I,J}$.
The map from $
H^{2k+1}(\Gamma_{2,2k})\iso H^1\big(\GL_2(\Z);H^{2k}(F_2^{2k})\big)$
to $ H^{2k+1}(M)$
induced by the inclusion $M \hookrightarrow \Gamma_{2,2k}$ factors as
\begin{equation*}\label{composition}
H^1\big(\GL_2(\Z);H^{2k}(F_2^{2k})\big) \longrightarrow H^1\big(P;H^{2k}(F_2^{2k})\big) \longrightarrow   H^1\big(P;H^{2k}(B )\big)
\end{equation*}
where the first map is induced by the inclusion $P\hookrightarrow \GL_2(\Z)$
and the second by the  map of coefficients induced by $B \hookrightarrow  F_2^{2k}$.

In Remark~\ref{rem:cusp} we pointed out that the first map is a surjection onto  the odd terms  of the decomposition of  $H^1\big(P;H^{2k}(F_2^{2k})\big)$ into irreducible $\SG_{2k}$-modules.  Here is a more explicit description of this map.  Since $P\iso \Z$, for any $P$-module $V$ we have $H^1(P;V)\iso V_P$. If $V$ is a vector space there is a canonical isomorphism $(V^*)_P \iso (V^P)^*$ (sending $f$ to its restriction to $V^P$).  In particular, using the universal coefficient theorem we get natural isomorphisms
$$
H^1\big(P;H^{2k}(F_2^{2k})\big)=
\big[H^{2k}(F_2^{2k})\big]_P\iso \big([H_{2k}(F_2^{2k})]^P\big)^* =
\big([H_1(F_2)^{\otimes 2k}\otimes \alt]^P\big)^*=
\big([\hone^{\otimes 2k}]^P\big)^*\otimes \alt,
$$
where $\alt$ refers to the $\SG_{2k}$-action.
Now recall that the space of $P$-invariants in $\hone\otimes\hone$ is spanned by   $\omega$ and $\x^2$. Since $\hone\otimes \hone= \ext^2\hone\oplus\Sym^2\hone$, this shows that the subspace of $P$-invariants in each summand is $1$-dimensional. This is a special instance of the general fact that the space of $P$-invariants in $ (\ext^2\hone)^{\otimes \ell}\otimes \Sym^{2k-2\ell}\hone$ is $1$-dimensional, spanned by $\omega^\ell\otimes  \x^{2k-2\ell}$.  The Schur-Weyl decomposition of $\hone^{\otimes 2k}$ then shows that the image of the first map can be identified with
$\bigoplus_{\ell<  k\,\, \mathrm{odd}}P_{(2^\ell,1^{2k-2\ell})}$, as in Remark~\ref{rem:cusp}, where each term is generated by $\omega^\ell\otimes  \x^{2k-2\ell}$ as an $\SG_{2k}$-module.

For the second map, note that
\begin{align*}
H^1\big(P;H^{2k}(B)\big)&=\big[\h\otimes \h \otimes X^{\otimes 2k-2}\big]_P\otimes \alt \\
&\iso\big([\hone\otimes \hone \otimes \Xone^{\otimes 2k-2} ]^P\big)^*\otimes \alt,
\end{align*}
where this $\alt$ refers to the action of  $\SG_2\times \SG_{2k-2}$ which permutes the factors of $B=F_2^2\times \Z^{2k-2}$ independently.  Thus the map
$H^1\big(P;H^{2k}(F_2^{2k})\big) \to H^1\big(P;H^{2k}(B)\big)$ becomes
\begin{align*}
\big([H_{2k}(F_2^{2k})]^P\big)^*  &\longrightarrow \big( [H_{2k}(B)]^P\big)^* \\
\intertext{i.e.,}
\big([\hone^{\otimes 2k}]^P\big)^*\otimes \alt  &\longrightarrow \big([ \tV{\hone}{2} \otimes \tV{\Xone}{2} ]^P\big)^*\otimes \alt.
\end{align*}

The map on the first factor is just the transpose of the inclusion map
$ [\tV{\hone}{2}\otimes \tV{\Xone}{2}]^P \hookrightarrow [\tV{\hone}{2k}]^P$ and in particular
sends
$(\omega\otimes \x^{ 2k-2})^*$ to itself.

Since $(\omega\otimes \x^{ 2k-2})^*$ is in the image of the first map $H^1\big(\GL_2(\Z);H^{2k}(F_2^{2k})\big)\to H^1\big(P;H^{2k}(F_2^{2k})\big)$, there is a cohomology class in $H^{2k+1}(\Gamma_{2,2k})$ which hits it under the composition
$$
H^{2k+1}(\Gamma_{2,2k})=H^1\big(\GL_2(\Z);H^{2k}(F_2^{2k})\big)\longrightarrow H^1\big(P;H^{2k}(F_2^{2k})\big)\longrightarrow H^1\big(P;H^{2k}(B)\big)=H^{2k+1}(M).
$$
This class evaluates nontrivially on the image $m\in H_{2k+1}(\Gamma_{2,2k})$ of $\omega\otimes \x^{ 2k-2} \in H_{2k+1}(M)$, showing that $m$ is nontrivial.

Any permutation of the indices $\{1,\ldots,2k\}$ gives another class in $H_{2k+1}(\Gamma_{2,2k})$.  The $\SG_{2k}$-submodule
generated by $m$   is isomorphic to $P_{(2,1^{2k-2})}$, which
coincides with $H_{2k+1}(\Gamma_{2,2k})$
only when $k=2,3$.
This completes the proof of the proposition  for $s=2k$.

The generalization to $s>2k$ is straightforward, since
$
H^{2k+1}(\Gamma_{2,s})\iso H^1\big(\GL_2(\Z);H^{2k}(F_2^s)\big) \iso H^1\big(\GL_2(\Z);\wH{2k}\big)\circ P_{(s-4k)}.
$
The $\SG_s$-module generated by the image of   $H_{2k+1}(M_{I,J} )$ in $H_*(\Gamma_{2,s})$ is isomorphic to $P_{(2,1^{2k-2})} \circ P_{(s-2k)}$.
\end{proof}

As 
was noted when the classes $m_{I,J}$ were defined, they are invariant under transposing the two indices in $I$ and anti-invariant under permutations of the indices in $J$.  When $s=2k$
we can obtain a class which is anti-invariant under a larger group of permutations by adding together signed images of $m_{I,J}$ under appropriate permutations. Specifically, let $I=\{1,2\},$  $J=\{3,\dots,2k\}$ and $m=m_{I,J}$ as in the proof of Proposition~\ref{prop:oddtwo}  and choose an index $i\in\{1,\dots,2k\}$. Then define
$$
m_{i} = \sum_{\sigma(1)=i} \sign(\sigma) \sigma(m),
$$
where the sum is over all permutations $\sigma\in \SG_{2k}$ which send $1$ to $i$.
The class $m_{i}$  is then anti-invariant under ${\rm stab}_{\SG_{2k}}(i)$.  For example when $I=\{1,2\}$ and $J=\{3,4\}$, so $m$ corresponds to $\x\y\x\x -\y\x\x\x$ (omitting tensor symbols for simplicity),  the class $m_1$ corresponds to $6\y \x \x \x -2\x \y \x \x -2\x \x \y \x-2\x \x \x \y$, up to sign.  The formula for $m_2$ is similar, and one sees that $m =\pm\frac{1}{8} (m_1-m_2)$.  For larger $J$ there are analogous formulas.

\subsubsection{The general case}

Now we consider the general case  $|I| = 2 \ell$ for odd $\ell\geq1$.  This is more involved because for $\ell>1$ the top-dimensional cohomology $H^{2k+1}(M_{I,J})$ is  quite large and it is not immediately clear how to pick a distinguished element dual to $\omega^{\otimes \ell}\otimes \x^{2k-2\ell}$.  We settle this by using the unique element which is invariant under the  action of certain involutions.  This is motivated by the case $\ell=1$, where the element $\omega \in \wV{\hone}{2}$ spans the invariants of $\wV{\hone}{2}$ under the action of
the involution $(12)$.

Given any set $T$ of  disjoint transpositions,  let $N_T$ denote the elementary abelian subgroup that  they generate.

\begin{proposition}\label{prop:min} Let $T$ be a set of $\,\ell$ disjoint transpositions of the set $I$.
For $\ell$  odd, the top homology  $H_{2k +1}(M_{I,J})$ contains a unique (up to scalar multiple) element  which is invariant under the action of
$N_T \subset \SG_I$.
The image of this element under the map $H_{2k+1}(M_{I,J}) \to H_{2k+1}(\Gamma_{2,s})$ induced by inclusion is nonzero. The $\SG_s$-module generated by this image is isomorphic to
$P_{(2^{\ell},1^{2k-2\ell})} \circ P_{(s-2k)} $.
\end{proposition}

\begin{proof}
In order to simplify the notation we will assume that
$I=\{1,\dots, 2\ell\}$, $J= \{2\ell + 1, \dots, 2k\}$ and $T=\{(1,2),(3,4),\dots ,(2\ell-1, 2\ell)\}$,  and set  $B_\ell=B_{I,J}$ and $M_\ell=M_{I,J}$.

Recall that $\Xone$ is the $1$-dimensional subspace of $H_1(F_2)$ spanned by $\x$. The actions of $P$ and $N_T$ on $H_{2k}(B_{\ell})=\wV{\hone}{2\ell} \otimes \wV{\Xone}{2k-2\ell} $ commute  so
\begin{align*}
\left[H_{2k +1}(M_{\ell}) \right]^{N_T}
    & = \left[H_1\big(P;H_{2k}(B_{\ell})\big) \right]^{N_T} \\
  & = \left[ [H_{2k}(B_{\ell})]^P \right]^{N_T} \\
    & = \left[[\wV{\hone}{2\ell} \otimes \wV{\Xone}{2k-2\ell}]^P\right]^{N_T} \\
        & = \left[[\wV{\hone}{2\ell} \otimes \wV{\Xone}{2k-2\ell}]^{N_T}\right]^{P} \\
    &=\left[[\wV{\hone}{2\ell}]^{N_T} \otimes \wV{\Xone}{2k-2\ell} \right]^{P}
 \end{align*}
 The space of invariants in $\wV{\hone}{2\ell}$ under the action of   $N_T$  is $1$-dimensional,  spanned by $\omega^{\otimes \ell}$, so the entire space $\left[H_{2k +1}(M_{\ell})\right]^{N_T}$  is  at most $1$-dimensional.  It is exactly $1$-dimensional since
  the element $$m_{\ell} = \omega^{\otimes \ell} \otimes \x^{2k-2\ell}\in   \big[\wV{\hone}{2\ell} \otimes \wV{\Xone}{2k-2\ell}\big]^P$$
is  invariant under the action of  $N_T$.  We will show that  $m_{\ell}$ has nontrivial image in
$H_{2k +1}(\Gamma_{2,s})$ if $\ell$ is odd.

Let $D_\ell$ be the subspace of $H^{2k}(F_2^{2k})=\h^{\wedge 2k}$ generated as a $GL_2(\Z)$-module  by $(\omega^*)^{\otimes\ell}\otimes( \x^*)^{\otimes 2k-2\ell}$.  Thus  $$
D_\ell\iso \SF{(2k - \ell, \ell)}(\h) \cong {\det}^{\ell} \otimes \Sym^{2k-2\ell}(\h).  $$
(See Lemma~\ref{lem:decomp}.) Viewing $D_\ell$ as a submodule of the induced module $H^{2k}(F_2^s)=H^{2k}(F_2^{2k})\circ P_{(s-2k)}$
we see that  $m_\ell$ pairs nontrivially with the $P$-coinvariants in  $D_{\ell}$ (to compute these coinvariants, note that the action of $P$ on $H=H^1(F_2)$ is dual to its action on $H_1(F_2)$ so sends $\x^*\mapsto \x^*+\y^*$ and fixes $\y^*$.)  Therefore $m_{\ell}$ pairs nontrivially with the cohomology class generating $H^1(P;D_{\ell})$.

Since  $\,\ell\,$ is odd  Remark~\ref{rem:cusp} shows that the map $H^1\big(\GL_2(\Z);D_{\ell}\big) \to H^1(P;D_{\ell})$ is surjective. Hence
the class $m_{\ell}$ pairs nontrivially with a class in $H^1\big(\GL_2(\Z);D_{\ell}\big)$ which is the image of a class in
$H^1\big(\GL_2(\Z);H^{2k}(F_2^{s})\big) = H^{2k+1}(\Gamma_{2,{s}})$. This shows that the homology class $m_{\ell}$ is nonzero in $H_*(\Gamma_{2,s})$.

The last statement of the proposition follows from the Schur-Weyl decomposition of  $H_{2k}(B_\ell)$:
$$
H_{2k}(B_{\ell})  =  \wV{\hone}{2\ell} \otimes \wV{\Xone}{2k-2\ell}
     ={\textstyle\bigoplus}_{i\leq \ell}\,\,\SF{(\ell+i,\ell-i)} \hone \otimes P_{(2^{\ell-i},1^{2i})} \otimes \wV{\Xone}{2k-2\ell}
$$
The element $m_\ell=\omega^{\ell}\otimes \x^{2k-2\ell}$ is in the $i=0$ term $\SF{(\ell,\ell)}\hone\iso \det^{\ell}\otimes \Sym^{2k-2\ell}\hone$, so the  $\SG_{2k}$-submodule of $H_{2k}(B_\ell)$ it generates is of type $P_{(2^\ell,1^{2k-2\ell})}$, which is then induced up to  $P_{(2^\ell,1^{2k-2\ell})}\circ P_{(s-2k)}$.
\end{proof}

\begin{remark}   The $\SG_s$-module generated by the element of  $H^{2k+1}(\Gamma_{2,s})$ found in Proposition~\ref{prop:min} must come from the term
$\mathcal M_{2k+2-2\ell}\otimes P_{(2^\ell,1^{2k-2-2\ell})}\circ P_{(s-2k)}$ of the computation of $H^{2k+1}(\Gamma_{2,s})$ in Theorem~\ref{thm:ranktwo}.  In fact it  comes from the map $\mathcal M_{2k+2-2\ell} \to \F$
obtained by evaluating the modular form at infinity.  This is clear from the construction since we are using the parabolic subgroup $P$ and  the inclusion of $P$ into $\GL_2(\Z)$ kills  all other classes (see Section~\ref{sec:modular}).
\end{remark}

\begin{remark}
Recall that  after the proof of Proposition~\ref{prop:oddtwo} we defined classes $m_i$  using a symmetrization procedure.  Similarly, we can use extra symmetrization to obtain classes $m_{I',J'}$ indexed by disjoint sets $I'$ and $J'$ with $|I'| = \ell$ and $|J'| = 2k-\ell$ for $\ell$ odd which span the module $P_{(2^\ell,1^{2k-2\ell})} \circ P_{(s-2k)}$ inside $H_{2k+1}(\Gamma_{2,s})$. These elements are anti-invariant under the action of $\SG_{I'} \times \SG_{J'}$ and invariant under permutations fixing $I'$ and $J'$  pointwise.
These elements generate all of $H_{2k+1}(\Gamma_{2,s})$ if $k \leq 5$.
For $k > 5$ they generate only the homology coming from the parabolic group $P \subset \SL_2(\Z)$.
The elements $m_{I',J'}$ are not linearly independent, but a subset similar to the one described in Remark~\ref{remark:alpha-basis} can be used to form a basis of the corresponding $\SG_s$-module.
\end{remark}

\section{Gluing classes together}
\label{sec:gluing}

Given a graph $\Gr_{n,s}$ we can obtain a new set of graphs $\{\Gr_{n_i,s_i}\}$ by snipping some of the edges at their midpoints; the snipped edges will become leaves in the $\Gr_{n_i,s_i}$.  Conversely, suppose we have a set of graphs $\{\Gr_{n_1,s_1},\ldots, \Gr_{n_k,s_k}\}$ and a {\em gluing pattern} $\phi$ which pairs up some or all of the leaf vertices to form a connected graph $\Gr_\phi$.  An example is shown in Figure~\ref{fig:gluing}.  If $X_\phi$ has rank $n$ and $s$ leaves, then the gluing
$$
\Gr_{n_1,s_1}\cup \cdots\cup \Gr_{n_k,s_k}\longrightarrow \Gr_\phi
$$
induces a homomorphism
$$
p_\phi\colon\Gamma_{n_1,s_1}\times \cdots\times  \Gamma_{n_k,s_k}\longrightarrow
\Gamma_{n,s}.
$$
This in turn induces an {\em assembly map} on homology via the cross product,
$$
A_\phi\colon H_*(\Gamma_{n_1,s_1})\otimes \cdots\otimes  H_*(\Gamma_{n_k,s_k})\longrightarrow H_*(\Gamma_{n,s}).
$$
In particular, if we glue all of the univalent vertices in pairs, we obtain a map to the homology of $\Outn$, and if we glue all but one we obtain a map to the homology of $\Autn$.
We allow leaf vertices of a single $\Gr_{n,s}$ to be glued together.  For example, gluing all four leaf vertices of $\Gr_{2,4}$ in pairs gives an assembly map $H_*(\Gamma_{2,4})\to H_*(\Gamma_{4,0})$ that we use in Section~\ref{subsec:GammaTwoFour}.

\begin{remark}
Assembly maps are associative since this is obviously true for gluing graphs together, and the cross product in homology is associative.  In particular, if a gluing is done in two stages, the assembly map factors through the intermediate stage.
\end{remark}

\begin{remark}
The $\vcd$ of $\Gamma_{n,s}$ is $2n+s-3$ if $n>0$, and for a $k$-fold assembly map as above with each $n_i>0$ this is given by the formula
$$
\vcd(\Gamma_{n,s}) = \vcd(\Gamma_{n_1,s_1}) + \cdots +  \vcd(\Gamma_{n_k,s_k}) + (k-1).
$$
To see this it suffices by induction to consider the case of gluing a single pair of leaves.  If $k=1$ we are gluing two leaves of the same graph together, increasing $n$ by one and decreasing $s$ by two, so the $\vcd$ is unchanged.  If $k=2$ and we glue a leaf of one graph to a leaf of the other  we have $n=n_1 + n_2$ and $s=s_1 +s_2 - 2$, so $2n+s-3$ is one more than the sum
$(2n_1 +s_1 -3)+(2n_2 +s_2 -3)$.
A consequence of this relation between the $\vcd$'s is that a $k$-fold assembly map with $k>1$ cannot produce homology classes in the $\vcd$ of $\Gamma_{n,s}$. This holds even when some $n_i$'s are $0$, provided we exclude trivial factors with $(n_i,s_i)=(0,2)$.
\end{remark}

A different gluing $\phi^\prime$ may also produce a graph  $\Gr_{\phi^\prime}$ of rank $n$ with $s$ leaves.  As noted at the beginning of Section~\ref{subsect:Gns}, a bijection between the leaf vertices of $X_\phi$ and $\Gr_{\phi^\prime}$ determines an isomorphism between the targets of both  assembly maps.  For $\phi^\prime=\phi$ this gives the action of the symmetric group $\SG_s$ on $H_*(\Gamma_{n,s})$.

The left-hand side of the assembly map $A_\phi$ is a priori an $\SG_{s_1}\times \dots\times \SG_{s_k}$-module.  The gluing $\phi$  interacts with the action of $\SG_{s_1}\times \dots\times \SG_{s_k}$ in various ways, which can be explained by the following  two observations:

\begin{enumerate}
\item  If $t_{ij}$ leaves of $\Gr_{n_i,s_i}$ are paired with leaves of $\Gr_{n_j,s_j}$, then a permutation that does the same thing to both sets of leaves does not change the result of the gluing.
\item   If there are $u_i$ unglued leaves in $\Gr_{n_i,s_i}$, then permuting them can be done before or after gluing with the same effect.
\end{enumerate}

The algebraic effect of the first observation is that  the map $A_\phi$ factors through the coinvariants of the diagonal action of $\SG_{t_{ij}}$ on $H_*(\Gamma_{n_i,s_i})\otimes H_*(\Gamma_{n_j,s_j})$.  Here   $\SG_{t_{ij}}$ acts on  $H_*(\Gamma_{n_i,s_i})$ and $H_*(\Gamma_{n_j,s_j})$ by restriction of the $\SG_{s_i}$ and $\SG_{s_j}$ actions.  For example, if the leaves of a graph with exactly $s$ leaves are glued to the leaves of another graph with exactly $s$ leaves, then the assembly map
factors through the space of $\SG_s$ coinvariants:
$$
\begin{tikzpicture} 
\node (D) at (.5,1.4) {$H_p(\Gamma_{n_1,s}) \otimes H_q(\Gamma_{n_2,s})$};
\node (T) at (6.5,1.4) {$H_{p+q}(\Gamma_{n_1+n_2+s-1,0})$};
\node (I) at (3.5,0) {$\big(H_p(\Gamma_{n_1,s}) \otimes H_q(\Gamma_{n_2,s})\big)_{\SG_s}$};
\draw[->] (D) to node[above, pos=0.45] {$A_\phi$} (T);
\draw [->>] (D) to (I);
\draw [dashed, ->] (I) to (T);
\end{tikzpicture}
$$
Now $(P_\lambda\otimes P_\mu)_{\SG_s}$ is zero unless $\lambda=\mu$, in which case it is  $1$-dimensional.
Therefore the assembly map  is trivial  unless some irreducible $P_\lambda$ appears in  the $\SG_s$ decompositions of both
$H_p(\Gamma_{n_1,s})$ and $H_q(\Gamma_{n_2,s})$.

The second observation says that  we can make $A_\phi$ into an $(\SG_{u_1}\times\cdots\times \SG_{u_k})$-module map by realizing $\SG_{u_1}\times\cdots\times \SG_{u_k}$  as the appropriate subgroup of $\SG_s$. (Here again $\SG_{u_i}$ acts on $H_*(\Gamma_{n_i,s_i})$ by restricting the $\SG_{s_i}$ action.)
We can sometimes obtain new information about $A_\phi$ by extending it to an $\SG_s$-module map, which we call $\hat A_{\phi}$.  Thus the range of $\hat A_{\phi}$ is still $H_*(\Gamma_{n,s})$  but the domain of $\hat A_{\phi}$ is the module obtained by inducing $H_*(\Gamma_{n_1,s_1})\otimes \cdots\otimes  H_*(\Gamma_{n_k,s_k})$, considered as a $(\SG_{u_1}\times\cdots\times \SG_{u_k})$-module, up  to  $\SG_s$.
The advantage here is that $\SG_s$-modules and $\SG_s$-module maps between them are very well understood.

\begin{figure}
\begin{center}
\begin{tikzpicture}[scale=.6]
\begin{tikzrankone}{0cm}{0cm}
\tikzhair{c1}{30}
\tikzhair{c2}{0}
\tikzhair{c3}{-30}
\end{tikzrankone}
\begin{tikzranktwov}{4cm}{0cm}
\tikzhair{d1}{120}
\tikzhair{d2}{160}
\tikzhair{d3}{-160}
\tikzhair{d4}{-120}
\end{tikzranktwov}
\draw [dotted,->-=0.5] (c1) to [out=30, in=160] (d2);
\draw [dotted,->-=0.5] (c3) to [out=-30, in=-160] node[below]{$\phi$}  (d3);
\begin{tikzrankone}{10cm}{0cm}
\tikzhair{rc1}{40}
\tikzhair{rc2}{0}
\tikzhair{rc3}{-40}
\end{tikzrankone}
\begin{tikzranktwov}{14cm}{0cm}
\tikzhair{rd1}{120}
\tikzhair{rd2}{150}
\tikzhair{rd3}{-150}
\tikzhair{rd4}{-120}
\end{tikzranktwov}
\draw (rc1) to [out=40, in=150] (rd2);
\draw (rc3) to [out=-40, in=-150]  (rd3);
\end{tikzpicture}
\end{center}
\caption{Making $\Gr_{4,3}$ from $\Gr_{1,3}$ and $\Gr_{2,4}$ using a gluing map $\phi$ }
\label{fig:gluing}
\end{figure}
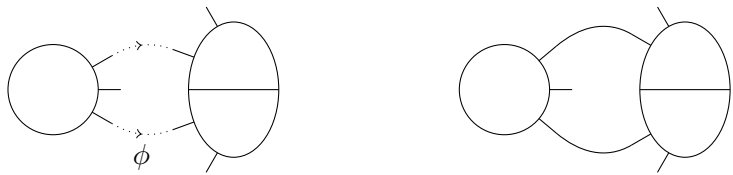

In the following sections we give examples of assembly maps.  In particular we show how all but one of the known nontrivial homology classes for $\Outn$ and $\Autn$ are obtained by assembling classes from the homology of $\Gamma_{1,s}$ and $\Gamma_{2,s}$.

\subsection{Morita's original series~\cite{Morita}}
\label{subsec:Morita}
Recall from Section~\ref{subsec:alpha_k}
that $H_{2k}(\Gamma_{1,2k+1})\iso \F$ with generator $\alpha_k$.
Fix a gluing pattern $\phi : \Gr_{1,2k+1}\cup \Gr_{1,2k+1}\to \Gr_{2k+2,0}$ which matches all of the leaves of the first graph with those of the second.  This gives an assembly map
$$
A_\phi\colon H_{2k}(\Gamma_{1,2k+1})\otimes H_{2k}(\Gamma_{1,2k+1})\longrightarrow H_{4k}\big(\Out(F_{2k+2})\big).
$$
The $k$th Morita class $\mu_k$ is the image under $A_\phi$ of $\alpha_k\otimes\alpha_k$.  Remark~\ref{rem:HairyMorita} explains why this viewpoint leads to the same classes as those originally  defined by Morita. The classes $\mu_1,\mu_2$, and $\mu_3$ are known to be nontrivial.

A lift of $\mu_k$ to $H_{4k}\big(\Aut(F_{2k+2})\big)$ can be obtained via assembly maps using the gluing pattern $\phi : \Gr_{1,2k+1}\cup\Gr_{0,3}\cup \Gr_{1,2k+1}\to \Gr_{2k+2,1}$
which matches one leaf of $\Gr_{0,3}$ with a leaf of one $\Gr_{1,2k+1}$, another leaf of $\Gr_{0,3}$ with a leaf of the other $\Gr_{1,2k+1}$, and then pairs the remaining leaves of the two copies of $\Gr_{1,2k+1}$ as before.
Let $\hat\mu_k $ be the image of $\alpha_k\otimes\iota\otimes \alpha_k$ under the resulting assembly map $H_{2k}(\Gamma_{1,2k+1})\otimes H_{0}(\Gamma_{0,3})\otimes H_{2k}(\Gamma_{1,2k+1})\to H_{4k}(\Gamma_{2k+2, 1})$, where $\iota$ is a generator of $H_0(\Gamma_{0,3})$.  The projection map $H_{4k}(\Gamma_{2k+2,1})\to H_{4k}(\Gamma_{2k+2,0})$ then sends $\hat\mu_k $ to $\mu_k$.

\begin{remark}
This argument shows more generally that every assembly map with target $H_i\big(\Outn\!\big)$ lifts to $H_i\big(\Autn \big)$.
\end{remark}

\begin{proposition}
The Morita class $\mu_k\in H_{4k}\big(\Out(F_{2k+2})\big)$ is supported in an abelian subgroup $\Z^{4k}$ of $\Out(F_{2k+2})$, and the analogous statement also holds for a lift to $\Aut(F_{2k+2})$.
\end{proposition}

\begin{proof}
As noted in Section~\ref{subsec:alpha_k}, the class $\alpha_k$ is the top-dimensional homology class of a subgroup $\Z^{2k}$ in $\Gamma_{1,2k+1}$.  The assembly that produces $\mu_k$ then gives a map $\Z^{4k}\to\Out(F_{2k+2})$
taking a generator of $H_{4k}(\Z^{4k})$ to $\mu_k$.
It is easy to see using the definition of $\Gamma_{n,s}$ as a group of homotopy equivalences that the image of the map $\Z^{4k}\to\Out(F_{2k+2})$ is generated by automorphisms $\lambda_{i1}$ and $\rho_{i2}$ for $3\leq i\leq 2k+2$, where $\lambda_{ij}$ is left multiplication of the
basis element $x_i$ by $x_j$, and $\rho_{ij}$ is right multiplication of $x_i$ by $x_j$, with all basis elements other than $x_i$ fixed in both cases.   From this description one can see that the map $\Z^{4k}\to \Out(F_{2k+2})$ is injective, as is its lift to $\Aut(F_{2k+2})$, so the two versions of $\mu_k$ for $\Out$ and $\Aut$ are supported on $\Z^{4k}$ subgroups.
\end{proof}

\subsection{Homology of $\Gamma_{2,4}$}
\label{subsec:GammaTwoFour}

If $\phi$ connects two copies of $\Gr_{1,3}$ by gluing just one pair of leaves as in Figure~\ref{fig:2X_13-X_34}, the result is a graph $\Gr_{\phi}$ of rank $2$ with $4$ leaves and an assembly map
$$
A_\phi\colon H_2(\Gamma_{1,3})\otimes H_2(\Gamma_{1,3})\longrightarrow H_4(\Gamma_{2,4}).
$$

\begin{figure}
\begin{center}
\begin{tikzpicture}[scale=.6]
\begin{tikzrankone}{0cm}{0cm}
\tikzhair{f1}{30}
\tikzhair{f2}{0}
\tikzhair{f3}{-30}
\end{tikzrankone}
\begin{tikzrankone}{4cm}{0cm}
\tikzhair{s1}{150}
\tikzhair{s2}{180}
\tikzhair{s3}{-150}
\end{tikzrankone}
\draw [dotted,->-=0.5] (f3) to[out=-30,in=-150] node[below]{$\phi$}  (s3);
\end{tikzpicture}
\end{center}
\caption{Gluing two copies of $\Gr_{1,3}$ }
\label{fig:2X_13-X_34}
\end{figure}
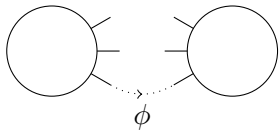

As before,  let $\alpha_1$ be a generator of $H_2(\Gamma_{1,3})\iso P_{(1^3)}$.  The image of $\alpha_1\otimes \alpha_1$ under $A_\phi$ is then the nonvanishing class $\alpha_{I,J}$ in Proposition~\ref{prop:alphaIJ} in the case $k=1$.
By Theorem~\ref{thm:ranktwo} we have $H_4(\Gamma_{2,4})=P_{(2,2)}$, which is $2$-dimensional.
Since $P_{(2,2)}$ is irreducible as an $\SG_4$-module, nontriviality of $A_\phi$ implies that the induced map
$$
\hat A_\phi\colon
\Res_{\SG_2}^{\SG_3} \big(H_2(\Gamma_{1,3})\big)
\circ
\Res_{\SG_2}^{\SG_3} \big(H_2(\Gamma_{1,3})\big)
\longrightarrow H_4(\Gamma_{2,4})
$$
is surjective.

\subsection{Gluing two leaves of a single rank  $1$ graph}
\label{subsec:selfgluing}
If $\phi$ glues two leaves of $\Gr_{1,s}$ together as in Figure~\ref{fig:X1s-self}, the result is a graph $\Gr_{\phi}$ of rank $2$ with $s-2$ leaves and an assembly map
$$
A_\phi\colon H_k(\Gamma_{1,s})\longrightarrow H_k(\Gamma_{2,s-2}).
$$
Let us show that this  $A_\phi$ is zero when $k>0$.

The map $A_\phi$ is an $\SG_{s-2}$-module map, where $\SG_{s-2}$ is the subgroup of $\SG_s$ which permutes the unglued leaves.
For $k>0$ either the domain or the range of $A_\phi$ is zero unless $k$ is a multiple of $4$, by Proposition~\ref{prop:rankone} and  Theorem~\ref{thm:ranktwo}. If $k=4\ell>0$, then $H_{4\ell}(\Gamma_{1,s})=P_{(s-4\ell,1^{4\ell})}$.  Restriction from $\SG_s$ to $\SG_{s-2}$ removes two boxes from the Young diagram for $P_{(s-4\ell,1^{4\ell})}$, so as an $\SG_{s-2}$-module the domain of $A_\phi$  is
$$
P_{(s-4\ell-2,1^{4\ell})} \oplus 2P_{(s-4\ell-1,1^{4\ell-1})} \oplus P_{(s-4\ell,1^{4\ell-2})}.
$$
(though if $s -4\ell<3$ some of the terms are not there).  On the other hand, by Theorem~\ref{thm:ranktwo} all partitions in $H_{4\ell}(\Gamma_{2,s-2})$ contain at least $2\ell$ boxes in the second column, so none of these modules appears in $H_{4\ell}(\Gamma_{2,s-2})$ and $A_\phi$ must be zero.

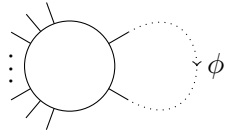
\begin{figure}
\begin{center}
\begin{tikzpicture}[scale=.6]
\begin{tikzrankone}{10cm}{0cm}
\tikzhair{r1}{30}
\tikzhair{r2}{-30}
\tikzhair{p1}{110}
\tikzhair{p2}{130}
\tikzhair{p3}{150}
\tikzhair{p4}{-150}
\tikzhair{p5}{-130}
\tikzhair{p6}{-110}
\end{tikzrankone}
\draw [dotted,->-=0.5] (r1) to[out=30, in=-30, looseness=4] node[right]{$\phi$}  (r2);
\path (p3) to node[sloped]{$\dots$}  (p4);
\end{tikzpicture}
\end{center}
\caption{A self-gluing of $\Gr_{1,s}$ }
\label{fig:X1s-self}
\end{figure}

\subsection{Rank $2$}
\label{subsec:GammaTwo2N}
For any positive $s_1$ and $s_2$ we can join   $\Gr_{1,s_1}$ to $\Gr_{1,s_2}$ by connecting one pair of leaves.  We obtain a graph $X_\phi$ of rank $2$ with $s_1+s_2-2$ leaves  and  assembly maps
$$
A_\phi\colon H_{2k_1}(\Gamma_{1,s_1})\otimes H_{2k_2}(\Gamma_{1,s_2})\longrightarrow H_{2k_1+2k_2}(\Gamma_{2,s_1+s_2-2}).
$$
If $k_1+k_2$ is odd then $H_{2k_1+2k_2}(\Gamma_{2,s_1+s_2-2})=0$,  so the map is obviously trivial.  In fact this map is trivial unless $k_1=k_2$, in which case it is nontrivial. This follows immediately from Proposition~\ref{prop:alphaIJ} and Remark~\ref{remark:alpha-unequalsize}.

In the special case $s_i=2k_i+1$  and $k_1\neq k_2$ there is an alternative argument for proving the assembly map is zero using representation theory.
From the discussion in the beginning of the section the assembly map induces a map
$$
\Res_{\SG_{2k_1}}^{\SG_{2k_1+1}} \big(H_{2k_1}(\Gamma_{1,2k_1+1})\big)
\circ
\Res_{\SG_{2k_2}}^{\SG_{2k_2+1}} \big(H_{2k_2}(\Gamma_{1,2k_2+1})\big)
\longrightarrow
H_{2k_1+2k_2}(\Gamma_{2,2k_1+2k_2}).
$$
By Proposition~\ref{prop:rankone} and Theorem~\ref{thm:ranktwo} we have
$H_{2k_i}(\Gamma_{1,2k_i+1})=P_{(1^{2k_i+1})}$ and
$H_{2k_1+2k_2}(\Gamma_{2,2k_1+2k_2})=P_{(2^{k_1+k_2})}$, so that the induced map is
$
P_{(1^{2k_1})} \circ P_{(1^{2k_2})} \to P_{(2^{k_1+k_2})}
$.
Since $k_1\neq k_2$ there is no way to add $2k_2$ boxes to distinct rows in the Young diagram for
 $P_{(1^{2k_1})}$ to obtain the diagram for $P_{(2^{k_1+k_2})}$, which means that the decomposition of
$P_{(1^{2k_1})} \circ P_{(1^{2k_2})}$ does not contain $P_{(2^{k_1+k_2})}$.
So the map $\hat A_\phi$ (and hence the assembly map $A_\phi$) must be zero.
One can use a similar argument  when $s_i \neq 2k_i+1$ but this requires the full Littlewood-Richardson rule instead of the much easier
Pieri rule.

\subsection{Generalized Morita Classes}
\label{subsec:genMorita}

In~\cite{CVMorita}, Morita's original series was generalized, and it is not hard to describe the generalization in terms of assembly maps arising from gluing together graphs of rank $0$ and rank $1$.  Suppose we are given a finite connected graph $G$
with no valence $2$ vertices, along with a partition of its non-leaf vertices into two subsets $V_0$ and $V_1$ such that all vertices in $V_1$ have odd valence.  Take a copy $X_v$ of $X_{1,2k+1}$ for each vertex  $v$ in $V_1$ of valence $2k+1$, and identify the leaves of $X_v$ with the edges of $G$ incident to $v$.  Similarly for each vertex in $V_0$ of valence $k$ take a copy of $X_{0,k}$.  The graph $G$ then gives gluing instructions for constructing a graph $X_{n,s}$ and a corresponding assembly map.  (The $s$ leaves of $X_{n,s}$ come from the valence $1$ vertices of $G$.)  By assembling copies of $\alpha_k\in H_{2k}(\Gamma_{1,2k+1})$ at the valence $2k+1$ vertices in $V_1$, along with standard generators for $H_0(\Gamma_{0,k})$ at the valence $k$ vertices in $V_0$, we obtain a \emph{generalized Morita class\/}  $\mu_G\in H_i(\Gamma_{n,s})$. The original Morita classes are the case that $G$ has two vertices, both in $V_1$, with all edges going from one vertex to the other.

Vertices in $V_1$ are called \emph{rank one\/} vertices, and vertices in $V_0$ are \emph{rank zero\/} vertices. There is no loss of generality in assuming that all rank zero vertices are isolated, in the sense that no edge of $G$ connects two different rank zero vertices, since such edges can be collapsed one by one without affecting the class $\mu_G$.

If the graph $G$ has certain ``orientation-reversing" symmetries, for example if it has an edge with both ends at the same vertex in $V_1$, then the class $\mu_G$ is automatically zero; this is spelled out in detail in~\cite{CVMorita} in the case $s=0$, which is the case considered there.
The following result shows that $\mu_G$ vanishes in many other cases as well.

\begin{theorem}
\label{thm:GenMorita}
If  $G$ has two rank one vertices of different valence then $\mu_G =0$.
\end{theorem}

\begin{proof}
The graph $G$ contains a path connecting two rank one vertices of different valence and passing only through rank zero vertices.
Using gluing instructions given by this linear subgraph we obtain a graph $X_\phi=X_{2,s}$ for some $s>0$.  We can also obtain an $X_{2,s}$ by first gluing one of the rank 1 graphs and all of the rank 0 graphs to obtain a graph $X_{1,s_1}$, then gluing $X_{1,s_1}$ to the other rank 1 graph $X_{1,s_2}$. Call the latter gluing $\psi$.  The assembly map $A_\phi$ factors through $A_\psi$, and we showed in Section~\ref{subsec:GammaTwo2N} that $A_\psi$ is zero since the two classes assembled by it have different homology degrees, by the hypothesis that the original two rank one vertices have different valence.
\end{proof}

\subsection{Eisenstein classes}
\label{subsec:eisenstein}
Consider a gluing $\phi$  that attaches leaves of $\Gr_{2,2k+2}$ to leaves of $\Gr_{1,2k+1}$, leaving one leaf of $\Gr_{2,2k+2}$  unpaired.  This gluing determines a map
$\Gamma_{2,2k+2}\times\Gamma_{1,2k+1}\to \Gamma_{2k+3,1}=\Aut(F_{2k+3})$
and an assembly map
$$
A_\phi:H_{2k+3}(\Gamma_{2,2k+2})\otimes H_{2k}(\Gamma_{1,2k+1})\longrightarrow H_{4k+3}\big(\Aut(F_{2k+3})\big).
$$
In Section~\ref{subsec:Rank2odd} we constructed classes $m_{I,J} \in H_{2k+3}(\Gamma_{2,2k+2})$ corresponding to a partition of $\{1,\dots,2k+2\}$ into sets $I$ and $J$ with $|I|=2$ and $|J|=2k$. (Note that the parameter $k$ now corresponds to $k-1$ in Section~\ref{subsec:Rank2odd}.) Choose $I$ to consist of the unglued leaf of $X_{2,2k+2}$ and one other leaf, with $J$ the remaining leaves.  Then the  {\em Eisenstein class} $\mathcal{E}_k$ is the image of $m_{I,J} \otimes \alpha_k$ under $A_\phi$.  This does not depend on the choice of the other leaf in $I$ or the ordering of the leaves in $J$ since permutations of the glued leaves in $X_{2,2k+2}$ become inner automorphisms of $\Aut(F_{2k+3})$ which therefore induce the identity on homology.

We choose $I$ to contain the unglued leaf because if we did not, then $\mathcal{E}_k$ would automatically be zero by the following symmetry argument. If both leaves in $X_{2,2k+2}$ indexed by $I$ were glued to leaves of $X_{1,2k+1}$ then the transposition switching these two leaves would extend to a transposition of two edges of the glued-together graph $X_{2k+3,1}$. This transposition sends $m_{I,J}\otimes\alpha_k$ to its negative since it preserves $m_{I,J}$ as we noted in Section~\ref{subsec:Rank2odd} and it sends $\alpha_k$ to its negative.  On the other hand, after gluing, the transposition gives an inner automorphism of $\Aut(F_{2k+3})$ inducing the identity on homology.  Choosing $I$ to contain the unglued leaf has the effect of breaking this symmetry, so $\mathcal{E}_k$ does not vanish for any obvious reason.  (Permutations of the leaves in $J$ act trivially on $\mathcal{E}_k$ since they act by their sign on both $m_{I,J}$ and $\alpha_k$.)

An alternative construction would be to use the class $m_i$ defined after the proof of Proposition~\ref{prop:oddtwo} instead of $m_{I,J}$.  Permutations of the glued leaves change $m_i$ by the sign of the permutation, and the same is true for $\alpha_k$, so $m_i\otimes \alpha_k$ is invariant under these permutations. From the definition of $m_i$ it follows that using $m_i$ instead of $m_{I,J}$ changes $\mathcal{E}_k$ only by a nonzero scalar multiple.

\begin{remark}\label{rmk:eisenstein} The Eisenstein class $\mathcal{E}_k$ maps to zero in $H_{4k+3}\big(\Out(F_{2k+3})\big)$ since the map $\Gamma_{2k+3,1}\to\Gamma_{2k+3,0}$ is induced by forgetting the leaf of $\Gr_{2k+3,1}$, and this leaf could just as well be omitted from $\Gr_{2,2k+2}$ before the gluing, but this puts the class $m_{i}$ in a dimension above the $\vcd$ of $\Gamma_{2,2k+1}$. This argument applies more generally whenever one has an assembly map with target $H_i(\Gamma_{n,1})$ and a source factor $H_{\vcd}(\Gamma_{n_j,s_j})$ whose graph $\Gr_{n_j,s_j}$ is the one with the unglued leaf.
\end{remark}

\subsection{Odd-dimensional classes in $H_*\big(\Out(F_n)\big)$.}
\label{subsec:odd}
 The Euler characteristic calculations for $n\leq 11$ imply that there must exist odd-dimensional classes in  $H_*\big(\Out(F_n)\big)$, probably in great abundance as $n$ increases.  However, only one such class has been found so  it becomes an  interesting challenge to find  nontrivial odd-dimensional classes in a systematic way. 

A sequence of candidates for such classes was introduced by Morita, Sakasai, and Suzuki in~\cite{MSS1}, Proposition~6.3.  These are classes $\gamma_k\in H_{4k+7}\big(\Out(F_{2k+6})\big)$ for $k\geq 1$. The class $\gamma_k$ can be interpreted as gluing $\Gr_{1,2k+3}\cup \Gr_{1,2k+1}\cup \Gr_{2,4}\to \Gr_{2k+6,0}$ (illustrated for $k=1$ and $2$ in Figure~\ref{fig:MSS}) and considering the image of the class $\alpha_{k+1}\otimes \alpha_{k}\otimes m_i$, where  $m_i$ is the class defined after the proof of Proposition~\ref{prop:oddtwo}, with $i$ labeling the leaf of $\Gr_{2,4}$ attached to $\Gr_{1,2k+1}$.
However, since the image of $\alpha_{k+1}\otimes \alpha_k$ in $H_{4k+2}(\Gamma_{2,4k-2})$ is trivial by Section~\ref{subsec:GammaTwo2N}, these classes must also be zero by associativity of the assembly map.

\begin{figure}
\begin{center}
\begin{tikzpicture}[scale=.6]
\begin{tikzrankone}{0cm}{0cm}
\tikzhair{sl3}{25}
\tikzhair{sl4}{10}
\tikzhair{sl5}{-80}
\end{tikzrankone}
\begin{tikzrankone}{8cm}{0cm}
\tikzhair{sr3}{155}
\tikzhair{sr4}{170}
\tikzhair{sr5}{-120}
\tikzhair{sr6}{-105}
\tikzhair{sr7}{-90}
\end{tikzrankone}
\begin{tikzranktwo}{4cm}{-3cm}
\tikzhair{sm1}{170}
\tikzhair{sm2}{30}
\tikzhair{sm3}{10}
\tikzhair{sm4}{-10}
\end{tikzranktwo}
\draw [dotted] (sl3) to [out=25, in=155] (sr3);
\draw [dotted] (sl4) to [out=10, in=170] (sr4);
\draw [dotted] (sl5) to [out=-80, in=170] (sm1);
\draw [dotted] (sr5) to [out=-120, in=30] (sm2);
\draw [dotted] (sr6) to [out=-105, in=10] (sm3);
\draw [dotted] (sr7) to [out=-90, in=-10] (sm4);
\begin{scope}[xshift=13cm]
\begin{tikzrankone}{0cm}{0cm}
\tikzhair{l1}{55}
\tikzhair{l2}{40}
\tikzhair{l3}{25}
\tikzhair{l4}{10}
\tikzhair{l5}{-80}
\end{tikzrankone}
\begin{tikzrankone}{8cm}{0cm}
\tikzhair{r1}{125}
\tikzhair{r2}{140}
\tikzhair{r3}{155}
\tikzhair{r4}{170}
\tikzhair{r5}{-120}
\tikzhair{r6}{-105}
\tikzhair{r7}{-90}
\end{tikzrankone}
\begin{tikzranktwo}{4cm}{-3cm}
\tikzhair{m1}{170}
\tikzhair{m2}{30}
\tikzhair{m3}{10}
\tikzhair{m4}{-10}
\end{tikzranktwo}
\draw [dotted] (l1) to [out=55, in=125] (r1);
\draw [dotted] (l2) to [out=40, in=140] (r2);
\draw [dotted] (l3) to [out=25, in=155] (r3);
\draw [dotted] (l4) to [out=10, in=170] (r4);
\draw [dotted] (l5) to [out=-80, in=170] (m1);
\draw [dotted] (r5) to [out=-120, in=30] (m2);
\draw [dotted] (r6) to [out=-105, in=10] (m3);
\draw [dotted] (r7) to [out=-90, in=-10] (m4);
\end{scope}
 \end{tikzpicture}
 \caption{On the left: Assembling $\gamma_1\in H_{11}\big(\Out(F_{8})\big)$ from $H_{2}(\Gamma_{1,3})$, $H_4(\Gamma_{1,5})$, and $H_5(\Gamma_{2,4})$. On the right: Assembling $\gamma_2\in H_{15}\big(\Out(F_{10})\big)$
from $H_{4}(\Gamma_{1,5})$, $H_6(\Gamma_{1,7})$, and $H_5(\Gamma_{2,4})$ }
\label{fig:MSS}
\end{center}
\end{figure}
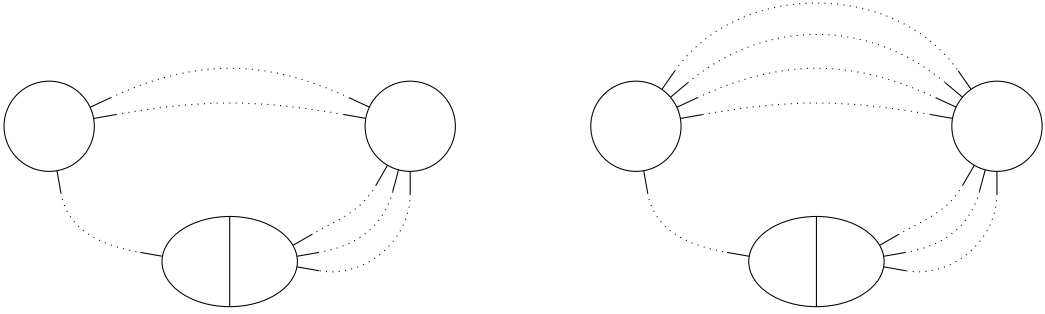

The class $\gamma_1\in H_{11}\big(\Out(F_8)\big)$ was introduced to account for the fact that $\mu_3\in H_{12}\big(\Out(F_8)\big)$ is nonzero while the Euler characteristic of $\Out(F_8)$ is $1$, so an odd-dimensional class must exist. An alternative candidate class in $H_{11}\big(\Out(F_8)\big)$ can be constructed by assembling two copies of $\alpha_1\in H_2(\Gamma_{1,3})$ with a class in $H_7(\Gamma_{2,6})$ by gluing all the leaves of two copies of $\Gr_{1,3}$ to one copy of $\Gr_{2,6}$ as shown in the left half of Figure~\ref{fig:candidate}.  Another possibility is to glue all the leaves of $\Gr_{1,1}$ and $\Gr_{1,5}$ to the leaves of $\Gr_{2,6}$, obtaining an assembly map $H_{0}(\Gamma_{1,1})\otimes H_{7}(\Gamma_{2,6})\otimes H_{4}(\Gamma_{1,5})\to H_{11}(\Out(F_8)\big)$ as in the right half of Figure~\ref{fig:candidate}.  However, this assembly map may well be zero since it produces classes that lift to classes in $H_{11}\big(\Aut(F_8)\big)$ that are in the image of the stabilization $H_{11}\big(\Aut(F_7)\big)\to H_{11}\big(\Aut(F_8)\big)$ (see Section~\ref{subsec:stab_n}) so if they were nonzero they would give counterexamples to Conjecture~\ref{conj:D}.

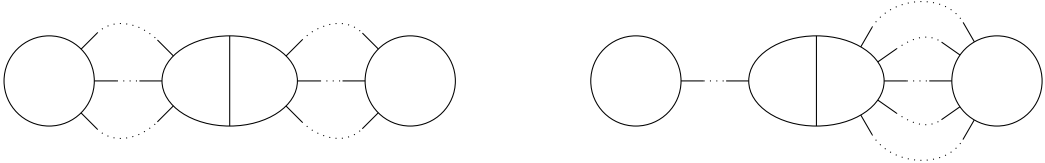
\begin{figure}[b]
\begin{center}
\begin{tikzpicture}[scale=.6]
\begin{tikzrankone}{0cm}{0cm}
\tikzhair{tl1}{45}
\tikzhair{tl2}{0}
\tikzhair{tl3}{-45}
\end{tikzrankone}
\begin{tikzranktwo}{4cm}{0cm}
\tikzhair{tm1}{135}
\tikzhair{tm2}{180}
\tikzhair{tm3}{-135}
\tikzhair{tm4}{45}
\tikzhair{tm5}{0}
\tikzhair{tm6}{-45}
\end{tikzranktwo}
\begin{tikzrankone}{8cm}{0cm}
\tikzhair{tr1}{135}
\tikzhair{tr2}{180}
\tikzhair{tr3}{-135}
\end{tikzrankone}
\draw [dotted] (tl1) to  [out=45, in=135] (tm1);
\draw [dotted] (tl2) to  [out=0, in=180] (tm2);
\draw [dotted] (tl3) to  [out=-45, in=-135] (tm3);
\draw [dotted] (tm4) to  [out=45, in=135] (tr1);
\draw [dotted] (tm5) to  [out=0, in=180] (tr2);
\draw [dotted] (tm6) to  [out=-45, in=-135] (tr3);
\begin{scope}[xshift=13cm]
\begin{tikzrankone}{0cm}{0cm}
\tikzhair{bl1}{0}
\end{tikzrankone}
\begin{tikzranktwo}{4cm}{0cm}
\tikzhair{bm1}{180}
\tikzhair{bm2}{60}
\tikzhair{bm3}{35}
\tikzhair{bm4}{0}
\tikzhair{bm5}{-35}
\tikzhair{bm6}{-60}
\end{tikzranktwo}
\begin{tikzrankone}{8cm}{0cm}
\tikzhair{br1}{120}
\tikzhair{br2}{145}
\tikzhair{br3}{180}
\tikzhair{br4}{-145}
\tikzhair{br5}{-120}
\end{tikzrankone}
\draw [dotted] (bl1) to  [out=0, in=180] (bm1);
\draw [dotted] (bm2) to  [out=60, in=120] (br1);
\draw [dotted] (bm3) to  [out=35, in=145] (br2);
\draw [dotted] (bm4) to  [out=0, in=180] (br3);
\draw [dotted] (bm5) to  [out=-35, in=-145] (br4);
\draw [dotted] (bm6) to  [out=-60, in=-120] (br5);
\end{scope}
\end{tikzpicture}
\caption{Candidates for a nontrivial class in $H_{11}\big(\Out(F_{8})\big)$}
\label{fig:candidate}
\end{center}
\end{figure}

\begin{figure}
\begin{center}
\begin{tikzpicture}[scale=.6]
\begin{tikzrankone}{0cm}{0cm}
\tikzhair{tl1}{45}
\tikzhair{tl2}{0}
\tikzhair{tl3}{-45}
\end{tikzrankone}
\begin{tikzranktwo}{4cm}{0cm}
\tikzhair{tm1}{135}
\tikzhair{tm2}{180}
\tikzhair{tm3}{-135}
\tikzhair{tm4}{60}
\tikzhair{tm5}{35}
\tikzhair{tm6}{0}
\tikzhair{tm7}{-35}
\tikzhair{tm8}{-60}
\end{tikzranktwo}
\begin{tikzrankone}{8cm}{0cm}
\tikzhair{tr1}{120}
\tikzhair{tr2}{145}
\tikzhair{tr3}{180}
\tikzhair{tr4}{-145}
\tikzhair{tr5}{-120}
\end{tikzrankone}
\draw [dotted] (tl1) to  [out=45, in=135] (tm1);
\draw [dotted] (tl2) to  [out=0, in=180] (tm2);
\draw [dotted] (tl3) to  [out=-45, in=-135] (tm3);
\draw [dotted] (tm4) to  [out=60, in=120] (tr1);
\draw [dotted] (tm5) to  [out=35, in=145] (tr2);
\draw [dotted] (tm6) to  [out=0, in=185] (tr3);
\draw [dotted] (tm7) to  [out=-35, in=-145] (tr4);
\draw [dotted] (tm8) to  [out=-60, in=-120] (tr5);
\begin{scope}[xshift=13cm]
\begin{tikzrankone}{0cm}{0cm}
\tikzhair{bl1}{60}
\tikzhair{bl2}{40}
\tikzhair{bl3}{20}
\tikzhair{bl4}{0}
\tikzhair{bl5}{-20}
\tikzhair{bl6}{-40}
\tikzhair{bl7}{-60}
\end{tikzrankone}
\begin{tikzranktwo}{4cm}{0cm}
\tikzhair{bm1}{0}
\tikzhair{bm2}{120}
\tikzhair{bm3}{140}
\tikzhair{bm4}{160}
\tikzhair{bm5}{180}
\tikzhair{bm6}{-160}
\tikzhair{bm7}{-140}
\tikzhair{bm8}{-120}
\end{tikzranktwo}
\begin{tikzrankone}{8cm}{0cm}
\tikzhair{br1}{180}
\end{tikzrankone}
\draw [dotted] (bm1) to  [out=0, in=180] (br1);
\draw [dotted] (bl1) to  [out=60, in=120] (bm2);
\draw [dotted] (bl2) to  [out=40, in=140] (bm3);
\draw [dotted] (bl3) to  [out=20, in=160] (bm4);
\draw [dotted] (bl4) to  [out=0, in=180] (bm5);
\draw [dotted] (bl5) to  [out=-20, in=-160] (bm6);
\draw [dotted] (bl6) to  [out=-40, in=-140] (bm7);
\draw [dotted] (bl7) to  [out=-60, in=-120] (bm8);
\end{scope}
%
\begin{scope}[yshift=-4cm]
\begin{tikzrankone}{1cm}{0cm}
\tikzhair{tf1}{20}
\tikzhair{tf2}{0}
\tikzhair{tf3}{-60}
\tikzhair{tf4}{-80}
\tikzhair{tf5}{-100}
\end{tikzrankone}
\begin{tikzrankone}{7cm}{0cm}
\tikzhair{ts1}{160}
\tikzhair{ts2}{180}
\tikzhair{ts3}{-120}
\tikzhair{ts4}{-100}
\tikzhair{ts5}{-80}
\end{tikzrankone}
\begin{tikzranktwo}{4cm}{-2.7cm}
\tikzhair{tt1}{140}
\tikzhair{tt2}{160}
\tikzhair{tt3}{180}
\tikzhair{tt4}{40}
\tikzhair{tt5}{20}
\tikzhair{tt6}{0}
\end{tikzranktwo}
\draw [dotted] (tf1) to  [out=20, in=160] (ts1);
\draw [dotted] (tf2) to  [out=0, in=180] (ts2);
\draw [dotted] (tf3) to  [out=-60, in=140] (tt1);
\draw [dotted] (tf4) to  [out=-80, in=160] (tt2);
\draw [dotted] (tf5) to  [out=-100, in=180] (tt3);
\draw [dotted] (ts3) to  [out=-120, in=40] (tt4);
\draw [dotted] (ts4) to  [out=-100, in=20] (tt5);
\draw [dotted] (ts5) to  [out=-80, in=0] (tt6);
\begin{scope}[xshift=13.5cm]
\begin{tikzrankzero}{4.5cm}{1.5cm}
\tikzhair{ff1}{180}
\tikzhair{ff2}{20}
\tikzhair{ff3}{-15}
\end{tikzrankzero}
\begin{tikzrankone}{1cm}{-2.2cm}
\tikzhair{fs1}{100}
\tikzhair{fs2}{30}
\tikzhair{fs3}{15}
\tikzhair{fs4}{0}
\tikzhair{fs5}{-15}
\tikzhair{fs6}{-30}
\tikzhair{fs7}{-45}
\end{tikzrankone}
\begin{tikzranktwo}{7cm}{-2.2cm}
\tikzhair{ft1}{70}
\tikzhair{ft2}{85}
\tikzhair{ft3}{150}
\tikzhair{ft4}{165}
\tikzhair{ft5}{180}
\tikzhair{ft6}{-165}
\tikzhair{ft7}{-150}
\tikzhair{ft8}{-135}
\end{tikzranktwo}
\draw [dotted] (ff1) to  [out=180, in=100] (fs1);
\draw [dotted] (ff2) to  [out=20, in=70] (ft1);
\draw [dotted] (ff3) to  [out=-15, in=85] (ft2);
\draw [dotted] (fs2) to  [out=30, in=150] (ft3);
\draw [dotted] (fs3) to  [out=15, in=165] (ft4);
\draw [dotted] (fs4) to  [out=0, in=180] (ft5);
\draw [dotted] (fs5) to  [out=-15, in=-165] (ft6);
\draw [dotted] (fs6) to  [out=-30, in=-150] (ft7);
\draw [dotted] (fs7) to  [out=-45, in=-135] (ft8);
\end{scope}
\end{scope}
\end{tikzpicture}
\caption{Candidates for a nontrivial class in $H_{15}\big(\Out(F_{10})\big)$}
\label{fig:candidate2}
\end{center}
\end{figure}
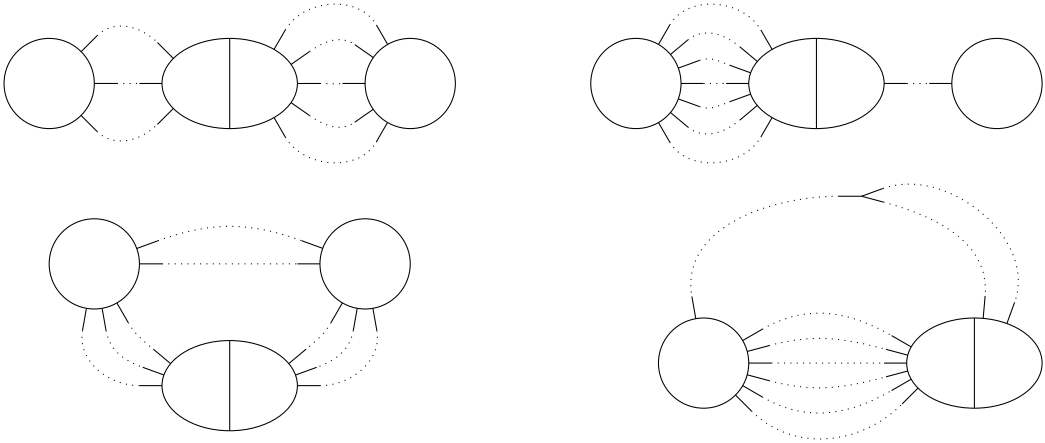

A few other ways to construct candidates for odd-dimensional classes are shown in Figure~\ref{fig:candidate2}, this time for classes in $H_{15}\big(\Out(F_{10})\big)$.

There are many other ways to construct candidates for odd-dimensional homology classes.
If we glue $\Gr_{1,s}$ to $\Gr_{2,s}$ by joining all of their leaves, we obtain a rank $s+2$ graph with no leaves and an associated assembly map
$$
H_2(\Gamma_{1,s}) \otimes H_5(\Gamma_{2,s}) \longrightarrow H_7(\Gamma_{s+2,0}).
$$

For $s \geq 4$ we have  $H_2(\Gamma_{1,s}) = P_{(s-2,1^2)}$. The decomposition of
$H_5(\Gamma_{2,s})=W_4\circ P_{(s-4)}=P_{(2,1^2)}\circ P_{(s-4)}$ has several terms but one ofxhem is $P_{(s-2,1^2)}$, with multiplicity one. Thus the space of coinvariants
 $
\big( H_2(\Gamma_{1,s}) \otimes H_5(\Gamma_{2,s}) \big)_{\SG_s}
$
is $1$-dimensional and the assembly map, which factors through these coinvariants, may well be nontrivial.
This construction produces potential classes in $H_7\big(\Out(F_n)\big)$ for all $n \geq 6$, although by homology stability these classes must be trivial for $n\geq 10$. According to the calculations in~\cite{Ohashi} and~\cite{Bartholdi} the classes for $n=6,7$ also vanish.

We can make a similar construction with $H_{2k}(\Gamma_{1,s})$ and $H_{2k+3}(\Gamma_{2,s})$ as long as $s\geq 2k+2$ using the partition $(s-2k,1^{2k})$. We have $H_{2k}(\Gamma_{1,s})= P_{(s-2k,1^{2k})}$ and $H_{2k+3}(\Gamma_{2,s})= W_{2k+2}\circ P_{(s-2k-2)}$.  The second term of $W_{2k+2}$ is $M_{2k+2}\otimes P_{(2,1^2k)}$, and one term of  $P_{(2,1^2k)} \circ P_{(s-2k-2)}$ is equal to $P_{(s-2k,1^{2k})}$.  Thus we obtain a map from $\mathcal{M}_{2k+2}$ to $H_{4k+3}\big(\Out(F_{s+2})\big)$.
For large $s$ this map must be trivial since the target group is trivial.
In fact we suspect that the map is trivial for $s\geq 2k+4$ but not for
$s=2k+3$.  For $s=2k+3$ the image lies in $H_{4k+3}\big(\Out(F_{2k+5})\big)$.
Since the $\vcd$ of $\Out(F_{2k+5})$ is $4k+7$,  these classes lie in codimension 4.

If $s\geq 2k+1$ there is another similar construction with $H_{2k}(\Gamma_{1,s})$ and $H_{2k+1}(\Gamma_{2,s})$,  again using the partition $(s-2k,1^{2k})$.  Here $H_{2k+1}(\Gamma_{2,s})$ contains the summand $S_{2k+2}\otimes P_{(1^{2k})} \circ P_{s-2k}$.  Since $P_{(1^{2k})} \circ P_{(s-2k)}$ contains a copy of $P_{(s-2k,1^{2k})}$ we get a map from $\mathcal{S}_{2k+2}$ to $H_{4k+1}\big(\Out(F_{s+2})\big)$. The first potentially nontrivial class occurs when $k = 5$,
in $H_{21}\big(\Out(F_{13})\big)$;
this class is in $H_{\vcd-2}$.

\subsection{Classes from pairs of cusp forms}
\label{subsec:classes}
Consider a gluing
$$
\Gr_{2,2m}\cup \Gr_{2,2m}\longrightarrow \Gr_{2m+3,0}
$$
matching the leaves of the first graph with those of the second, and the associated assembly map
$$
H_{2m+1}(\Gamma_{2,2m})\otimes H_{2m+1}(\Gamma_{2,2m}) \longrightarrow H_{4m+2}\big(\Out(F_{2m+3})\big).
$$
There is a  $\Z_2$ action on $H_{2m+1}(\Gamma_{2,2m})\otimes H_{2m+1}(\Gamma_{2,2m})$ which switches the factors.  By basic properties of the cross product (which gives the K\"unneth isomorphism) we have
$$
a\otimes b = (-1)^{(2m+1)(2m+1)}b\otimes a = -b\otimes a,
$$
so this assembly map factors through  the exterior product $\ext^2 H_{2m+1}(\Gamma_{2,2m})$
in addition to factoring through the $\SG_{2m}$-coinvariants.

Applying Lemma~\ref{thm:GL} we now have a map
$$
\left(\ext^2 H_{2m+1}(\Gamma_{2,2m})\right)_{\SG_{2m}}
=
\left(\ext^2
    \big(
        \bigoplus_{\scriptscriptstyle 0\leq i<m}\mathcal X_{2m,i}\otimes P_{(2^i,1^{2m-2i})}
    \big)
\right)_{\SG_{2m}}
\longrightarrow
H_{4m+2}\big(\Out(F_{2m+3})\big),
$$
where $\mathcal X_{2m,i}$ is the space of either cusp forms (if $i$ is even) or all modular forms (if $i$ is odd) of weight $2m{+}2{-}2i$.
Since  the $P_{(2^i,1^{2m-2i})}$ are pairwise non-isomorphic this  gives   a map
$$
\bigoplus_{0\leq i <m}
    \left(
        {\textstyle \ext^2} \mathcal{X}_{2m,i}
    \right)
\longrightarrow
H_{4m+2}\big(\Out(F_{2m+3})\big).
$$
The term corresponding to $i=0$ was first mentioned in~\cite{CKV1}.

\section{Stabilization}
\label{sec:vanish}

In this section we consider the two ways to stabilize $\Gamma_{n,s}$ by letting one of the parameters $n$ and $s$ increase while keeping the other fixed.  Both stabilizations can be viewed as special cases of assembly maps. 

\subsection{Stabilization with respect to $s$.}
\label{subsec:stab_s}

For $s\geq 1$ the map gluing $\Gr_{0,3}$ to $\Gr_{n,s}$  by  a single edge  simply increases the number of leaves, and the associated assembly map
$$
H_0(\Gamma_{0,3})\otimes H_i(\Gamma_{n,s})\longrightarrow H_i(\Gamma_{n,s+1})
$$
gives a stabilization map $H_i(\Gamma_{n,s})\to H_i(\Gamma_{n,s+1})$ sending a class $\alpha$ to $\iota \otimes\alpha$ where $\iota$ is the standard generator of $H_0(\Gamma_{0,3})$. This is the same as the map induced by the splitting of the natural projection $\Gamma_{n,s+1}\to\Gamma_{n,s}$ defined in the proof of Proposition~\ref{thm:extension}. Strictly speaking, there are $s$ different stabilization maps depending on which leaf of $\Gr_{n,s}$ we attach $\Gr_{0,3}$ to, although these stabilizations differ only by the action of $\SG_s$.  The stabilization maps are always injective since they are induced by splittings of the groups.
By~\cite{HV} the stabilization map is surjective if $n\geq 2i+2$, but in this case the homology groups are trivial by Galatius' theorem.

Switching from homology to cohomology, Proposition~\ref{prop:rankone} and Theorem~\ref{thm:ranktwo}  demonstrate directly that for fixed $i$ the groups $H^i(\Gamma_{1,s})$ and $H^i(\Gamma_{2,s})$ satisfy {\em representation stability\/} as $s$ increases, i.e., for large enough $s$ the partitions which appear in their irreducible decompositions as  $\SG_s$-modules differ only by the number of boxes in the first row.
This leads one to suspect that the cohomology of $\Gamma_{n,s}$ satisfies representation stability for all $n$.  This is indeed the case and can be deduced easily from a theorem of Jim\'enez Rolland~\cite{JR} giving the corresponding result for mapping class groups of certain manifolds with punctures.

\begin{proposition}
\label{prop:repstab}
For fixed $i$ and $n$ the groups $H^i(\Gamma_{n,s})$ satisfy representation stability as $s$ increases.
\end{proposition}

\begin{proof}
As described in Section~\ref{subsect:Gns} we can view $\Gamma_{n,s}$ as the quotient of the mapping class group of the $3$-manifold $M_{n,s}$ by the subgroup generated by Dehn twists along $2$-spheres. This subgroup is normal and is just a direct product of finitely many cyclic groups of order $2$.  In particular it is a finite group so the projection from the mapping class group to $\Gamma_{n,s}$ induces an isomorphism on cohomology with coefficients in $\F$ by the Leray-Serre spectral sequence. Thus it suffices to prove representation stability for the mapping class group, and this was done in~\cite{JR}, with a specific stable range $s\geq 3i$.  To apply~\cite{JR} one uses the manifold $M=M_{n,0}$ as the base manifold, and one needs to check that this satisfies certain hypotheses: (1) $\pi_1M=
F_n$ is of type $FP_\infty$ and has trivial center, which is obviously true; and (2) the mapping class group of $M$ is of type $FP_\infty$, which follows from $\Outn$ being $FP_\infty$ and the kernel of the map from the mapping class group to $\Outn$ being finite abelian and hence of type $FP_\infty$.
\end{proof}

We remark that A. Saied has recently shown that $H^i(\Gamma_{n,s})$ satisfies representation stability with respect to $s$ whenever $s\geq n+i$ \cite{Saied}.  If $n<2i$ this is an obvious improvement on the stable range mentioned in the above proposition, and if $n\geq 2i$ the cohomology is zero for all $s$.  

\subsection{Stabilization with respect to $n$.}
\label{subsec:stab_n}

A stabilization map $H_i(\Gamma_{n,s})\to H_i(\Gamma_{n+1,s})$ can be obtained in a similar way by gluing $\Gr_{1,2}$ to $\Gr_{n,s}$ along one edge.  Here sufficiently many iterations take one to the stable range where the homology groups are zero, so the interest is in what happens unstably.  We can describe completely what happens when rank one classes are stabilized to rank two:

\begin{proposition}
\label{prop:OneStabilization}
If a class in $H_i(\Gamma_{n,s})$ is obtained from an assembly map with a factor group $H_j(\Gamma_{1,k})$, $j>0$, then this class maps to zero under the stabilization $H_i(\Gamma_{n,s})\to H_i(\Gamma_{n+1,s})$ obtained by gluing $\Gr_{1,2}$ to the corresponding factor graph $\Gr_{1,k}$.
\end{proposition}

An immediate consequence is the following result, first proved in~\cite{Stable} by combinatorial arguments.

\begin{corollary}  \label{cor:OneStabilization}
The Morita class $\mu_k$, lifted from $\Out$ to $\Aut$, vanishes under the stabilization map $H_{4k}(\Gamma_{2k+2,1})\to H_{4k}(\Gamma_{2k+3,1})$.
\end{corollary}

\begin{proof}[Proof of Proposition~\ref{prop:OneStabilization}]
It suffices to prove that the stabilization $H_i(\Gamma_{1,s})\to H_i(\Gamma_{2,s})$ is trivial for $i>0$.
The extension of   $A_\phi\colon H_i(\Gamma_{1,s}) \otimes H_0(\Gamma_{1,2}) \to H_i(\Gamma_{2,s})$ to the $\SG_s$-module map
$$
\hat A_\phi\colon
\Res^{\SG_s}_{\SG_{s-1}}\big(H_i(\Gamma_{1,s})\big)
\circ
\Res^{\SG_2}_{\SG_{1}}\big(H_0(\Gamma_{1,2})\big)
\longrightarrow H_i(\Gamma_{2,s})
$$
is accomplished by first restricting the $\SG_s$-action on $\Gamma_{1,s}$ to $\SG_{s-1}$ and then inducing it back up to~$\SG_s$.
In terms of Young diagrams, restriction to $\SG_{s-1}$ is accomplished by removing one box (in all possible ways), while induction adds a box (also in all possible ways).
Since the diagrams of all partitions appearing in $H_i(\Gamma_{1,s})$
have  boxes only in the first row and column, all resulting diagrams will have at most two boxes in the second column.

Using Theorem~\ref{thm:ranktwo} we can see that the simple modules appearing in $H_i(\Gamma_{2,s})$ for $i$ even
always have at least $i/2$ boxes in the second column.  Thus if $i\geq 6$  there is no partition which appears in both the domain and range of $\hat A_\phi$,  which forces $\hat A_\phi$ (and therefore $A_\phi$) to be zero.

For $i=2$ the target $H_2(\Gamma_{2,s})$ is always zero. For  $i=4$ the diagrams appearing in $H_4(\Gamma_{1,s})$ have  five rows, but the diagrams appearing in $H_4(\Gamma_{2,s})$ have at most three rows, so it is not possible to obtain one from the other by changing the position of a single box, and again the assembly map must be trivial.
\end{proof}

In Section~\ref{sec:geometric} below we show precisely how the class $\alpha_k$ becomes trivial after one stabilization,  using a natural geometric interpretation of this class.


\section{A more geometric viewpoint}
\label{sec:geometric}

A rational model for $B\Outn$ is the quotient $Q_n$ of Outer Space for rank $n$ graphs by the action of $\Outn$ changing the marking.  (One could instead use just the spine of Outer Space, but for our present purposes it is more convenient not to restrict to the spine.) Points of $Q_n$ are thus isometry classes of finite connected graphs of rank $n$ with no vertices of valence $1$ or $2$ and with lengths assigned to the edges, normalized so that the sum of the lengths of all the edges is $1$.  Collapsing edges to points by shrinking their lengths to zero is allowed provided this does not decrease the rank of the graph.  There is a similar rational model $Q_{n,s}$ for $B\Gamma_{n,s}$ consisting of graphs in $Q_n$ with $s$ leaves attached to them at arbitrary points.  There is no need to assign lengths to the leaf edges since they are not allowed to collapse to points.  An assembly map is induced from a map $Q_{n_1,s_1}\times\cdots\times Q_{n_k,s_k} \to Q_{n,s}$ where edge lengths on a glued-together graph $X_{n,s}$ are obtained by first assigning a fixed length, say $1$, to the new edges created by the leaf pairings, then renormalizing the lengths of all the nonleaf edges of $X_{n,s}$.

\subsection{Geometric Morita cycles}

The classes $\alpha_k\in H_{2k}(\Gamma_{1,2k+1})$ are particularly easy to describe from this perspective, and hence also the Morita classes and their generalizations.  The class $\alpha_k$ is the image of the top-dimensional homology class of a $2k$-dimensional torus under a map $f\colon T^{2k}\to Q_{1,2k+1}$ described as follows.  Consider graphs $X_{1,2k+1}$ obtained from a circle $c$ by attaching $2k+1$ leaves.  By rotating the circle if necessary, we can assume the first leaf attaches at a fixed basepoint of $c$.  The other leaves attach at $2k$ arbitrary points of $c$ which need not be distinct.  Letting these points vary independently around $c$ then gives the map $f\colon T^{2k}\to Q_{1,2k+1}$.  This is surjective but not injective since graphs differing by a reflection of $c$ fixing the basepoint are identified in $Q_{1,2k+1}$, so $Q_{1,2k+1}$ is $T^{2k}$ modulo the action of $\Z_2$ reflecting each circle factor. From this point of view one can see why we require the total number of leaves to be odd, because if it were $2k$ instead of $2k+1$ then the $\Z_2$-action would reverse the orientation of the torus $T^{2k-1}$ and hence the map $T^{2k-1}\to Q_{1,2k}$  would induce the trivial map on the top-dimensional homology of the torus.

For the map $Q_{1,2k+1} \times Q_{1,2k+1} \to Q_{2k+2,0}$ used to construct the Morita class $\mu_k =\alpha_k \otimes \alpha_k$ we glue all the leaves of the first copy of $X_{1,2k+1}$ to the leaves of the second copy.  Thus we have two circles joined by $2k+1$ edges.  One of these edges serves as a ``basepoint" edge, and then by varying where the remaining $2k$ edges attach we obtain a family of graphs $X_{2k+2,0}$ corresponding to a map
$T^{4k}\to Q_{2k+2,0}$ taking a generator of $H_{4k}(T^{4k})$ to $\mu_k$.  (The basepoint edge could be collapsed to a point, giving a map $T^{4k}\to Q_{2k+2,0}$ homotopic to the original one.)  The generalized Morita classes have similar geometric descriptions as maps from a torus to the appropriate $Q_{n,s}$.

For the Morita class $\mu_k$ the map $T^{4k}\to Q_{2k+2,0}$ is invariant under certain symmetries. To start, there is the $\Z_2\times\Z_2$ symmetry coming from the symmetries of the two $\alpha_k$ factors reflecting each of the two circles.  There is another $\Z_2$ symmetry from interchanging the two circles. Finally, there is an $\SG_{2k}$ symmetry group permuting the $2k$ arcs connecting the two circles. Altogether this gives a symmetry group $G_k$ of order $8(2k)!$ with the map $T^{4k}\to Q_{2k+2,0}$ factoring through the quotient $T^{4k}/G_k$.  One can regard $G_k$ as acting on choices of an ordering and orientations of the two circles and an ordering of the $2k$ connecting arcs.  This makes it clear that the induced map $T^{4k}/G_k\to Q_{2k+2,0}$ is injective.

In the case $k=1$ the quotient $T^4/G_1$ can be determined explicitly.  

\begin{proposition}
$T^4/G_1 = S^4$.
\end{proposition}

\begin{proof} The quotient of $T^4$ by the reflections of the two circles gives $S^2\times S^2$ since the quotient of $T^2$ by reflection of its two circle factors is the familiar $2$-sheeted branched covering space $T^2\to S^2$.  Next, factor out the $\Z_2$-action interchanging the two circles, corresponding to interchanging the two factors of $S^2\times S^2$. This gives the two-fold symmetric product $SP_2(S^2)$ which is well-known to be $\C P^2$.  Explicitly, $\C P^2$ can be identified with nonzero polynomials $a_2 z^2 + a_1 z + a_0$ in $\C[z]$ up to scalar multiplication, and these are determined by their unordered pair of roots in $S^2=\C \cup \infty$ where linear factors corresponding to roots at $\infty$ are deleted.  (See e.g.\ \cite{Hatcher}, Example 4K.4.) Finally we need to factor out by the $\Z_2$-action interchanging the two connecting arcs.  This corresponds to reflecting each torus $T^2$ across its diagonal.  In the quotient $S^2\times S^2$ of $T^2 \times T^2$ this is equivalent to reflecting each $S^2$ across its equator.  In the space of quadratic polynomials this is given by complex conjugation of the roots, hence also of the coefficients.  Thus we are forming the quotient of $\C P^2$ by complex conjugation. This quotient is $S^4$ by a classical result of Massey \cite{Massey} and Kuiper \cite{Kuiper}.
\end{proof}

The quotient $T^{4k}/G_k$ for $k>1$ cannot be a sphere since one can compute that its rational homology consists of a copy of $\Q$ in each dimension $4i \leq 4k$.

\medskip

We can use the geometric viewpoint to give another proof that Morita classes and their generalizations vanish after one stabilization of the rank:

\begin{proof}[Geometric proof that $\alpha_k$ vanishes after one stabilization.]
The idea is to see how the commutator relation $[e_{ij},e_{jk}]=e_{ik}$ among elementary matrices can be translated into a two-parameter family of graphs.

We can reinterpret the stabilization map as the map $Q_{1,2k+1} \to Q_{2,2k+1}$ obtained by attaching both ends of a new edge $b$ at the basepoint of the graphs $X_{1,2k+1}$ described above consisting of a circle $c$ with $2k+1$ leaves attached, where by the basepoint we mean the point of $c$ where the fixed leaf attaches. Let $a$ be any one of the remaining $2k$ leaves.  Sliding $a$ around $c$ gives one of the $S^1$ factors of the torus $T^{2k}$ whose map to $Q_{1,2k+1}$ sends a generator of $H_{2k}(T^{2k})$ to $\alpha_k$.

Figure~\ref{fig:2parafam} describes a two-parameter family of graphs in which one end of the arc $b$ moves across $c$ while one end of the arc $a$ moves across $b$ and $c$. As we proceed from left to right in the sequence of four pictures we see one end of $b$ sliding around $c$.  The dotted arc denotes the path followed by the attaching point of $a$.  Initially it just goes across $b$, then when $b$ has moved partway around $c$ the end of $a$ must backtrack across part of $c$ after it crosses $b$, in order to return to the basepoint.  In the last picture we see that $a$ crosses both $b$ and $c$.

\begin{figure}
\begin{center}

 \begin{tikzpicture}[scale=.5]
\draw  (0,0) circle(1);
\node [left] (c) at (-1,0) {$c$};
\draw (45:1)..controls (2.3,3.1) and   (-.9,3.1) ..(45:1);
\draw [thick, densely dotted] (45:1)..controls (2.5,3.2) and (-1.1,3.2)..(45:1);
 \node [above] (b) at (.7,2.5) {$b$};
 \draw (45:1) to [out=20, in=180] (2.5, 1);
\node [above] (a) at (2.5, 1) {$a$};
 \end{tikzpicture}
   \begin{tikzpicture}[scale=.5]
\draw  (0,0) circle(1);
\node [left] (c) at (-1,0) {$c$};
 \draw (45:1) to [out=20, in=180] (2.5, 1);
\node [above] (a) at (2.5, 1) {$a$};
\draw (45:1)..controls (70:3.2) and   (140:3.2) ..(165:1);
\draw [thick, densely dotted] (45:1)..controls (70:3.4) and (142:3.3)..(167:1.1);
\draw [thick, densely dotted] (47:1.1) arc(47:165:1.1);
 \node [above] (b) at (105:2.1) {$b$};
 \end{tikzpicture}
    \begin{tikzpicture}[scale=.5]
\draw  (0,0) circle(1);
\node [left] (c) at (-1,0) {$c$};
 \draw (45:1) to [out=20, in=180] (2.5, 1);
\node [above] (a) at (2.5, 1) {$a$};
\draw (45:1)..controls  (70:3.2) and   (140:3.2) ..(285:1);
\draw [thick, densely dotted] (45:1)..controls (70:3.5) and (145:3.4)..(285:1.1);
\draw [thick, densely dotted] (47:1.1) arc(47:285:1.1);
 \node [above] (b) at (120:1.8) {$b$};
 \end{tikzpicture}
   \begin{tikzpicture}[scale=.5]
\draw  (0,0) circle(1);
\node [left] (c) at (-1.05,0) {$c$};
\draw (45:1)..controls  (2.3,3.1) and   (-.9,3.1) ..(45:1);
\draw [thick, densely dotted] (45:1)..controls (2.5,3.2) and (-1.1,3.2)..(45:1);
 \node [above] (b) at (.7,2.5) {$b$};
 \draw (45:1) to [out=20, in=180] (2.5, 1);
\node [above] (a) at (2.5, 1) {$a$};
\draw [thick, densely dotted] (-.05,-.05) circle(1.07);
 \end{tikzpicture}
 \end{center}
\caption{A 2-parameter family of graphs $X_{2,1}$ }
\label{fig:2parafam}
\end{figure}
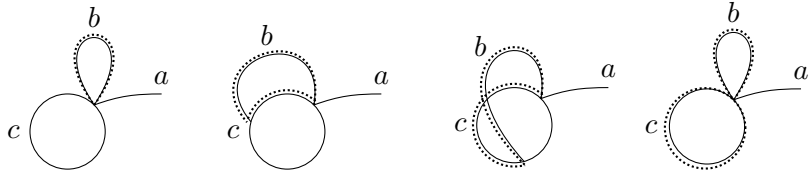

On the boundary of the parameter square for this two-parameter family one thus has five slides of one arc over another, as indicated in Figure~\ref{fig:puncturedtorus}. The quotient space of the square obtained by identifying the two $a/b$ edges and the two $b/c$ edges is a surface $S_{1,1}$ of genus one with one boundary circle, where this boundary circle parametrizes the $a/c$ slide.  The $a/c$ slide was the restriction of the map $f\colon T^{2k}\to Q_{2,2k+1}$ representing the stabilization of $\alpha_k$ to one of the circle factors of $T^{2k}$, so we can extend $f$ to a map $T^{2k-1}\times S_{1,1}\to Q_{2,2k+1}$.  This implies that $f$ induces the zero map $H_{2k}(T^{2k})\to H_{2k}(Q_{2,2k+1})$.
\end{proof}

\begin{figure}
\begin{center}
\begin{tikzpicture}[scale=1.25]
 \draw [thick] (0,0) -- (2,0) -- (2,2) -- (0,2) -- (0,0);
 \draw [thick] (1,2) -- (2,0);
\node[right] (rbc) at (2,1) {$b/c$};
\node[left] (lbc) at (0,1) {$b/c$};
\node[below] (bab) at (1,0) {$a/b$};
\node[above] (aab) at (.5,2) {$a/b$};
\node[above] (aac) at (1.5,2) {$a/c$};
\end{tikzpicture}
\end{center}
\caption{The parameter space, a punctured torus }
\label{fig:puncturedtorus}
\end{figure}
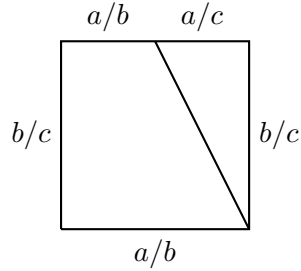

\subsection{Representing Eisenstein classes geometrically.}
\label{subsec:geoEisenstein}

Let us describe how the Eisenstein classes ${\mathcal E}_k\in H_{4k+3}\big(\Aut(F_{2k+3})\big)$ can be realized as families of graphs parametrized by certain manifolds $E_k$ that are analogous to the tori representing Morita classes, but a little more complicated.  Consider first the case $k=1$, so $E_1$ will be a closed orientable $7$-manifold.  This is the product of $T^4$ with a $3$-manifold $N$, where $N$ splits along a $2$-torus into submanifolds $N_1$ and $N_2$ each homeomorphic to a product of $S^1$ with a compact surface $S_{1,1}$ of genus $1$ with one boundary component. We obtain $N$ from $N_1\coprod N_2$ by gluing the two boundary tori via a homeomorphism of $\bdry N_1=\bdry N_2 = T^2$ switching the two circle factors of $T^2$. From this description one can see that $\pi_1 E_1$ is the product of $\Z^4$ with two copies of $\Z\times F_2$ amalgamated along $\Z\times \Z$ where the second $\Z$ is  
 generated by the commutator of the generators of $F_2$ and the amalgamation interchanges the two factors of $\Z\times\Z$.  Also from the construction of $E_1$  one can easily see that it is a $K(\pi,1)$.

Now we describe a map $E_1\to Q_{5,1}$ corresponding to a $7$-dimensional family of rank $5$ graphs with one leaf. To construct these graphs, start with the family of graphs parametrized by $S_{1,1}$ indicated in Figure~\ref{fig:2parafam}, consisting of a circle $c$ with edges $a$ and $b$ attached.  One end of $b$ is attached to the basepoint of $c$ and the other end to a point moving around $c$. The arc $a$ attaches at one end to a point that moves across $b$ and then returns to the basepoint along an arc of $c$.  Next we attach one end of another arc $a'$ at a point that moves only around $c$, independently of how $a$ and $b$ attach.  This gives a family of graphs parametrized by $S^1 \times S_{1,1}$.  Reversing the roles of $a$ and $a'$ gives another family parametrized by $S^1\times S_{1,1}$.  On $S^1\times\bdry S_{1,1}$ both families consist of graphs in which $b\cup c$ is $S^1\vee S^1$ with $a$ and $a'$ attached to arbitrary points of $c$.  The two families parametrized by $S^1\times S_{1,1}$ then fit together to form a family parametrized by $N$, so we have a map $N\to Q_{2,2}$.   Attaching two more arcs $d$ and $e$ at an endpoint of each that moves freely around $c$ gives a family parametrized by $N\times T^2$ and so a map $N\times T^2 \to Q_{2,4}$, hence a class $m\in H_5(\Gamma_{2,4})$. Finally, we assemble this class $m$ with the torus $T^2$ representing $\alpha_1\in H_2(\Gamma_{1,3})$ by adjoining another circle $c'$ and attaching the free ends of $a$, $d$, and $e$ at points that move around $c'$, where by rotating $c'$ we can assume that $a$ attaches just at the basepoint of $c'$.  The arc $a'$ has one end unattached, so it is a leaf. Figure~\ref{Allen} shows a graph in the resulting $7$-parameter family in the case that $a'$ attaches to $c$, but $a'$ could also attach to $b$ when $a$ attaches to $c$.  These two possibilities correspond to the two submanifolds $N_1$ and $N_2$ of $N$.

   \begin{figure}
\begin{center}
\begin{tikzpicture}[scale=.8]
\draw  (0,0) circle(1);
 \draw (.78,-.605) arc (-120:120:0.7);
\draw (0:-1) to (0:-1.75);
\draw (0:1.84) to (0:3);
\draw (60:1)  to  [out=20, in=160] (3.27,.7);
\draw (-60:1)  to  [out=-20, in=-160] (3.27,-.7);
\node (e) at (2,1.3) {$e$};
\node (a) at (2.4,0.2) {$a$};
\node (b) at (1.3,-0.35) {$b$};
\node (d) at (2,-1.3) {$d$};
\node (ap) at (-1.4,0.3) {$a'$};
\node (c) at (-.5,-1.2) {$c$};
\begin{scope}  [xshift=4cm]
\draw (0,0) circle (1);
\node (cp) at (.5,-1.1) {$c'$};
\end{scope}
\end{tikzpicture}
\end{center}
\caption{A graph in the $\mathcal{E}_1$ family}\label{Allen}
\label{gluing}
\end{figure}
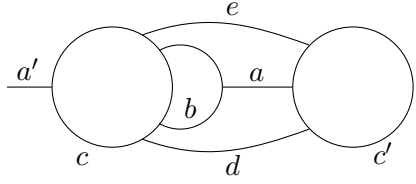
To relate this to the earlier construction of Eisenstein classes, note first that the group $M_{I,J}$ in Proposition~\ref{prop:oddtwo} is the semidirect product $\Z\ltimes(F_2^2\times \Z^2)$ in the case at hand when $|I|=2=|J|$.  A $K(\pi,1)$ for this group is the mapping torus $T_f$ of a map $f$ from $(S^1\vee S^1)^2 \times T^2$ to itself that is the identity on the $T^2$ factor, and on each $S^1\vee S^1$ induces the automorphism of $\pi_1(S^1\vee S^1)=F_2$ fixing the first basis element $x$ and sending the second basis element $y$ to $xy$.  We can compute $H_5(T_f)$ geometrically as the elements of $H_4\big((S^1\vee S^1)^2\times T^2\big)$ fixed by $f_*$ using the exact sequence in Example~2.48 of~\cite{Hatcher}.  This $H_4$ has dimension $4$ with basis corresponding to the $4$-cells of $(S^1\vee S^1)^2\times T^2$ by the K\"unneth formula. The basis elements can be written (omitting tensor product symbols for simplicity) as $xxzz$, $xyzz$, $yxzz$, and $yyzz$. (It would make sense to simplify the notation further by replacing  $z$ by $x$ throughout.) The map $f_*$ fixes $x$ and $z$ and takes $y$ to $x+y$.  A short calculation shows that the elements invariant under $f_*$ form a $2$-dimensional subspace with basis $xxzz$ and $xyzz-yxzz$.  The element $xxzz$ is not of interest to us since it can be shown to give the trivial element of $H_5(\Gamma_{2,4})$.  For the element $xyzz-yxzz$ the four letters of each of these two words correspond to the edges $a,a',d,e$ in that order.  For example the four letters of $xyzz$ correspond to $a$ moving around $c$, $a'$ moving across $b$, and $d$ and $e$ moving around $c$.  Interchanging $a$ and $a'$ gives the other word $yxzz$. To go from $H_4\big((S^1\vee S^1)^2\times T^2\big)$ to $H_5(T_f)$ involves a fifth parameter, and this corresponds to one end of the edge $b$ moving around $c$.  The whole mapping torus $T_f$ corresponds to a $5$-dimensional family of graphs in which $a$ and $a'$ move freely around both $b$ and $c$, while $d$, $e$, and one end of $b$ move around $c$.  The reason for restricting to the graphs parametrized by $N\times T^2$ is to get a manifold as parameter space and thus reduce the dimension of $H_5$ from two to one.

To generalize from $\mathcal{E}_1$ to $\mathcal{E}_k$ is easy since all one has to do is replace the arcs $d$ and $e$ by $2k$ arcs $d_1,\ldots,d_k$ and $e_1,\ldots,e_k$ that behave in exactly the same way as $d$ and $e$.  Thus $E_k=N \times T^{4k}$.  Incidentally, the manifold $N$ belongs to the class of 3-manifolds known as graph manifolds, which seems an especially appropriate name in the present context.

One can see that $\mathcal{E}_k\in H_{4k+3}(\Gamma_{2k+3,1})$ maps to zero in $H_{4k+3}(\Gamma_{2k+3,0})$ since ignoring the leaf $a'$ replaces the $3$-manifold $N$ by a $2$-dimensional quotient, namely $S_{1,1}$, so the composition $E_k\to Q_{2k+3,1} \to Q_{2k+3,0}$ factors through $S_{1,1}\times T^{4k}$ which has one lower dimension than $E_k$.

\section{Homological triviality of a standard ``maximal torus'' for $\Autn$}
\label{subsec:maximal}

Let $A$ be the subgroup of $\Autn$ generated by the automorphisms $\lambda_i$ and $\rho_i$ for $1\leq i\leq n-1$, where $\lambda_i$ sends the basis element $x_i$ to $x_nx_i$ and fixes $x_j$ for $j\neq i$, and $\rho_i$ is defined similarly but multiplies $x_i$ by $x_n$ on the right. Thus $A$ is isomorphic to $\Z^{2n-2}$, realizing the maximal rank of an abelian subgroup since the $\vcd$ of $\Autn$ is $2n-2$.

\begin{theorem}
\label{thm:maxabelian}
The inclusion $A\incl\Autn$ of the standard free abelian subgroup of maximal rank induces the trivial map on rational homology in all positive dimensions.
\end{theorem}

Note that the $\Z^{2n-4}\subset \Aut(F_n)$ realizing the Morita class $\mu_k$ is not contained in this $\Z^{2n-2}\subset \Aut(F_n)$ or any subgroup conjugate to this by permuting basis elements for $F_n$, and the theorem gives a good reason why this must be the case.

We will give three different proofs of this theorem, each with its own advantages. The first proof is probably the most elementary.

\begin{proof}[Algebraic proof]
We can enlarge $A$ to a subgroup $G\subset \Autn$ by adjoining the automorphisms that permute the basis elements $x_1,\ldots,x_{n-1}$ and send a subset of them to their inverses.  These automorphisms form a copy of the signed permutation group $\SG^\pm_{n-1}$ in $\Autn$, and $G$ is the semidirect product $\SG^\pm_{n-1} \ltimes A$.  It will suffice to show $H_i(G)=0$ for $i>1$ since $H_1\big(\Autn \big)=0$ from the classical presentations of $\Autn$.

Passing from homology to cohomology and applying the usual argument with transfer homomorphisms, we can compute $H^*(G)$ as the invariants of $H^*(A)$ under the action of $\SG^\pm_{n-1}$ induced by conjugation.  The cohomology ring $H^*(A)$ is an exterior algebra on generators $a_i$ and $b_i$ corresponding to $\lambda_i$ and $\rho_i$.  Conjugation by the map inverting $x_i$ sends $\lambda_i$ to $\rho_i^{-1}$ and $\rho_i$ to $\lambda_i^{-1}$, so in $H^*(A)$ this sends $a_i$ to $-b_i$ and $b_i$ to $-a_i$.  Conjugation by a permutation of the $x_i$'s has the effect of permuting the subscripts on the $a_i$'s and $b_i$'s.

Elements of $H^k(A)$
are linear combinations of degree $k$ monomials in the $a_i$'s and $b_i$'s.
We will show that for any monomial $m$ of degree $k > 1$ there exists $\sigma \in \SG^\pm_{n-1}$ such that $\sigma m =-m$.  This implies that $m$ cannot appear in any element of $H^k(A)$ that is invariant under the action of $\SG^\pm_{n-1}$, and hence $H^k(G)=0$ for $k>1$.  There are three cases:
if $m$ contains both $a_i$ and $b_i$ for some $i$, then
inverting $x_i$ changes the sign of $a_i \wedge b_i$ and thus the sign of $m$; if $m$ contains $a_i$ and $a_j$ but not either of $b_i$ or $b_j$ we use the involution in $G$ interchanging $x_i$ and $x_j$; and if $m$ contains $a_i$ and $b_j$ but not either of $b_i$ or $a_j$ we use the involution in $G$ interchanging $x_i$ and $x_j^{-1}$.
(This argument does not apply when $k=1$, but it is easy to check that $H^1(G)=\F$ generated by $\sum_i a_i - \sum_i b_i$.)
\end{proof}

\begin{proof}[Geometric proof]
As in the previous section we consider the rational model $Q_{n,1}$ for $\Autn$.
The inclusion $\Z^{2n-2}\incl\Autn$ corresponds to a map $f\colon T^{2n-2}\to Q_{n,1}$ of the $(2n-2)$-torus to $Q_{n,1}$.  This specifies a family of graphs parametrized by $T^{2n-2}$ constructed as follows.  Start with a basepointed circle $c$, then attach $n-1$ arcs $a_i$ by identifying their endpoints with points $s_i$ and $t_i$ in $c$.  The $s_i$ and $t_i$ are the coordinates on $T^{2n-2}$, and we can write $f$ as a function $f(s_1,t_1,\ldots,s_{n-1},t_{n-1})$.

The map $f$ is not injective since there are some symmetries present.  One can interchange $s_i$ and $t_i$, switching the ends of $a_i$, without changing the graph, and one can permute the arcs $a_i$.  Switching $s_i$ and $t_i$ gives a quotient of the $i$\hskip1pt th $2$-torus factor of $T^{2n-2}=(T^2)^{n-1}$.  The quotient of a $2$-torus by interchanging the two circle factors is a triangle with two edges identified.  This deformation retracts to a single circle, say the $s_i$ circle.  The quotient of $T^{2n-2}$ by these coordinate transpositions thus has the homotopy type of $T^{n-1}$.  This already implies that the inclusion $A\hookrightarrow \Autn$ induces the trivial map on $H_k$ for $k> n-1$.

Now we can factor out the permutations of the $n-1$ factors of this $T^{n-1}$, producing the $(n-1)$-fold symmetric product of $S^1$.  This is well-known to have the homotopy type of a single circle.  (See for example the end of Example~4K.4 in~\cite{Hatcher}.) Thus the map $f$ factors through a space homotopy equivalent to $S^1$ so it induces the trivial map on $H_i$ for $i>1$.  It also induces the trivial map on $H_1$ since $H_1\big(\Autn\big)=0$ as noted in the first proof.
\end{proof}

\begin{proof}[Proof via representation theory]
We have inclusions $A \subset \Gamma_{1,2n-1} \subset \Aut(F_n)$ where the second inclusion corresponds to the self-gluing  $\Gr_{1,2n-1}\to\Gr_{n,1}$.  (In fact $\Gamma_{1,2n-1}$ is contained in the subgroup $G$ used in the first proof above).
The gluing $\Gr_{1,2n-1}\to\Gr_{n,1}$ factors as the composition of two gluings $\Gr_{1,2n-1}\to\Gr_{2,2n-3}\to\Gr_{n,1}$ as shown in Figure~\ref{fig:X_1s_to Xn1}, so the map $H_*(A) \to H_*\big(\Aut(F_n)\big)$ factors as
$H_*(A)\to H_*(\Gamma_{1,2n-1}) \to H_*(\Gamma_{2,2n-3})\to H_*(\Gamma_{n,1})$.
The middle of these three maps is an assembly map which we showed is trivial in Section~\ref{subsec:selfgluing} (unless the degree is $0$), which implies that $H_k(A) \to H_k\big(\Aut(F_n)\big)$ is trivial if $k > 0$.
\end{proof}

One of the advantages of this last proof is that it also works in degree $1$, so one does not need a separate argument for this case.

\begin{remark}
The composition $A\to \Autn \to \Outn$ has kernel $\Z$  and is injective when restricted to a suitable subgroup $\Z^{2n-3}$ realizing the $\vcd$ of $\Outn$.  The theorem implies that this inclusion $\Z^{2n-3}\hookrightarrow \Outn$ is also trivial on homology since it factors through $A\hookrightarrow\Autn$.
\end{remark}

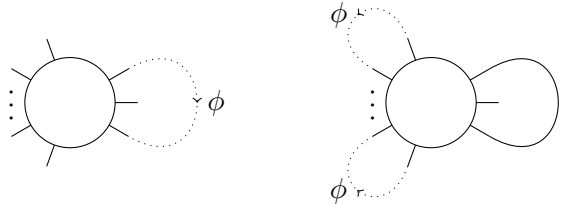
\begin{figure}
\begin{center}
\begin{tikzpicture}[scale=.6]
\begin{tikzrankone}{4cm}{0cm}
\tikzhair{m1}{110}
\tikzhair{m2}{150}
\tikzhair{m3}{-150}
\tikzhair{m4}{-110}
\tikzhair{m5}{30}
\tikzhair{m6}{0}
\tikzhair{m7}{-30}
\end{tikzrankone}
\path (m2) to node[sloped]{$\dots$}  (m3);
\draw [dotted,->-=0.5] (m5) to[out=30, in=-30, looseness=4] node[right]{$\phi$}  (m7);
\begin{tikzrankone}{12cm}{0cm}
\tikzhair{r1}{110}
\tikzhair{r2}{150}
\tikzhair{r3}{-150}
\tikzhair{r4}{-110}
\tikzhair{r5}{30}
\tikzhair{r6}{0}
\tikzhair{r7}{-30}
\end{tikzrankone}
\path (r2) to node[sloped]{$\dots$}  (r3);
\draw 
(r5) to[out=30, in=-30, looseness=4] node[right]{}  (r7);
\draw [dotted,->-=0.5] (r1) to[out=110, in=150, looseness=4] node[left]{$\phi$}  (r2);
\draw [dotted,->-=0.5] (r4) to[out=-110, in=-150, looseness=4] node[left]{$\phi$}  (r3);
\end{tikzpicture}
\end{center}
\caption{The gluings $\Gr_{1,2n-1}\to\Gr_{2,2n-3}\to\Gr_{n,1}$ }
\label{fig:X_1s_to Xn1}
\end{figure}

\section{Connections with hairy graph homology and the Lie algebra of symplectic derivations}
\label{sec:HairyLie}

\subsection{Hairy graph homology}
In this section we note the connection between our calculations and the hairy graph homology theory of~\cite{CKV1}.  As above $\SG_k$  denotes the symmetric group on $k$ letters and $\Sym^k$  the $k$-th symmetric power functor on vector spaces. The following lemma is an immediate consequence of Proposition~\ref{prop:vcd-nonvanish}.
\begin{lemma}
\label{lem:VF}
Let $\h=H^1(F_n)\iso \F^n$. For any $\F$-vector space $V$
$$
H^{2n-3+s}(\Gamma_{n,s})\otimes_{\SG_s} \wV{V}{s}\iso H^{2n-3}\big(\Out(F_n);\Sym^s(\h\otimes V)\big).
$$
\end{lemma}
\begin{proof}
The K\"unneth isomorphism $H^s(F_n^{\,s})\iso \wV{H^1(F_n)}{s}=\wV{\h}{s}$ is $\Out(F_n)\times \SG_s$-equivariant. So
\begin{align*}
H^{2n-3+s}(\Gamma_{n,s})\otimes_{\SG_s} \wV{V}{s}&
    \iso H^{2n-3}\left(\Out(F_n);H^s(F_n^{\,s})\right)\otimes_{\SG_s} \wV{V}{s}\,\,\text{(Prop.~\ref{prop:vcd-nonvanish})}\\
&\iso H^{2n-3}\left(\Out(F_n);\wV{\h}{s}\right)\otimes_{\SG_s} \wV{V}{s}\\
&\iso H^{2n-3}\left(\Out(F_n);\wV{\h}{s} \otimes_{\SG_s} \wV{V}{s} \right) \\
&\iso H^{2n-3}\left(\Out(F_n);\tV{\h}{s} \otimes_{\SG_s} \tV{V}{s} \right) \\
&\iso H^{2n-3}\left(\Out(F_n);\Sym^s(\h\otimes V)\right).
\end{align*}
\end{proof}

Let $\hairy_V^{n,s}$ denote the hairy Lie graph complex (see~\cite{CKV1, CKV2}), where graphs have rank $n$ and $s$ hairs labeled by vectors from $V$.  In~\cite{CKV1}
the following theorem, with the twist accidentally omitted, was proved by a direct analysis of the chain complex.

\begin{theorem}[\cite{CKV1}, Theorem 11.1]
\label{thm:oldthm}
There is an isomorphism
$$
H_k(\hairy_V^{n,s}) \iso  H^{2n+s-2-k}(\Gamma_{n,s})\otimes_{\SG_s} \wV{V}{s}.
$$
\end{theorem}

Combining this with Lemma~\ref{lem:VF} gives a shorter proof of the following theorem from~\cite{CKV2} relating   the first homology of the hairy Lie graph complex with the cohomology of $\Out(F_n)$ with twisted coefficients.

\begin{theorem}[\cite{CKV1}, Theorem 8.8]
For $n\geq 2, s\geq 0$ there is an isomorphism
$$
H_{1}(\hairy_V^{n,s}) \iso H^{2n-3}\big(\Out(F_n); \Sym^s(\h\otimes V)\big).
$$
\end{theorem}

This theorem has been slightly restated here to be more compatible with current notation.

\subsection{The Lie algebra of symplectic derivations}
According to Theorem~\ref{thm:oldthm} hairy graph homology is obtained by twisting the homology of $\Gamma_{n,s}$ with $\wV{V}{s}$.   In~\cite{CKV1, CKV2}, it was shown that if the dimension of $V$ is sufficiently large, the $k$-th homology of the Lie algebra $\lplus_V$ of positive degree symplectic derivations  embeds in hairy Lie graph homology:
$$
H_k(\lplus_V)\subset H_k(\hairy_V)=\bigoplus_{n,s} H_{k}( \hairy^{n,s}_V)
\iso  \bigoplus_{n,s} H^{2n+s-2-k}(\Gamma_{n,s})\otimes_{\SG_s} \wV{V}{s}.
$$
Furthermore,   every irreducible  $\GL(V)$-module $\SF{\lambda}V$ in the decomposition of    hairy graph homology corresponds to an irreducible $\SP(V)$-module $\SF{\langle\lambda\rangle}V$ in the decomposition of  of $H_*(\lplus_V)$.  Thus the cohomology classes found in this paper produce homology classes for $\lplus_V$.  These classes can be used to show that $H_{k}(\lplus)$ contains infinitely many different $\SP$-modules, as we now show.

\begin{theorem} $H_{3n+d-2}(\lplus)$ contains the $\SP$-module $\SF{\langle((2m{+}1)^n,1^d)\rangle}$ for all $m$.
\end{theorem}
\begin{proof}
Let  $s=2n(m{+}1)+d$.  By Theorem~\ref{thm:2mn},  $H^{2nm}(\Gamma_{n,s})$ contains the $\SG_{s}$-submodule $P_{(n+d,n^{2m})}$ with multiplicity $1$.  Setting  $k=3n+d-2$ we have
$$
\begin{aligned}
H_{3n+d-2}(\hairy_V^{n,s})&\iso H^{2mn}(\Gamma_{n,s})\otimes_{\SG_s} \wV{V}{s}\\
& \supset P_{(n+d,n^{2m})}\otimes_{\SG_s} \wV{V}{s}\\
\end{aligned}
$$

Since $\wV{V}{s}=\bigoplus_\lambda P_{\lambda^\prime}\otimes \SF{\lambda}(V)$ and $P_\lambda\otimes_{\SG_s} P_\mu= 0$ unless $\lambda=\mu$, in which case it is the trivial module $\F$, we see that for $\dim(V)\geq \max(2m{+}1,n{+}d)$,
 $
H_{3n{+}d{-}2}(\hairy_V^{n,s})
$
contains the $\GL(V)$-submodule $\SF{(n{+}d,n^{2m})^\prime}(V) = \SF{((2m{+}1)^n,1^d)}(V)$.
Therefore for $\dim(V)$ sufficiently large
$H_{3n+d-2}(\lplus_V)$ contains the corresponding $\SP(V)$-submodule $\SF{\langle ((2m{+}1)^n,1^d)\rangle}(V)$.
Taking the limit as $\dim(V)$ goes to infinity, gives that  $H_{k}(\lplus)$ contains $\SF{\langle ((2m{+}1)^n,1^d)\rangle}$ for all $m$.
\end{proof}

\begin{remark}
\label{rem:HairyMorita}
In~\cite{CVMorita} the original Morita classes in the cohomology of $\Outn$ were re-interpreted in terms of hairy graphs and this point of view was then used to construct more classes, called generalized Morita classes.  The fact that the classes described in Sections~\ref{subsec:Morita} and~\ref{subsec:genMorita} represent the same classes stems from the identification
$$
H_{1}(\hairy_V^{1,2k+1})\iso H^{2k}(\Gamma_{1,2k+1})\otimes_{\SG_{2k+1}} \wV{V}{2k+1} = \SF{2k+1}V = \Sym^{2k+1} V
$$
from Theorem~\ref{thm:oldthm}.   A generator of $H_{1}(\hairy_V^{1,2k+1})$ is a linear Lie tree with two ends joined by an edge and $2k+1$ commuting $V$-labeled leaves (``hairs''), whereas a generator of $H^{2k}(\Gamma_{1,2k+1})$ is the same thing, but with unlabeled leaves. Call this generator $\alpha_k^*$, as it is dual to $\alpha_k$. The graphical cocycle of~\cite{CVMorita} is nonzero only if the graph is the union of $\alpha_{k_1}^*,\ldots,\alpha_{k_m}^*$ with the hairs connected up by edges. In that case the cocycle evaluates to $\pm$ the graph obtained by shrinking each $\alpha_{k}$ to a point. Further projecting the graphical cocycle into the subspace spanned by a single graph gives the Morita cocycle $\mu_G$. By construction, this process is dual to the gluing map defined by $G$.
\end{remark}

\begin{remark}
\label{rem:modular}
In~\cite{CKV2} it was explained that hairy graph homology  can be viewed as the Feynman transform of a cyclic operad.  In light of Theorem~\ref{thm:oldthm}  this implies that the
cohomology groups $H^*(\Gamma_{n,s})$ can be combined into a (twisted) modular co-operad~\cite{GK}. Therefore the duals $H_*(\Gamma_{n,s})$ of these groups  form a (twisted) modular operad. The assembly maps defined in Section~\ref{sec:gluing} are the structure maps of this modular operad.
\end{remark}

\section{Open Questions}
\label{sec:conj}

We finish with several questions and conjectures related to results in the paper.  The conjectures have been verified for all  but the most recently discovered nontrivial classes in $H_k(\Gamma_{n,s})$, but there are not enough of these verified cases to provide overwhelming evidence for the conjectures.

In what follows we always exclude trivial assembly maps, those that involve a component graph  $\Gr_{0,2}$ where the gluing involves only one leaf, because the map
$\beta: H_k(\Gamma_{n,s}) \to H_k(\Gamma_{n,s})$ induced by the assembly map
$\alpha: H_k(\Gamma_{n,s}) \otimes H_0(\Gamma_{0,2}) \to H_k(\Gamma_{n,s})$
is equal to the identity.

\subsection{Surjectivity of the assembly maps below the virtual cohomological dimension}
An assembly map
$H_*(\Gamma_{n_1,s_1})\otimes \cdots  \otimes H_*(\Gamma_{n_k,s_k}) \to H_*(\Gamma_{n,s})$ for $k\geq 2$ can never have nontrivial image in the virtual cohomological dimension (\vcd) of $\Gamma_{n,s}$ since in that case the \vcd\ of $\Gamma_{n,s}$ is strictly greater than the sum of the \vcd's of the factors.  An assembly map induced by gluing pairs of leaves of a single graph preserves the \vcd, so such an assembly map might conceivably be nontrivial in that dimension,
though it seems unlikely that such a map can be surjective (with nontrivial image).
By contrast, in other dimensions  we expect that all classes are constructed by assembly  from lower-rank graphs.

\begin{conjecture}
\label{conj:B}
Suppose that  $k \not= 2n+s-3$, the \vcd\ of $\Gamma_{n,s}$.
Then  $H_k(\Gamma_{n,s})$ is generated by the images of the assembly maps $A_\phi$ over all possible gluings which give the graph $\Gr_{n,s}$.
\end{conjecture}

The examples in Section~\ref{sec:gluing} confirm this conjecture for $n=1$ and $n=2$.
It can be seen from the description of the cohomology groups that $H_k(\Gamma_{n,s+1})$
is generated by the images of  the assembly maps $H_k(\Gamma_{n,s})  \otimes H_0(\Gamma_{0,3})  \to H_k(\Gamma_{n,s+1})$
in all cases except when $k=2n+s-3$, or when $n=2$ and $k=s$ is divisible by $4$.
The first exception is excluded from the conjecture since these classes are in the \vcd\ and the second one is covered by Section~\ref{subsec:GammaTwo2N}.

The  first three Morita classes and the first two Eisenstein classes   are obtained as images of
 assembly maps. By Bartholdi's calculations there are nontrivial classes in $H_{11}(\Out(F_7))$ and $H_8(\Out(F_7))$.  
The first of these is in the $\vcd$, and there is a natural candidate for the second one, obtained by assembling four copies of $\alpha_1\in H_2(\Gamma_{1,3})$ in a tetrahedral pattern.  By Theorem~\ref{thm:injection} the class in $H_{11}(\Out(F_7))$ also produces a class in $H_{11}(\Aut(F_7)),$  which cannot be the Eisenstein class by Remark~\ref{rmk:eisenstein}. It is possible that this class can be obtained by assembling  classes in $H_{11}(\Gamma_{6,2})$ and $H_0(\Gamma_{0,3})$.

As mentioned in Remark~\ref{rem:modular}, the homology groups $H_*(\Gamma_{n,s})$ for $n \geq 1$ form  a twisted modular operad. Conjecture~\ref{conj:B} together with computations in Section~\ref{sec:gluing} imply that this modular operad is generated by $H_0(\Gamma_{0,3})$ and $H_{2n+s-3}(\Gamma_{n,s})$.

\subsection{Injectivity of classes constructed from modular forms}
In Section~\ref{subsec:classes} we constructed maps from spaces of modular forms to $H_{4m+2}\big(\Out(F_{2m+3})\big)$  by gluing two rank $2$ graphs together along all of their leaves. These maps take the form
$$
\chi:\bigoplus_{0\leq i <m} \left({\textstyle\ext}^2 \mathcal{X}_{2m,i}\right)\longrightarrow  H_{4m+2}\big(\Out(F_{2m+3})\big)
$$
The first of these maps which could be nontrivial has target $H_{42}\big(\Out(F_{23})\big)$. We do not know how to show that the image is nontrivial and it is beyond the reach of computer calculations.  Nevertheless,  we make the following conjecture:

\begin{conjecture}
\label{conj:C}
The restriction of $\chi$  to the term  ${\textstyle\ext}^2 \mathcal{X}_{2m,0}$ is injective.
\end{conjecture}
This conjecture appeared as a question in~\cite{CKV1}.
It seems unlikely that $\chi$ is injective on ${\textstyle\ext}^2 \mathcal{X}_{2m,i}$ for $i$ close to $m$, but injectivity is still plausible for small values of $i$, so we ask the following question:
\begin{question}
\label{question:C}
For  which  $i$  is the restriction of $\chi$ to the term $\ext^2 \mathcal{X}_{2m,i}$ injective?
\end{question}

We remark that  Conjecture~\ref{conj:C} would contradict a conjecture made by Church, Farb, and Putman~\cite[Conjecture~12]{CFP} concerning a stability property of the groups $H_{2n-3-k}\big(\Out(F_n)\big)$ for fixed $k$ as $n \to \infty$.

\subsection{Odd-dimensional classes}
In Section~\ref{subsec:odd} we constructed  maps from the space $\mathcal{M}_{2k+2}$ of modular forms of weight $2k+2$ to $H_{4k+3}\big(\Out(F_{s+2})\big)$ for all $k$ and $s$.  For small $s$ this map is trivial because the homology dimension is above the \vcd\ and for large $s$ the map must be trivial by homology stability, but for $s=2k+2$ we conjecture that it is highly nontrivial:

\begin{conjecture}
\label{conj:CC}
The map $\mathcal{M}_{2k+2}\to H_{4k+3}\big(\Out(F_{2k+4})\big)$ constructed in Section~\ref{subsec:odd} is injective for all $k$.
\end{conjecture}

This too is incompatible with the conjecture of Church, Farb, and Putman.  Since the \vcd\ of $\Out(F_{2k+4})$ is $4k+5$,  these classes lie in codimension 2, and Conjecture~\ref{conj:CC} implies that the dimension of $ H_{\vcd-2}\big(\Out(F_{n})\big)$ grows at least linearly with $n$ when $n$ is even.

\subsection{Vanishing under stabilization}
For $s>0$ there are two natural ways to stabilize $\Gamma_{n,s}$, by increasing $n$ or by increasing $s$, as described in Section~\ref{sec:vanish}.  Here we consider the stabilization increasing $n$.

\begin{conjecture}
\label{conj:D}
If $\phi$ is any gluing, then all positive-dimensional classes in the image of the assembly map $A_\phi$ vanish after a single stabilization with respect to $n$.
\end{conjecture}

The condition of positive dimension is of course necessary to exclude the classes in $H_0(\Gamma_{n,s})$ which clearly do not vanish after stabilization.
In Section~\ref{sec:vanish} we showed that the conjecture is true for classes in the image of an assembly map where stabilization  is done using a rank $1$ factor.

We   point out that Conjectures~\ref{conj:B} and~\ref{conj:D} imply the following for classes in dimensions below the \vcd:
\begin{conjecture}
\label{conj:E}
All classes in $H_k(\Gamma_{n,s})$ for $0<k<2n+s-3$ vanish after a single stabilization with respect to $n$.
\end{conjecture}
Since the \vcd\ of $\Gamma_{n,s}$ increases with $n$, this conjecture implies that two stabilizations kill {\em all} homology classes.

\section{Tables of results}
\label{sec:tables}
In this section we write out tables of the cohomology of $\Gamma_{n,s}$ for small values of $s$.  The space above the diagonal in each table is left blank since it represents terms above  the virtual cohomological dimension, where the cohomology must vanish.

\begingroup
\fontsize{10pt}{12pt}\selectfont
\begin{table}[H]
$$
\begin{array}{c|cccccccc}
&H^0&H^1&H^2&H^3&H^4&H^5&H^6&H^7\\
\hline
\Gamma_{1,0} &\F&&&&&&&\\
\Gamma_{1,1} &P_{(1)}&&&&&&&\\
\Gamma_{1,2} &P_{(2)}&0&&&&&&\\
\Gamma_{1,3} &P_{(3)}&0&P_{(1^3)}&&&&&\\
\Gamma_{1,4}&P_{(4)}&0&P_{(2,1^2)}&0&&&&\\
\Gamma_{1,5}&P_{(5)}&0&P_{(3,1^2)}  &0&P_{(1^5)}&&&\\
\Gamma_{1,6}&P_{(6)}&0&P_{(4,1^2)}&0&P_{(2,1^4)}&0&&\\
\Gamma_{1,7}&P_{(7)}&0&P_{(5,1^2)}&0&P_{(3,1^4)}&0&P_{(1^7)}&\\
\Gamma_{1,8}&P_{(8)}&0&P_{(6,1^2)}&0&P_{(4,1^4)}&0&P_{(2,1^6)}&0
\end{array}
$$
\caption{$H^i(\Gamma_{1,s};\F)$ for $s\leq 8$}
\label{tbl:n-one}
\end{table}

\bigskip

\begin{table}[H]
$$
\begin{array}{c|ccrcrccc}
&H^0&H^1&H^2&H^3&H^4&H^5&H^6&H^7\\
\hline
\Gamma_{1,0} & 1 &&&&&&&\\
\Gamma_{1,1} & 1 &&&&&&&\\
\Gamma_{1,2} & 1 & 0 &&&&&&\\
\Gamma_{1,3} & 1 & 0 & 1  &&&&&\\
\Gamma_{1,4} & 1 & 0 & 3  & 0 &&&&\\
\Gamma_{1,5} & 1 & 0 & 6  & 0 & 1  &&&\\
\Gamma_{1,6} & 1 & 0 & 10 & 0 & 5  & 0 &&\\
\Gamma_{1,7} & 1 & 0 & 15 & 0 & 15 & 0 & 1 &\\
\Gamma_{1,8} & 1 & 0 & 21 & 0 & 35 & 0 & 7 & 0
\end{array}
$$
\caption{Dimensions of $H^i(\Gamma_{1,s};\F)$ for $s\leq 8$}
\label{tbl:n-onedim}
\end{table}

\vfill\eject

\begin{table}[H]
{\renewcommand{\arraystretch}{1.3}
\begin{tabular}{c|llllll}
    &\multicolumn{1}{c}{$H^0$} & \multicolumn{1}{c}{$H^1$} &\multicolumn{1}{c}{$H^2$}&
    \multicolumn{1}{c}{$H^3$}&\multicolumn{1}{c}{$H^4$} & \multicolumn{1}{c}{$H^5$}\\
\hline
$\Gamma_{2,0}$&$\F$&$0$&&&&\\%
$\Gamma_{2,1}$&$P_{(1)}$&$0$&$0$&&&\\%
$\Gamma_{2,2}$&$P_{(2)}$&$0$&$0$&$0$&&\\%
$\Gamma_{2,3}$&$P_{(3)}$&$0$&$0$&$0$&$0$&\\%
$\Gamma_{2,4}$&$P_{(4)}$&$0$&$0$&$0$&$P_{(2^2)}$&$P_{(2,1^2)}$\\%
$\Gamma_{2,5}$&$P_{(5)}$&$0$&$0$&$0$&$P_{(3,2)}\oplus P_{(2^2,1)}$&$P_{(3,1^2)} \oplus P_{(2^2,1)}\oplus P_{(2,1^3)}$\\%
$\Gamma_{2,6}$&$P_{(6)}$&$0$&$0$&$0$&$P_{(4,2)}\oplus P_{(3,2,1)} \oplus P_{(2^3)}$&$P_{(4,1^2)} \oplus P_{(3,2,1)}\oplus P_{(3,1^3)} \oplus P_{(2^2,1^2)}$\\%
$\Gamma_{2,7}$&$P_{(7)}$&$0$&$0$&$0$&$P_{(5,2)}\oplus P_{(4,2,1)} \oplus P_{(3,2^2)}$&$P_{(5,1^2)} \oplus P_{(4,2,1)}\oplus P_{(4,1^3)} \oplus P_{(3,2,1^2)}$\\%
$\Gamma_{2,8}$&$P_{(8)}$&$0$&$0$&$0$&$P_{(6,2)}\oplus P_{(5,2,1)} \oplus P_{(4,2^2)}$&$P_{(6,1^2)} \oplus P_{(5,2,1)}\oplus P_{(5,1^3)} \oplus P_{(4,2,1^2)}$\\%
$\Gamma_{2,9}$&$P_{(9)}$&$0$&$0$&$0$&$P_{(7,2)}\oplus P_{(6,2,1)} \oplus P_{(5,2^2)}$&$P_{(7,1^2)} \oplus P_{(6,2,1)}\oplus P_{(6,1^3)} \oplus P_{(5,2,1^2)}$\\%
$\Gamma_{2,10}$&$P_{(10)}$&$0$&$0$&$0$&$P_{(8,2)}\oplus P_{(7,2,1)} \oplus P_{(6,2^2)}$&$P_{(8,1^2)} \oplus P_{(7,2,1)}\oplus P_{(7,1^3)} \oplus P_{(6,2,1^2)}$\\%
\end{tabular}\hfill
}

\vspace{1em}
{\renewcommand{\arraystretch}{1.3}
\begin{tabular}{c|llllll}
    &\multicolumn{1}{c}{$H^6$} & \multicolumn{1}{c}{$H^7$} &\multicolumn{1}{c}{$H^8$}&
    \multicolumn{1}{c}{$H^9$}&\multicolumn{1}{c}{$H^{10} $} \\
\hline
$\Gamma_{2,5}$&$0$&&&\\
$\Gamma_{2,6}$&$0$&$P_{(2,1^4)}$&&&\\
$\Gamma_{2,7}$&$0$&
$\begin{array}[t]{@{}l}P_{(3,1^4)} \oplus P_{(2^2,1^3)}  \\ \; \oplus \, P_{(2,1^5)} \end{array}$
&$0$&&\\
$\Gamma_{2,8}$&&
$\begin{array}[t]{@{}l}P_{(4,1^4)} \oplus P_{(3,2,1^3)}  \\ \; \oplus \, P_{(3,1^5)} \oplus P_{(2^2,1^4)}\end{array}$
&$P_{(2^4)}$&$P_{(2,1^6)} \oplus P_{(2^3,1^2)}$&\\
$\Gamma_{2,9}$&$0$&
$\begin{array}[t]{@{}l}P_{(5,1^4)} \oplus P_{(4,2,1^3)}  \\ \;\; \oplus \,P_{(4,1^5)}\oplus P_{(3,2,1^4)}\end{array}$
&$P_{(3,2^3)} \oplus P_{(2^4,1)}$&
$\begin{array}[t]{@{}l} P_{(3,1^6)} \oplus P_{(2^2,1^5)} \oplus P_{(2,1^7)}  \\  \;\; \oplus \,P_{(3,2^2,1^2)} \oplus P_{(2^4,1)}\oplus P_{(2^3,1^3)}\end{array}$&$0$\\
$\Gamma_{2,10}$&$0$&
$\begin{array}[t]{@{}l}P_{(6,1^4)} \oplus P_{(5,2,1^3)}\\ \;\;  \oplus \, P_{(5,1^5)}\oplus P_{(4,2,1^4)}\end{array}$&
$\begin{array}[t]{@{}l} P_{(4,2^3)} \oplus\, P_{(3,2^3,1)} \\\:\:  \oplus \, P_{(2^5)} \end{array} $&
$\begin{array}[t]{@{}l} P_{(4,1^6)} \oplus P_{(3,2,1^5)} \oplus P_{(3,1^7)}  \\ \;\; \oplus \,P_{(2^2,1^6)} \oplus P_{(4,2^2,1^2)} \oplus P_{(2^4,1^2)} \\ \:\: \oplus \, P_{(3,2^2,1^3)} \oplus P_{(3,2^3,1)} \end{array}$&$0$\\
\end{tabular}\hfill
}
\medskip
\caption{$H^i(\Gamma_{2,s};\F)$ for $i,s\leq 10$}
\label{tbl:n-two}
\end{table}

\begin{table}[H]
$$
\begin{array}{c|ccccrrcrrrc}
    &\multicolumn{1}{r}{H^0} & \multicolumn{1}{r}{H^1} &\multicolumn{1}{r}{H^2}&
    \multicolumn{1}{r}{H^3}&\multicolumn{1}{r}{H^4} & \multicolumn{1}{r}{H^5} &
    \multicolumn{1}{r}{H^6}&\multicolumn{1}{r}{H^7} & \multicolumn{1}{r}{H^8} &
    \multicolumn{1}{r}{H^9}&\multicolumn{1}{r}{H^{10}} \\
\hline
\Gamma_{2,0} & 1 & 0 &&&&&&&&&\\
\Gamma_{2,1} & 1 & 0 & 0 &&&&&&&&\\%
\Gamma_{2,2} & 1 & 0 & 0 & 0 &&&&&&&\\%
\Gamma_{2,3} & 1 & 0 & 0 & 0 & 0   &&&&&&\\%
\Gamma_{2,4} & 1 & 0 & 0 & 0 & 2   & 3   &&&&&\\%
\Gamma_{2,5} & 1 & 0 & 0 & 0 & 10  & 15  & 0 &&&&\\%
\Gamma_{2,6} & 1 & 0 & 0 & 0 & 30  & 45  & 0 & 5  &&&\\%
\Gamma_{2,7} & 1 & 0 & 0 & 0 & 70  & 105 & 0 & 35   & 0   &&\\%
\Gamma_{2,8} & 1 & 0 & 0 & 0 & 140 & 210 & 0 & 140  & 14  & 35   &\\%
\Gamma_{2,9} & 1 & 0 & 0 & 0 & 252 & 378 & 0 & 420  & 126 & 315  & 0\\%
\Gamma_{2,10}& 1 & 0 & 0 & 0 & 420 & 630 & 0 & 1050 & 630 & 1575 & 0\\
\end{array}\hfill
$$
\caption{Dimensions of $H^i(\Gamma_{2,s};\F)$ for $i,s\leq 10$}
\label{tbl:n-twodim}
\end{table}
\endgroup

\end{document}